\documentclass{amsart}

\makeatletter
\newcommand{\mainsectionstyle}{%%%%% section font change
  \renewcommand{\@secnumfont}{\bfseries}
}
\makeatother

\usepackage[utf8x]{inputenc}
\usepackage{comment}
\usepackage{amsmath, amssymb, amsgen, amsthm, amscd, xspace, color, mathrsfs, cite, tensor, enumitem, MnSymbol, amsthm, xfrac, url}

\usepackage{mathtools}
\DeclarePairedDelimiter{\ceil}{\lceil}{\rceil}
\DeclarePairedDelimiter\floor{\lfloor}{\rfloor}

\usepackage{tikz-cd}
\usepackage[colorlinks,
            linkcolor=blue,
            anchorcolor=blue,
            citecolor=blue,
            pagebackref
            ]{hyperref}
\newcommand{\tens}[2]{%
  \mathbin{\mathop{\otimes}\displaylimits_{#1}^{#2}}%
}

\newcommand{\relome}[1]{\Omega^{{}#1}_{X/\Delta}(\log Y)}
\newcommand{\lome}[1]{\Omega^{{}#1}_{X}(\log Y)}

\newcommand{\Addresses}{{% additional braces for segregating \footnotesize
  \bigskip
  \footnotesize

  \textsc{Department of Mathematics, the University of Michigan,
    MI 48109, USA}\par\nopagebreak
  \textit{E-mail address}: \href{mailto:qyc@umich.edu}\texttt{qyc@umich.edu}

}}

\newcommand{\Z}{{\mathbb{Z}}}
\newcommand{\Q}{{\mathbb{Q}}}
\newcommand{\R}{{\mathbb{R}}}
\newcommand{\C}{{\mathbb{C}}}

\newcommand{\sfw}{{\mathsf w}}

\newcommand{\fsl}{{\mathfrak s}{\mathfrak l}}

\newcommand{\cA}{{\mathcal A}}
\newcommand{\bD}{{\mathbf D}}

\newcommand{\sD}{{\mathscr D}}
\newcommand{\cF}{{\mathcal F}}
\newcommand{\cG}{{\mathcal G}}
\newcommand{\cH}{{\mathcal H}}
\newcommand{\sH}{{\mathscr H}}
\newcommand{\bH}{{\mathbf H}}
\newcommand{\sfH}{{\mathsf H}}
\newcommand{\sI}{{\mathscr I}}
\newcommand{\cM}{{\mathcal M}}

\newcommand{\cN}{{\mathcal N}}

\newcommand{\cP}{{\mathcal P}}

\newcommand{\cL}{{\mathcal L}}
\newcommand{\bR}{{\mathbf R}}

\newcommand{\sT}{{\mathscr T}}

\newcommand{\sO}{{\mathscr O}}
\newcommand{\cV}{{\mathcal V}}
\newcommand{\sV}{{\mathscr V}}
\newcommand{\sfX}{{\mathsf X}}
\newcommand{\sfY}{{\mathsf Y}}

\newcommand{\DR}{{\mathrm D}{\mathrm R}}
\newcommand{\Gal}{{\mathrm G}{\mathrm a}{\mathrm l\,}}
\newcommand{\gr}{{\mathrm g}{\mathrm r}}

\newcommand{\Cone}{{\mathrm C}{\mathrm o}{\mathrm n}{\mathrm e}}

\newcommand{\prim}{{\mathrm p}{\mathrm r}{\mathrm i}{\mathrm m\,}}
\newcommand{\End}{{\mathrm E}{\mathrm n}{\mathrm d\,}}

\newcommand{\Ext}{{\mathrm E}{\mathrm x}{\mathrm t}}

\newcommand{\cHom}{{\mathcal H}{o}{m}}

\newcommand{\id}{{\mathrm i}{\mathrm d\,}}

\newcommand{\sign}{{\mathrm s}{\mathrm i}{\mathrm g}{\mathrm n}}

\newcommand{\Red}{{\mathrm R}{\mathrm e}{\mathrm d}}
\newcommand{\Res}{{\mathrm R}{\mathrm e}{\mathrm s}}

 \makeatother

\newtheorem{thm}{Theorem}[section]

\newtheorem{lem}[thm]{Lemma}
\newtheorem{cor}[thm]{Corollary}

\newtheorem{thmA}{Theorem}
\newtheorem{corA}[thmA]{Corollary}

\newtheorem*{theorem}{Theorem}

\theoremstyle{definition}

\newtheorem{rem}[thm]{Remark}

\newtheorem{ex}[thm]{Example}

\begin{document}
\pagenumbering{roman}
\title{Limits of Hodge Structures via holonomic D-modules}
\date{\today}
\centerline{\author{Qianyu Chen}}

\begin{abstract}
We construct the limiting mixed Hodge structure of a degeneration of compact K\"ahler manifolds over the unit disk with a possibly non-reduced simple normal crossing singular central fiber via holonomic $\sD$-modules, generalizing some results of Steenbrink. Our limiting mixed Hodge structure does not carry a $\Q$-structure; instead, we use sesquilinear pairings on $\sD$-modules as a replacement. The limiting mixed Hodge structure can be computed by the cohomology of the cyclic coverings of certain intersections of components of the central fiber. Additionally, we prove the local invariant cycle theorem in this setting.
\end{abstract}

\maketitle

\pagenumbering{arabic}

\section{Introduction}

\subsection{Limits of Hodge structures} 
Consider a degeneration of compact K\"ahler manifolds over the unit disk. The cohomology of each smooth fiber carries a polarizable Hodge structure. It is natural to ask how the family of Hodge structures on the cohomologies of smooth fibers degenerates and how the cohomology of the central fiber relates to that of nearby fibers. These are two classical and central questions in Hodge theory. Before Saito's theory of mixed Hodge modules~\cite{Sai88,Sai90}, Schmid showed the existence of a limiting mixed Hodge structure for an abstract polarized variation of Hodge structures over the unit disk \cite{Schmid}, using Lie theoretic methods. For the variation of Hodge structures coming from a semistable family of compact K\"ahler manifolds over a $1$-dimensional base, the limiting mixed Hodge structure was first established by Steenbrink~\cite{Ste76}. His construction is equivalent to Schmid's in~\cite{Schmid}, but purely geometric.

In this paper, we construct the limits of Hodge structures for a family with the central fiber not necessarily reduced. The new feature of this paper is the usage of $\sD$-modules, which makes everything less topological and more combinatorial. It also removes the need for using derived categories and semistable reduction.

Let us now give a more precise description of the problem and the main results. Let $f:X\to \Delta$ be a proper holomorphic morphism from a K\"ahler manifold of dimension $n+1$ onto the unit disk. Suppose that $0$ is the only singular value and the central fiber $Y$ of $f$ has (possibly nonreduced) simple normal crossings. Denote by $X^*=X\setminus Y$ and $\Delta^*=\Delta\setminus \{0\}$. Then the higher direct image of the relative de Rham complex $\bR^k f_*\Omega^{\bullet+n}_{X^*/\Delta^*}$ over $\Delta^*$ is a vector bundle, where the shifting is needed to adopt the convention of the theory of perverse sheaves and $\sD$-modules; it underlies a polarizable variation of Hodge structure of weight $n$ over the punctured disk $\Delta^*$.  However, the higher direct image of the relative de Rham complex $\Omega^{\bullet+n}_{X/\Delta}$ may not be a vector bundle when $Y$ is singular. Steenbrink discovered a natural extension of the vector bundle $\bR^k f_*\Omega^{\bullet+n}_{X^*/\Delta^*}$ over the origin via the relative log de Rham complex. Let 
\[
\Omega_{X/\Delta}(\log Y)= \Omega_{X}(\log Y)/f^*\Omega_\Delta(\log 0)\text{ \,and\,  }\Omega^p_{X/\Delta}(\log Y)= \bigwedge^p \Omega_{X/\Delta}(\log Y),
\]
where $\Omega_X(\log Y)$ (resp. $\Omega_\Delta(\log 0)$) is the sheaf of meromorphic one-forms on $X$ (resp. $\Delta$) with log poles along $Y$ (resp. $0$). Then the \textit{relative log de Rham complex} is defined to be
\[
\Omega^{\bullet+n}_{X/\Delta}(\log Y)=\{\sO_X\to \Omega_{X/\Delta}(\log Y) \to \cdots \to \Omega^n_{X/\Delta}(\log Y)\}[n].
\]
Steenbrink showed in~\cite{Ste76} that $\bR^k f_*\Omega^{\bullet+n}_{X/\Delta}(\log Y)$ is a locally free integrable logarithmic connection with a pole along the origin and its residue operator $R$ has eigenvalues in $[0,1)\cap \Q$ for each $k\in\Z$. It follows from the projection formula that there exists a canonical isomorphism 
\[
\bR^k f_*\Omega^{\bullet+n}_{X/\Delta}(\log Y)\otimes \C(p) \simeq \bH^k(X,\Omega^{\bullet+n}_{X/\Delta}(\log Y)|_{X_p})
\]
for every fiber $X_p$ over any point $p\in \Delta$, where $\C(p)$ denotes the residue field of $p$. Indeed, the vector bundle $\bR^k f_*\Omega^{\bullet+n}_{X/\Delta}(\log Y)$ is Deligne's canonical extension \cite{DP70} of $\bR^k f_* \Omega^{\bullet+n}_{X^*/\Delta^*}$ with eigenvalues of the residues of the log connection in the interval in $[0,1)$. We can think of the vector space $\bH^k(X,\Omega^{\bullet+n}_{X/\Delta}(\log Y)|_Y)$ as a canonical specialization of $\bH^k(X_p, \Omega^{\bullet+n}_{X_p})$ for general fibers $X_p$.

\begin{theorem}[Steenbrink]
Suppose that the central fiber $Y$ is a reduced divisor with simple normal crossings. Then $\bH^k(X,\relome{\bullet+n}|_Y )$ underlies a limiting mixed Hodge structure with a $\Q$-structure which can be expressed in terms of the cohomology of certain intersections of components of $Y$ via spectral sequences.
\end{theorem} 

When $Y$ is reduced, the residue operator $R$ is nilpotent on the hypercohomology of $\relome{\bullet+n}|_Y$ so it gives a monodromy filtration $W_\bullet=W_\bullet(R)$ on the vector space $\bH^k(X,\relome{\bullet+n}|_Y )$ uniquely characterized by two properties: (1) $R\cdot W_\bullet\subset W_{\bullet-2}$ and (2) $R^r:\gr^W_r\to \gr^W_{-r}$ is an isomorphism for every $r\geq 0$. The filtration $W_\bullet(R)$ is called the monodromy filtration because $\exp\left(-2\pi\sqrt{-1}R\right)$ is the monodromy induced by the generator of the fundamental group of $\Delta^*$. Steenbrink showed that the monodromy filtration is the weight filtration of the limiting mixed Hodge structure when $f$ is projective, and this was later generalized to the K\"ahler case by Guill\'en and Navarro Aznar in~\cite{hl}. 

Steenbrink further pointed out that his limiting mixed Hodge structure only depends on the log structure associated to the semistable family $f:X\to \Delta$~\cite{SteLog}. Fujisawa extended Steenbrink's results in~\cite{Ste76, SteLog} to semistable K\"ahler families over the polydisk and furthermore to the log geometry setting~\cite{Taro99,Taro08, Taro14}. Recently, Nakkajima announced a simpler proof of Fujisawa's results~\cite{Nak21}.

\subsection{Main results}

We revisit Steenbrink's theorem and construct the limiting mixed Hodge structure of the degeneration over the unit disk $\Delta$ via the theory of holonomic $\sD$-modules. One key difference with Steenbrink's result is that we allow the central fiber $Y$ to be nonreduced. 

\begin{thmA}\label{thm:main}
Suppose that the central fiber $Y$ is a divisor with possibly non-reduced simple normal crossings. Let $R_n$ (resp. $R_s$) denote the nilpotent (resp. semisimple) part of the Jordan-Chevalley decomposition of the residue operator $R$ on the graded vector space $\bigoplus_k \bH^k(X,\relome{\bullet+n}|_Y)$. Then, each eigenspace of $R_s$ on
\[
\bigoplus_{k,\ell} \gr^W_\ell \bH^k(X,\relome{\bullet+n}|_Y)
\]
underlies a limiting polarized bigraded Hodge-Lefschetz structure over $\C$ of central weight $n$, where $W_\bullet=W_\bullet(R_n)$ is the monodromy filtration associated with $R_n$. In particular, $\bH^k(X,\relome{\bullet+n}|_Y)$ underlies a limiting mixed Hodge structure for each $k\in \Z$.
\end{thmA}

A polarized bigraded Hodge-Lefschetz structure is essentially a direct sum of polarized Hodge structures of different weights, preserved by an $\fsl_2(\C)\times \fsl_2(\C)$-action. Refer to \S\ref{subsec:pbhl} for the precise definition. In the context of Theorem \ref{thm:main}, the $\fsl_2(\C)\times \fsl_2(\C)$-action is induced by the operator $R_n$ and $2\pi\sqrt{-1}L$, where $L=[\omega]\wedge$ represents the Lefschetz operator for a K\"ahler form $\omega$. Specifically, each component $\gr^W_\ell \bH^k(X,\relome{\bullet+n}|_Y)$ is a Hodge structure of weight $n+k+\ell$, and there are two Hard Lefschetz-type isomorphisms of Hodge structures:

\begin{itemize}
  \item $\left(2\pi\sqrt{-1}L\right)^k:\gr^W_\ell \bH^{-k}(X,\relome{\bullet+n}|_Y)  \to \gr^W_\ell \bH^k(X,\relome{\bullet+n}|_Y)(k)$  \quad for $k\geq 0$ and $\ell\in\Z$;\\
  \item $R^\ell_n: \gr^W_{\ell} \bH^k(X,\relome{\bullet+n}|_Y)  \to \gr^W_{-\ell} \bH^k(X,\relome{\bullet+n}|_Y)(-\ell)$  \quad for $\ell \geq 0$ and $k\in\Z$.
\end{itemize}

Our argument also says that the limiting mixed Hodge structure can be computed in terms of the cohomology of certain cyclic coverings of intersections of components of $Y$ via spectral sequences.

As an application, we establish the local invariant cycle theorem, which is a key component in the Clemens-Schmid sequence \cite{Cle1977}, even when the central fiber $Y$ is non-reduced. 

\begin{thmA}[Local Invariant Cycle Theorem]\label{thm:lic}%[Theorem \ref{loc}]
Suppose we are in the same setting as in Theorem \ref{thm:main}. Then the following sequence of mixed Hodge structures is exact:
\[
H^{\ell}(Y,\C[n])\to \bH^{\ell}\left(X,\relome{\bullet+n}|_Y\right) \xrightarrow{R} \bH^{\ell}\left(X,\relome{\bullet+n}|_Y\right)(-1).
\]
In particular, all cohomology classes that are invariant under the monodromy action come from the cohomologies of $Y$. 
\end{thmA}

The local invariant cycle theorem was originally proven by Deligne in an algebraic setting for schemes \cite[Theorem 4.1.1]{Hodge2}, and later studied in \cite{Ste76}, \cite{Cle1977}, and \cite{hl} for semistable K"ahler degenerations. It was also generalized to the theory of mixed Hodge modules by Saito \cite{Sai88,Sai90}.

\subsection{What is new} 
The traditional way to construct the limiting mixed Hodge structure, when the central fiber is not reduced, is to use the semistable reduction theorem. For example, Steenbrink mentioned his hypothesis on the central fiber is very general in~\cite[p230]{Ste76}, and there is a more detailed explanation in his joint books with Peters~\cite[11.2.2]{PS08}. Mumford’s semistable reduction~\cite{KKMS73} performs a sequence of base changes, normalizations, and blow-ups, that turns a given degeneration of compact K\"ahler manifolds over the unit disk into a semistable degeneration. This procedure is not canonical, and is hard to carry out in practice and it gives no information on how the non-reduced central fiber, especially the multiplicities of each component, affects the limiting mixed Hodge structure. Moreover, the information of eigenvalues of the monodromy will be lost after running semistable reduction.

By contrast, our approach via $\sD$-modules is very direct and concrete: the local formulas for the $\sD$-module are very simple and combinatorial so it is easy to do computations with them. We can construct the limiting mixed Hodge structure directly from the given degeneration, without using semistable reduction. In the process, we also get much more information about the action of the monodromy transformation, in particular about its eigenvalues and generalized eigenspaces. These contain interesting geometric information about the family, which is lost by going to a semistable reduction. For example, we can readily get the Jordan block of the residue operator. Also, it is obvious how the multiplicities of the central fiber affect the limiting mixed Hodge structure; in particular, we can see very clearly what the different eigenvalues are. Our approach also enables us to generalize the local invariant cycle theorem into the non-reduced case by local calculations.

\subsection{Strategy Overview}

Steenbrink noticed that $\relome{\bullet+n}|_Y$ can be seen as a realization of the nearby cycles $\psi_f\left(\C_X[n+1]\right)$, which implies that the function $p\mapsto \dim \bH^k(X,\relome{\bullet+n}|_{X_p})$ is constant on $\Delta$. Thanks to Grauert's theorem, it follows that the sheaf $\bR^k f_\Omega^{\bullet+n}{X/\Delta}(\log Y)$ is locally free. The logarithmic connection on $\bR^k f_*\relome{\bullet+n}$ arises as the higher direct image of an operator $\nabla$ fits into a distinguished triangle in $\bD^b(X,\C)$:
\[
f^*\Omega_\Delta\otimes \relome{\bullet+n-1} \to \Omega^{\bullet+n}_X(\log Y) \to \relome{\bullet+n} \xrightarrow{\nabla}  f^*\Omega_\Delta\otimes\relome{\bullet+n}.
\]
Trivializing $f^*\Omega_\Delta=\sO_X\cdot \frac{dt}{t}$, the induced operator 
\[
[\nabla]: \relome{\bullet+n}|_Y\to \relome{\bullet+n}|_Y
\] 
has a characteristic polynomial whose roots lie in the interval $[0,1)\cap\Q$. The action of $[\nabla]$ on the hypercohomology of $\relome{\bullet+n}|_Y$ coincides with the residue operator $R$ of the logarithmic connection. Therefore, the study of the monodromy filtration of $R$ on the cohomology is equivalent to making the monodromy filtration of $[\nabla]$ on the complex $\relome{\bullet+n}|_Y$ explicit. However, one of the main challenges we face is constructing the monodromy filtration on the complex $\relome{\bullet+n}|_Y$ since the operator $[\nabla]$ exists only in the derived category. Steenbrink tackled this problem indirectly by resolving the relative log de Rham complex using a specific double complex, but this approach is limited to the case when $Y$ is reduced. Additionally, he had to demonstrate that the monodromy filtrations are defined over $\Q$ through intricate topological arguments, ensuring that all the data yields a rational cohomological mixed Hodge complex.

Through the Riemann-Hilbert correspondence~\cite{KM84,MZ84}, there exists a regular holonomic $\sD$-module whose de Rham complex is isomorphic to $\relome{\bullet+n}|_Y$ in $\bD^b(X,\C)$ since the nearby cycle functor preserves perversity~\cite{Be87}. This $\sD$-module perspective allows us to derive the monodromy filtration easily through local calculations on a single $\sD$-module, bypassing the need for derived categories. Furthermore, we provide a concrete description of the primitive parts of the associated quotient of the monodromy filtrations. Instead of relying on $\Q$-structures, we consider sesquilinear pairings on $\sD$-modules, which serve as analogs of polarizations on Hodge structures. The polarization on the bigraded Hodge-Lefschetz structure in Theorem~\ref{thm:main} is induced by a sesquilinear pairing. Although some topological information is lost, the sesquilinear pairings we use can be constructed purely algebraically and involve only symbolic calculations. The local calculations and the sesquilinear pairing support the choice of the monodromy filtration of $[\nabla]$ as the appropriate weight filtration. Our method also naturally accommodates the case when $Y$ is non-reduced. Now, let's delve into the construction with more details.

We begin by providing an alternative proof of the fact that $\bR^k f_*\Omega^{\bullet+n}_{X/\Delta}(\log Y)$ is locally free. Our approach relies solely on the eigenvalues of $[\nabla]$ being in the interval $[0,1)$ (Theorem~\ref{mainlogdr}). Next, we proceed to translate the data of the relative log de Rham complex to the $\sD$-module framework (see Section \ref{sec:trand}):
 
\begin{thmA}
There exists a filtered holonomic $\sD_X$-module $(\cM, F_\bullet\cM)$ such that its de Rham complex $\DR_X\cM$ with the induced filtration $F_\bullet\DR_X\cM$ is isomorphic to $\relome{\bullet+n}|_Y$ with the stupid filtration in the derived category of filtered complexes of $\C$-vector spaces. Furthermore, there exists an operator $R:(\cM,F_\bullet\cM)\to (\cM,F_{\bullet+1}\cM)$ whose eigenvalues are in the interval $[0,1)\cap\Q$, and under the above isomorphism, $\DR_XR$ can be identified with $[\nabla]$.
\end{thmA}

Then we investigate the Jordan block of the operator $R$. Let $\cM_{\geq \alpha}$ (resp. $\cM_{>\alpha}$) be the submodule of $\cM$ spanned by the generalized eigen-modules $\ker(R-\lambda)^{\infty}$ for $\lambda\geq\alpha$ (resp. $\lambda >\alpha$).  Let $\cM_{\alpha}=\cM_{\geq\alpha}/\cM_{>\alpha}$. Note that $\cM_\alpha$ is canonically isomorphic to $\ker(R-\alpha)^\infty$ and therefore $R_\alpha=R-\alpha$ acts nilpotently on $\cM_\alpha$. Using an idea of Saito~\cite{Sai90}, we filter $\cM_\alpha$ by
\[
F_\ell\cM_\alpha=\frac{F_\ell\cM\cap\cM_{\geq \alpha}+\cM_{>\alpha}}{\cM_{>\alpha}},\quad \text{ for } \ell\in\Z.
\]
The filtration $F_\bullet\cM_\alpha$ is different from the naive one $F_\bullet\cM\cap \ker(R-\alpha)^\infty$. The reason why we do not use the naive filtration is twofold: $F_\bullet\cM_\alpha$ provides the correct weight and is computationally more tractable. We proceed to prove that any power of the operator $R_\alpha$ strictly respects the filtration  $F_\bullet\cM_\alpha$. In other words, for every $\ell\geq 0$, we have the relation $R_\alpha^\ell F_\bullet\cM_\alpha=F_{\bullet+\ell}\cM\cap R_\alpha^\ell\cM_\alpha$  (Theorem~\ref{strict} for the case $Y$ is reduced and Theorem~\ref{strictN} for the general case).

This indicates that the monodromy filtration $W_\bullet\cM_\alpha$ and $F_\bullet\cM_\alpha$ interact very well. Note that the monodromy filtration associated to $R_\alpha$ is the same as that of $R_n$ on $\cM_\alpha$, which is the nilpotent part of the Jordan-Chevalley decomposition of $R$. We have the induced good filtrations 
\[
F_\bullet W_r\cM_\alpha=F_\bullet\cM\cap W_r\cM_\alpha \quad\text{ and } \quad F_\bullet\gr^W_r\cM_\alpha=F_\bullet W_r\cM_\alpha/F_{\bullet}W_{r-1}\cM_\alpha.
\] 
Denote by $\cP_{\alpha,\ell}=\ker R_\alpha^{\ell+1}\cap \gr^W_\ell\cM_\alpha$ the $\ell$-th primitive for $\ell\geq 0$, which is isomorphic to 
\[
\frac{\ker R_\alpha^{\ell+1}}{\ker R_\alpha^\ell+\mathrm{im\,}R_\alpha\cap \ker R_\alpha^{\ell+1}}.
\] 
We endow it with the induced good filtration $F_\bullet \cP_{\alpha,\ell}= \mathrm{im}\,(F_\bullet\cM\cap\ker R^{\ell+1}_\alpha \to \cP_{\alpha,\ell}).$ A corollary of the strictness of every power of $R_\alpha$ is that the Lefschetz decomposition of $\gr^W\cM_\alpha$ respects the good filtrations, i.e.
\[
F_\bullet\gr^W_r\cM_\alpha=\bigoplus_{\substack{\ell\geq 0,-\frac{r}{2}}} R_\alpha^{\ell}F_{\bullet-\ell}\cP_{\alpha,r+2\ell}.
\]
See Theorem~\ref{FW} for the case $Y$ is reduced and Theorem~\ref{FWN} for the general case. This corollary suggests that it suffices to study the hypercohomology of each primitive part, as then the primitive parts serve as the source for the pure polarized Hodge structures.

We will construct a sesquilinear pairing $S_\alpha:\cM_\alpha\otimes_\C \overline{\cM_\alpha}\to\mathfrak{C}_X$ using the Mellin transformation~\cite{SC02}, where $\overline{\cM_\alpha}$ is the naive conjugation of $\cM_\alpha$ and $\mathfrak C_X$ is the sheaf of currents. Both $\cM_\alpha\otimes_\C \overline{\cM_\alpha}$ and $\mathfrak C_X$ are canonically $\sD_X\otimes_\C \sD_{\overline X}$-modules where $\sD_{\overline X}$ denotes the sheaf of anti-holomorphic differential operators and the sesquilinear pairing is just a morphism of $\sD_X\otimes_\C \overline{\sD_X}$-modules. The sesquilinear pairings on $\cM_\alpha$ is an analogy of polarization on a Hodge structure:  a complex polarized Hodge structure of weight $n$ can be described as a filtered vector space $(V,F^\bullet)$ with a Hermitian pairing $S$ such that $(-1)^{n-p}S$ is a Hermitian inner product on $F^p\cap G^{n-p}$ where $G^{n-p}$ is the $S$-orthogonal complement of $F^{p+1}$. The sesquilinear pairing $S_\alpha$ induces the second filtration on the hypercohomology of $\DR_X\cM_\alpha$. For example, if $Y$ is reduced, the pairing on $\cM$ is induced by 
\[
\mathrm{Res}_{s=0} \frac{\varepsilon(n+2)}{(2\pi\sqrt{-1})^{n+1}}\int_\Delta |t|^{2s}\frac{dt}{t}\wedge\frac{d\bar t}{\bar t} \int_{X_t}: \Omega^{n}_{X/\Delta}(\log Y) \otimes_\C \overline{\Omega^{n}_{X/\Delta}(\log Y)} \to \mathfrak{C}_X,
\]
where the constant scalar ${\varepsilon(n+2)}{(2\pi\sqrt{-1})^{-(n+1)}}$ depending on the dimension is used to make the pairing independent of the choice of orientation. The Mellin transformation is used here to extract the principal part of the asymptotic expansion of fiber-wise integration $\int_{X_t}:\omega_{X_t}\otimes_\C \overline{\omega_{X_t}}\to \mathscr{C}_{X_t}$. We refer to the $\S$\ref{subsec:ds} for the definition of sesquilinear pairings on $\sD$-module.

The operator $R_\alpha$ is self-adjoint with respect to the pairing $S_\alpha:\cM_\alpha \otimes_\C \overline \cM_\alpha\to \mathfrak{C}_X$, i,e, $S_\alpha(-,R_\alpha-)=S_\alpha(R_\alpha-,-)$. 
See $\S$\ref{sec:redses} for the case that $Y$ is reduced $\S$\ref{sec:sesquil} for the general case. This implies we have an induced pairing on the associated graded modules:
\[
S_{\alpha,r}:\gr^W_r\cM_\alpha\otimes_\C \gr^W_{-r}\cM_\alpha \to \mathfrak C_X.
\]
Then $P_{R_\alpha}S_{\alpha,r}=S_{\alpha,r}\circ \left(\id\otimes_\C R_\alpha^r\right)$ defines a sesquilinear pairing on the primitive part $\cP_{\alpha,r}$.

\begin{thmA}\label{D}%[\ref{iden}, \ref{idenN}, \ref{main}, \ref{mainN}]
The graded vector space
\[
\bigoplus_{\ell\in\Z} \bH^\ell(X,\DR_X\cP_{\alpha,r})
\]
together with the filtration induced by $F_\bullet\cP_{\alpha,r}$ and the sesquilinear pairing induced by $P_{R_\alpha}S_{\alpha,r}$ determine a polarized Hodge-Lefschetz structure of central weight $n+r$ with $\fsl_2(\C)$-action induced by $2\pi\sqrt{-1}L$.
\end{thmA}

A polarized Hodge-Lefschetz structure basically is a direct sum of Hodge structures of different weights preserving by an $\fsl_2(\C)$-action modeled by the direct sum of all the cohomology groups of a compact K\"ahler manifold. We refer to $\S$\ref{subsec:phl} for the definition of polarized Hodge-Lefschetz structures. To illustrate the idea of Theorem~\ref{D},  assume for a moment that $Y$ is reduced and $Y=\sum_{i\in I} Y_i$ where the $Y_i$'s are smooth components and $I$ a finite index set. Then the endomorphism $R$ is nilpotent and this implies that $\cM=\cM_0$. Denote by $Y^J=\bigcap_{i\in J} Y_i$ for any non-empty subset $J$ of $I$. Let $\tau^J:Y^J \to X$ be the closed embedding and $\tau^{(r+1)}:\tilde Y^{(r+1)}= \coprod_{|J|=r+1}Y^J \to X$ be the natural morphism for every $r\geq 0$. Put $\cP_r=\cP_{0,r}$ for simplicity. We will show that there exists a filtered isomorphism (Theorem~\ref{iden})
\[
\phi_r:(\cP_r,F_\bullet\cP_r)\to \tau^{(r+1)}_+\omega_{\tilde Y^{(r+1)}}(-r).
\]
Here, the Tate twist of a filtered $\sD$-module is $(\cN,F_\bullet\cN)(-r)=(\cN,F_{\bullet+r}\cN)$. Moreover, the isomorphism respects the pairing $P_RS_r$ on $\cP_r$ (Theorem~\ref{main}):
\[
P_RS_r(-,-)= \frac{(-1)^r}{(r+1)!}\tau^{(r+1)}_+S_{{\tilde Y^{(r+1)}}} (\phi_r -,{\phi_r-}),
\]
where $S_{{\tilde Y^{(r+1)}}}$ is the standard pairing on $\omega_{\tilde Y^{(r+1)}}$. Therefore, $\bH^k(X,\DR_X\cP_r)$ is isomorphic to $H^{n-r+k}(\tilde Y^{(r+1)},\C)(-r)$ as polarized Hodge structures of weight $n+r+k$. We get a polarized Hodge-Lefschetz structure of central weight $n+r$ with $\fsl_2(\C)$-action induced by $2\pi\sqrt{-1}L$ on $\bigoplus_{\ell\in\Z} \bH^\ell(X,\DR_X\cP_r)$. For the case when $Y$ is non-reduced, we will identify the primitive parts $\cP_{\alpha,r}$ with certain filtered holonomic $\sD$-modules coming from the cyclic coverings (Theorem~\ref{idenN}), and the identification also respects the sesquilinear pairing (Theorem~\ref{mainN}). As a direct consequence, we obtain

\begin{thmA}
Let $V^\alpha_{\ell,k}=\bH^{\ell}(X,\gr^W_k\DR_X\cM_\alpha)$ be the relabeling of the first page of the weight spectral sequence. Then $V^\alpha=\bigoplus_{k,\ell \in\Z} V^\alpha_{\ell,k}$ is a polarized bigraded Hodge-Lefschetz structure of central weight $n$ with the polarization induced by $S_\alpha$ and $\fsl_2(\C)\times\fsl_2(\C)$-action induced by $2\pi\sqrt{-1}L$ and $R_\alpha$. Moreover, the differential $d_1$ of the first page of the weight spectral is a differential of polarized bigraded Hodge-Lefschetz structure.
\end{thmA}

By a formal argument of Guill\'en and Navarro Aznar~\cite{hl}, which follows some ideas of Deligne and Saito:

\begin{corA}
We have the following statements: 
\begin{enumerate}
  \item the Hodge spectral sequence degenerates at $\prescript{}{F}{E_1}$;
  \item the weight spectral sequence degenerates at $\prescript{W}{}{E_2}$;
  \item the $\alpha$-generalized eigenspace of the bigraded vector space 
  \[
  \prescript{W}{}{E_2}=\bigoplus_{\ell,k \in\Z}\gr^W_{\ell} \bH^k(Y,\Omega^{\bullet+n}_{X/\Delta}(\log Y)|_Y)
  \] 
  with respect to $R$ underlies a polarized bigraded Hodge-Lefschetz structure of central weight $n$ with polarization induced by $S_\alpha$ and $\fsl_2(\C)\times\fsl_2(\C)$-action induced {by $2\pi\sqrt{-1}L$ and $R_\alpha$}.
\end{enumerate}
\end{corA}
Note that the third statement in the above Corollary is equivalent to the Theorem~\ref{thm:main}; therefore, we finish the proof of Theorem~\ref{thm:main}. See Theorem~\ref{rd1pbhl} and Corollary~\ref{rmhc}, when $Y$ is reduced. See Theorem~\ref{d1pbhl} and Corollary~\ref{mhc}, when $Y$ is allowed to be non-reduced,.

\subsection{Outline} We first review basic notions on holonomic filtered $\sD$-modules, integrable logarithmic connections and polarized bigraded Hodge-Lefschetz structures in $\S$\ref{sec:pre}. Then we set up the relative log de Rham complex and construct a log connection on its higher direct images in $\S$\ref{sec:logdr}. We transfer all of the data on the relative log de Rham complex into a filtered holonomic $\sD$-module in $\S$\ref{sec:trand}. To avoid the messy calculations, we first prove everything in the reduced case in $\S$\ref{sec:redstrict} and $\S$\ref{sec:redses}. The idea for the non-reduced case is almost the same but requires some Hodge theory of cyclic coverings. We construct some $\sD$-modules in $\S$\ref{subsec:sp} as the summand of the primitive part and prove their hypercohomology underlie canonical polarized Hodge structures in $\S$\ref{subsec:kpcyc}. Lastly, we prove the local invariant cycle theorem in $\S$\ref{sec:app}.

\subsection{Acknowledgement} The author gratefully acknowledges his advisor, Christian Schnell, for introducing him to this topic and for sharing ideas and insights during their weekly meetings. Many of the ideas presented in this paper should be attributed to him. The author would also like to express his gratitude to Guodu Chen and Nathan Chen for reading a draft of this paper and providing helpful feedback. Mircea Mus\-ta\-\c{t}\u{a}'s and Brad Dirks's insights were also greatly appreciated. The author is grateful to Yilong Zhang and Ruijie Yang for catching some typos in the paper. Finally, the author thanks the anonymous referee for providing valuable suggestions and corrections.

\section{Preliminaries}\label{sec:pre}
\subsection{Filtered $\sD$-modules with sesquilinear pairings}\label{subsec:ds}

We will work with right $\sD$-modules unless further specified. Let $Z$ be a complex manifold of dimension $n$ and denote by $\Omega^p_Z$ the sheaf of holomorphic $p$-forms and $\sT_Z$ the sheaf of holomorphic tangent vectors fields. For a filtered $\sD_Z$-module we mean a pair $(\cN, F_\bullet\cN)$ where $\cN$ is a coherent $\sD_Z$-module and $F_\bullet\cN$ is a good filtration. Occasionally, we will abuse notation and use $\cN$ to also denote the filtered $\sD_Z$-module when the filtration is clear. Denote by $\gr^F\sD_Z=\bigoplus_{\ell\in Z}
\gr^F_{\ell}\sD_X$ the associated graded algebra and $\gr^F\cN=\bigoplus_{\ell\in\mathbb{Z}}\gr^F_{\ell}\cN$ the associated graded module. Note that $\gr^F\cN$ is a coherent $\gr^F\sD_Z$-module. Let $T^*Z=\mathrm{Spec}_Z\left(\gr^F\sD_X\right)$ be the algebraic cotangent bundle and $T^*_V Z$ the geometric conormal bundle of a subvariety $V$ in $Z$. The \textit{characteristic variety} of $\cN$ is the support of $\gr^F\cN$ on $T^*Z$ and is denoted by $char(\cN)$. The \textit{characteristic cycle} of $\cN$ is the cycle associated to the coherent sheaf $\gr^F\cN$ on $T^*Z$ and is denoted by $cc(\cN)$. Neither the characteristic variety nor the characteristic cycle depends on the choice of filtration~\cite{HTT}. For example, the canonical bundle $\omega_Z$ is naturally a holonomic $\sD_Z$-module with action 
\[
m.\xi=-d(\xi\righthalfcup m)
\]
for local sections $\xi$ of $\sT_Z$ and $m$ of $\omega_Z$. It also naturally has a good filtration
\begin{equation}\label{eq:constant}
F_\ell\omega_Z=\left\{
\begin{aligned}
\omega_Z & , & \ell\geq -n; \\
0 & , & \ell< -n.
\end{aligned}
\right.
\end{equation}
Then one can compute $cc(\omega_Z)=[T^*_Z Z]$ which is the cycle of the zero section of the cotangent bundle. We call $\cN$ a \textit{holonomic} $\sD_Z$-module if $\dim char(\cN)=n$. See more details in~\cite{HTT}. A \textit{Tate twist} of filtered $\sD_Z$-module is defined to be $\cN(-r)=(\cN,F_{\bullet+r}\cN)$ for any $r\in\Z$.

Denote by $\bD^b(Z, \C)$ the bounded derived category of complexes with values in finite-dimensional $\C$-vector spaces and $\bD^b(Z,\sD)$ the bounded derived category of $\sD_Z$-modules. Denote by $\bD^b_h(Z,\sD)$ the full subcategory of $\bD^b(Z,\sD)$ whose objects are complexes with holonomic cohomologies. For a morphism $f: Z\to W$ between complex manifolds, denote by $\bR f_*, \bR f_! : \bD^b(Z,\C)\to \bD^b(W,\C)$ the derived pushforward and proper pushforward functors respectively and $\bR^k f_*,\bR^k f_!$ the $k$-th cohomology functors respectively. For any $\cN^\bullet\in \bD^b(Z,\sD)$, the pushforward functor and the proper pushforward functor $f_+,f_\dag: \bD^b(Z,\sD)\to \bD^b(W,\sD)$ are by definition, respectively
\[
f_+\cN^\bullet= \bR f_*\left(\cN^\bullet\tens{\sD_Z}{\mathbf L} \sD_{Z\to W}\right) \quad \text{and} \quad f_\dag\cN^\bullet= \bR f_!\left(\cN^\bullet\tens{\sD_Z}{\mathbf L} \sD_{Z\to W}\right),
\]
where $\sD_{Z\to W}=f^*\sD_W$ is the transfer module. In fact, the functor $f_\dag$ preserves the holonomicty, i.e., $f_\dag: \bD^b_h(Z,\sD)\to \bD^b_h(W,\sD)$ (see~\cite{HTT}). Of course, if $f$ is proper or proper on the support of $\cN$ then $f_+=f_\dag$. The \textit{de Rham complex} of $\cN$ is 
\[
\DR_Z\cN\colon =\cN\otimes\bigwedge^{-\bullet}\sT_Z=\{\cN\otimes \bigwedge^{n}\sT_Z\cN\to \cN\otimes \bigwedge^{n-1}\sT_Z \to \cdots \to\cN\}
\]
placed in degree $-n,\dots,-1,0$. If without further indication, tensor products are always taken over $\sO$-modules.  Some authors also call it the Spencer complex. The de Rham complex of $\omega_Z$ 
\[
\omega_Z\otimes\bigwedge^{-\bullet}\sT_Z=\{\omega_Z\otimes \bigwedge^{n}\sT_Z\to \omega_Z\otimes \bigwedge^{n-1}\sT_Z \to \cdots \to\omega_Z\}
\]
is isomorphic to the usual de Rham complex $\DR_Z\sO_Z=\Omega^{n+\bullet}_Z$ of $Z$ under the isomorphisms 
\begin{equation}\label{drlr}
\omega_Z\otimes \bigwedge^p \sT_Z \to \Omega^{n-p}_Z, \,\,\, \omega\otimes \partial_J\mapsto (-1)^{n-j_1+\cdots+n-j_p}dz_{{\bar J}},
\end{equation}
where $\partial_J$ is a local section of $\bigwedge^p\sT_Z$, $J$ is ordered index set and ${\bar J}$ is the complement with the natural ordering, and $\omega=dz_1\wedge dz_2\wedge\cdots\wedge dz_{n}$. If $F_\bullet\cN$ is a good filtration, the de Rham complex is also filtered by
\[
\begin{aligned}
  F_{\ell}\DR_Z\cN &=F_{\ell+\bullet}\cN\otimes\bigwedge^{-\bullet}\sT_Z \\
  &=\left\{F_{\ell-n}\cN\otimes \bigwedge^{n}\sT_Z\cN\to F_{\ell-n+1}\cN\otimes \bigwedge^{n-1}\sT_Z \to \cdots \to F_{\ell}\cN\right\}.
\end{aligned}
\]
The direct image functor and the de Rham functor commute~\cite[Corollary 4.4.4]{MS}: 
\[
\bR f_!\circ\DR_Z=\DR_W\circ f_\dag.
\] 

A \textit{sesquilinear pairing} $S$ on $\sD_Z$-module $\cN$ is a $\sD_{Z,\overline{Z}}$-module morphism $S:\cN\otimes_\C \overline \cN\to \mathfrak{C}_Z$. Here, $\sD_{Z,\overline{Z}}=\sD_Z\otimes_\C \sD_{\overline Z}$ for $\overline{\sD_Z}$ is the sheaf antiholomorphic differential operators, $\overline \cN$ is the stupid conjugate of $\cN$ as a $\overline \sD_Z$-module and $\mathfrak{C}_Z$ is the sheaf of currents on $Z$ with natural $\sD_{Z,\overline{Z}}$-module structure. We have the proper pushforward functor similarly as above on $\sD_{Z,\overline{Z}}$-modules and also call it $f_\dag$:
\[
f_\dag (-)\colon = \bR f_! (- \tens{\sD_{Z,\overline Z}}{\mathbf L} \sD_{Z,\overline Z\to W,\overline W} ),
\] 
where $\sD_{Z,\overline Z\to W,\overline W}\colon =f^*\sD_{W,\overline W}$ is the transfer module. Because of the natural morphism $f_\dag\mathfrak{C}_Z \to \mathfrak{C}_W$, we can pushforward the sesquilinear pairing to get
\[
\sH^0f_\dag S_k: \sH^kf_\dag\cN\otimes_\C \overline{\sH^{-k}f_\dag \cN}\to \sH^0f_\dag \cN\otimes_\C \overline{\cN}\to \mathfrak{C}_W.
\]
 If $f$ is a closed embedding then $f_+S: f_+\cN\otimes_\C f_+\overline\cN \to \mathfrak{C}_W$. If $W$ is a point, then we have an induced pairing on the complex
\[
f_\dag S: \DR_{Z,\overline Z}\cN\otimes_\C \overline{\cN} \to \DR_{Z,\overline Z}\mathfrak{C}_Z\simeq \C[2n],
\]
where $\DR_{Z,\overline Z}\cN\otimes_\C \overline{\cN} \simeq \DR_Z \cN\otimes_\C \overline{\DR_Z \cN}$. Taking cohomology at $0$-th degree gives, for each $k\in\Z$,
\begin{equation}
\bH^k_c(Z,\DR_Z\cN)\otimes \bH^{-k}_c(Z,\overline{\DR_Z\cN}) \to \bH^0_c(Z,\DR_{Z,\overline Z}\cN\otimes_\C \overline{\cN}) \to H^{2n}_c(Z, \C)\simeq \C.
\end{equation}
\begin{ex}
The $\sD_Z$-module $\omega_Z$ carries a natural pairing $S_Z:\omega_Z\otimes_\C \overline{\omega_Z}\to \mathfrak{C}_Z$,
\begin{equation}\label{np}
\langle S_Z(m', \overline{m''}),\eta\rangle =\frac{\varepsilon(n+1)}{(2\pi\sqrt{-1})^{n}}\int_Z \eta\cdot m'\wedge \overline{m''},
\end{equation}
for $m', m''$ local sections of $\omega_Z$, $\eta$ a test function on $Z$ and $\varepsilon(n+1)=(-1)^{\frac{n(n+1)}{2}}$. The coefficient $\frac{\varepsilon(n+1)}{(2\pi\sqrt{-1})^{n}}$ in the definition is chosen so that $\frac{\varepsilon(n+1)}{(2\pi\sqrt{-1})^{n}}m\wedge\overline{m}=|m|^2$ is a positive current for any local section $m$ of $\omega_Z$ and elimination the choice of orientation (see more details in $\S$\ref{subsec:phl}). The pairing $S_Z:\omega_Z\otimes_\C \overline{\omega_Z}\to \mathfrak{C}_Z$ yields a collection of pairings
\[
\bH^k_c(Z,\DR_Z\omega_Z)\otimes_\C \overline{\bH^{-k}_c(Z,\DR_Z\omega_Z)} \to \C.
\]
\end{ex}

\subsection{Logarithmic connections} 
If $D=\sum a_iD_i$ is a simple normal crossing divisor on a complex manifold $Z$ of $\dim Z=n$ for $a_i\geq 0$, denote by $\Omega_Z(\log D)$ the sheaf of meromorphic differential $1$-forms with logarithmic poles along $D_{\mathrm{red}}=\sum D_i$ and denote by $\Omega^p_Z(\log D)=\bigwedge^p\Omega_Z(\log D)$ the meromophic $p$-forms with logarithmic pole along $D$. Each $\Omega^p_Z(\log D)$ is a locally free $\sO_Z$-module. 

In our convention, the \textit{de Rham complex} of $Z$ is $\DR_Z\sO_Z$
\[
\Omega^{\bullet+n}_{Z}=\{\sO_Z\to \Omega_Z \to \cdots \to \Omega^{n}_Z\}[n].
\]
The \textit{log de Rham complex} is
\[
\Omega^{\bullet+n}_{Z}(\log D)=\{\sO_Z\to \Omega_Z(\log D)\to \cdots \to \Omega^{n}_Z(\log D)\}[n].
\]
We will follow the Koszul sign rule: for a chain complex $C^\bullet$ with differential $d$, the shifted complex $C^{\bullet+n}=C^\bullet [n]$ equipped with differential $(-1)^nd$. We define residue along $D_i$ by (see~\cite[2.5]{EV})
\[
\Res_{D_i}: \Omega^{\bullet+n}_Z(\log D)\to \Omega^{\bullet+n-1}_{D_i}(\log (D-D_i)|_{D_i})\text{,  } \frac{dz_i}{z_i}\wedge\alpha\mapsto \alpha|_{D_i},
\]
where $z_i$ is the local defining equation of $D_i$ and $\frac{dz_i}{z_i}\wedge\alpha$ is a local section of $ \Omega^{\bullet+n}_Z(\log D)$. It factors through
\[
\Omega^{\bullet+n}_Z(\log D)|_{D_i}\to \Omega^{\bullet+n-1}_{D_i}(\log (D-D_i)|_{D_i}).
\]
By abuse of notations, we still call the above morphism $\Res_{D_i}$. Let $D^J=\cap_{j\in J} D^J$ and $D_J=\sum_{j\in J} D_j$. Then we have a collection of residue maps, by choosing an order on the indices and successively applying $\Res_{D_j}$ for $j\in J$, 
\[
\Res_{D^J}: \Omega^{\bullet+n}_Z(\log D) \to \Omega^{\bullet+\dim D^J}_{D^J}(\log (D-D_J)|_{D^J}).
\]

A \textit{log connection} $\nabla$ with poles along $D$ on a coherent $\sO_Z$-module $\cF$ is a $\C$-linear morphism $\nabla: \cF\to \Omega_Z(\log D)\otimes \cF$ satisfying the Leibniz rule $\nabla (f\cdot s)=df\otimes s+f\nabla s$ for $f$ local section of $\sO_Z$ and $s$ local section of $\cF$. One can extend standardly $\nabla$ to a log de Rham complex
\[
\left\{\cF \xrightarrow{\nabla}  \Omega_Z(\log D)\otimes\cF \xrightarrow{\nabla}  \cdots \xrightarrow{\nabla} \Omega^{n}_Z(\log D)\otimes \cF \right\}[n].
\]
If the above is a chain complex, i.e., $\nabla^2=0$ we say $(\cF,\nabla)$ is an \textit{integrable} log connection. We call the morphism $\Res_{D_i}\nabla: \cF\to \cF|_{D_i}$ induced by $\Res_{D_i}: \Omega_{Z}(\log D)\to \sO_{D_i}$ the \textit{residue} of the integrable log connection $\nabla$ along $D_i$. Note that $\Res_{D_i}$ is $\sO_Z$-linear and factors through again $\cF|_{D_i} \to \cF|_{D_i}$. 

Let $\sD_Z(\log D)$ be the sub-algebra of $\sD_Z$ generated locally by the differential operators $P$ such that $P\cdot \sI_D\subset\sI_D$. Here, we denote by $\sI_D$ the ideal sheaf of the normal crossing divisor $D$. An integrable log connection is then the same as a left $\sD_Z(\log D)$-module. We can extend the definition of residues of a log connection as follows. The sheaf $\sO_{D_i}=\sO_Z/\sI_{D_i}$ naturally has a left $\sD_Z(\log D)$-module structure because $\sI_{D_i}$ is also stable under by the $\sD_Z(\log D)$-action by the naive reason. Let $\cF^\bullet$ be a complex of integrable log connections. Then the complex 
\[
\cF^\bullet \tens{\sO_Z}{\mathbf L} \sO_{D_i}
\]
is a complex of $\sD_Z(\log D)$-modules because taking tensor products over $\sO_Z$ is closed in the category of $\sD_Z(\log D)$-modules and one can resolve either $\cF^\bullet$ or $\sO_{D_i}$ using locally $\sD_Z(\log D)$-free resolutions. The $\ell$-th cohomology $\sH^\ell(\cF^\bullet\otimes^{\mathbf L} \sO_{D_i})$ is indeed an $\sO_{D_i}$-module equipped with an integrable log connection. The residue of this log connection is $\sO_{D_i}$-linear and is called the $\ell$-th \textit{residue} of the complex $\cF^\bullet$.

As in the case of $\sD$-module, the sheaf $\omega_Z(\log D)=\Omega^{n}_Z(\log D)$ carries a canonical right $\sD_Z(\log D)$-module structure and we have the left to right transformation $\cF\mapsto \omega_Z(\log D)\otimes \cF$ for any left $\sD_Z(\log D)$-module $\cF$. 

\begin{ex}\label{starex}
We will use the following fact: the complex of right $\sD_Z$-modules
\[
\{\sD_Z \to \Omega_Z(\log D)\otimes\sD_Z\to \cdots \to \Omega^{n}_Z(\log D)\otimes\sD_Z\}[n]
\]
is a filtered resolution of $\omega_Z(* D)=\cup_{k\in \Z}\omega_Z(kD)$, equipped the induced filtration by $\Omega^{n+\bullet}_Z(\log D)\otimes F_{\ell+n+\bullet}\sD_Z$. In fact, it is well-known that the inclusion $\Omega^{n+\bullet}_Z(\log D)\to \Omega^{n+\bullet}_Z(*D)$ is a filtered quasi-isomorphism~\cite{Hodge2}. The inclusion extends to a filtered quasi-isomorphism $\Omega^{n+\bullet}_Z(\log D)\otimes\sD_Z \to \Omega^{n+\bullet}_Z(* D)\otimes \sD_Z$. Since $ \Omega^{n+\bullet}_Z(* D)\otimes \sD_Z$ is a filtered resolution of $\omega_Z(*D)$, we conclude the proof.  It follows that, for $f:Z\to W$,
\[
f_\dag\omega_Z(*D)=\bR f_!\left(\omega_Z(*D)\otimes^{\mathbf L}_{\sD_Z}\sD_{Z\to W}\right)=\left(\bR f_!\Omega^{n+\bullet}_Z(\log D)\right)\otimes\sD_W.
\]
In particular, if $f$ is a closed embedding then $f_!=f_+$ is exact and $f_\dag=\sH^0f_\dag$, which means
\[
\{\sD_W\to f_+\Omega_Z(\log D)\otimes\sD_W \to\cdots \to f_+\Omega^{n}_Z(\log D)\otimes\sD_W\}[n]
\]
is a resolution of $f_\dag\omega_Z(* D)$. We put the induced filtration to make it a filtered resolution and denote by 
\[
f_\dag(\omega_Z(*D), F_\bullet\omega_Z(*D))=(f_\dag\omega_Z(*D), F_\bullet f_\dag\omega_Z(*D)),
\]
or for simplicity just $f_\dag \omega_Z(*D)$. 
\end{ex}

The $\sD_Z$-module looks like $\cL\otimes \sD_Z$ for $\cL$ is a $\sO_Z$-module is called \textit{induced} $\sD_Z$-module. For example, we have seen $\Omega^{n+\bullet}_Z\otimes\sD_Z$ and $\Omega^{n+\bullet}(\log D)_Z\otimes\sD_Z$ are complexes of induced $\sD_Z$-modules.

\subsection{Polarized Hodge-Lefschetz structures}\label{subsec:phl}
The goal of this subsection is to introduce polarized bigraded Hodge-Lefschetz structures. Our treatment is an adjustment of~\cite[\S4]{hl}. The prototype of polarized Hodge-Lefschetz structures one should keep in mind is the graded vector space given by the direct sum of the cohomology groups of a compact K\"ahler manifold. Polarized bigraded Hodge-Lefschetz structures are the degenerations of polarized Hodge-Lefschetz structures. We begin with the convention on Hodge structures and we only consider complex Hodge structures. 

A \textit{Hodge structure} of \textit{weight} $n$ is a finite dimensional vector space $V$ with two decreasing filtrations $F^\bullet$ and $G^\bullet$ satisfying 
\[
V=F^p\oplus G^{n+1-p}, 
\]
for each $p\in\Z$. The definition is equivalent to
\[
V=\bigoplus_{p+q=n} V^{p,q}
\]
if we let $V^{p.q}=F^p\cap G^q$ for $p+q=n$. A morphism of Hodge structures is just a morphism of vector spaces respecting the two filtrations $F$ and $G$. A \textit{polarization} on the Hodge structure $(V,F^\bullet,G^\bullet)$ is a non-degenerated hermitian pairing $S:V\otimes_\C \overline V\to \C$ such that
\begin{enumerate}
\item $F^p$ is orthogonal to $G^{n+1-p}$ with respect to $S$ for every $p\in\Z$;
\item $(-1)^qS(-,-)$ is hermitian inner product on $V^{p,q}$. 
\end{enumerate}

\begin{rem}
A polarized Hodge structure of weight $n$ is completely determined by the triple $(V,F_\bullet V, S)$ because 
\[
G^{n+1-p}V=\left\{a\in V: S(a,\overline{b})=0\text{ for all } b \text{ in } {F^pV}\right\}=\overline{F^pV}^{\perp_S}.
\]
We will also call the triple $(V,F_\bullet V,S)$ a polarized Hodge structure.
\end{rem}

A \textit{Tate twist} $(V,F^\bullet,S)(r)$ on a polarized Hodge structure $(V,F^\bullet,S)$ is the triple $(V,F^{\bullet+r},(-1)^rS)$, for any integer $r$.

Now let us move on to the geometric case. It is well-known that the $k$-th cohomology group of a compact K\"ahler manifold $Z$ of dimension $n$ has Hodge decomposition:
\[
H^k(Z,\C)=\bigoplus_{p+q=k}H^{p,q}(Z)
\] 
and thus it is a Hodge structure of weight $k$. Fix a choice of $\sqrt{-1}$. Let $h$ be any K\"ahler metric on $Z$. We denote the K\"ahler form by $\omega=-\mathrm{Im}\, h\in A^2(Z,\R)$ and denote  its cohomology class by $[\omega]\in H^2(Z,\R)$; note that this depends on the choice of $\sqrt{-1}$ through the function $\mathrm{Im}:\C\to \R$. The choice of $\sqrt{-1}$ endows the two-dimensional real vector space $\C$ with an orientation on $Z$. The induced orientation on $Z$ has the property that 
\[
\int_Z \frac{\omega^n}{n!}=\mathrm{vol}(Z)>0.
\]
The integral also depends on the orientation, hence on the choice of $\sqrt{-1}$. To remove the dependence, instead of the usual integral, we should use 
\[
\frac{1}{(2\pi \sqrt{-1})^n}\int_Z: A^{2n}(Z,\C)\to \C.
\]
Let ${L=[w]\wedge}$ be the Lefschetz operator for a K\"ahler class $[w]$. Then for $k\leq n$ the primitive part 
\[
P_LH^k(Z,\C)\colon =\ker L^{n-k}\cap H^k(X,\C)
\]
is a polarized Hodge structure of weight $k$ with the polarization 
\[
S(a, \overline{b})={\frac{\varepsilon(k+1)}{(2\pi\sqrt{-1})^{n}}}\int_Z (2\pi\sqrt{-1}L)^{n-k}a\wedge \overline{b},
\]
for $a,b\in P_LH^k(Z,\C)$ because of the Hodge-Riemann bilinear relation. 

If we consider the cohomology groups altogether, we will get the Hodge-Lefschetz structure of central weight $n$. Denote by $(\mathsf{X,Y,H})$ the $\fsl_2(\C)$-triple, i.e.,
\[
[\sfX,\sfY]=\sfH, \quad [\sfH,\sfX]=2\sfX,\quad [\sfH,\sfY]=-2\sfY.
\] 
In the Lie group $\mathrm{SL_2}(\C)$, we have the Weil element $\mathsf{w}=e^\mathsf{X}e^{-\mathsf{Y}}e^{\mathsf{X}}$  with the property that $\sfw^{-1}=-\sfw,$ and under the adjoint action of $\mathrm{SL}_{2}(\mathbb{\C})$ on its Lie algebra, one has the identities
\[
\mathsf {w H w^{-1}}=\mathsf{-H}, \quad \mathsf{w X w^{-1}=-Y}, \quad \mathsf{w Y w^{-1}=-X}
\]
From this, one deduces that $e^{\mathsf X}=\mathsf{w} e^{-\mathsf X} e^{\mathsf Y}=e^{\mathsf Y} \mathsf w e^{\mathsf Y}$. Now $A^{\bullet}(Z)$ becomes a representation of $\mathfrak{s l}_{2}(\mathbb{\C})$ if we set
\[
{\sfX=2 \pi \sqrt{-1} L} \quad \text { and } \quad \sfY=(2 \pi \sqrt{-1})^{-1} \Lambda
\]
and let $\sfH$ act as multiplication by $k-n$ on the subspace $A^{k}(Z) .$ The reason for this (non-standard) definition is that it makes the representation not depend on the choice of $\sqrt{-1}$. It is easy to see how $\sfw$ acts on primitive forms. Suppose that $\alpha \in A^{n-k}(Z)$ satisfies $\sfY \alpha=0 .$ Then $\sfw \alpha \in A^{n+k}(Z) .$ If we now expand both sides of the identity
\[
e^{\sfX} \alpha=e^{\sfY} \sfw e^{\sfY} \alpha=e^{\sfY} \sfw \alpha
\]
into power series, and then compare terms in degree $n+k,$ we get
\[
\sfw \alpha=\frac{\sfX^{k}}{k !} \alpha.
\]
This formula is the reason for using $\sfw$ (instead of the otherwise $\sfw^{-1}$):
there is no sign on the right-hand side. 

A Hodge-Lefschetz structure is linear algebra data encoding both representation theoretic and Hodge theoretic information. Recall that a finite-dimensional $\fsl_2(\C)$-representation is a graded vector space $V=\bigoplus_{\ell\in\Z}V_{\ell}$ with a set of endomorphisms $(\sfH, \sfX,\sfY)$ such that 
\begin{enumerate}
  \item each graded piece $V_{\ell}$ is the $\ell$-eigenspace of $\sfH$;
  \item the morphism $\sfX^{\ell}: V_{-\ell}\to V_{\ell}$ is an isomorphism for each $\ell\geq 0$;
  \item the morphism $\sfY^{\ell} :V_{\ell}\to V_{-\ell}$ is an isomorphism for each $\ell\geq 0$.
\end{enumerate}

\begin{ex}\label{defwtf}
For any finite-dimensional vector space $V$ together with a nilpotent operator $N$, there exists a so-called monodromy filtration $W_\bullet$ uniquely determined by the following two conditions
\begin{itemize}
  \item for each $\ell\in\Z$, $N: W_\ell\to W_{\ell-2}$;
  \item the induced operator $N^\ell: \gr^W_{\ell} \to \gr^W_{-\ell}$ is an isomorphism for each $\ell\geq 0$.
\end{itemize}
Let $\gr^W=\bigoplus_{\ell\in\Z}\gr^W_\ell$. The $\ell$-th primitive part $P_N\gr^W_\ell=\ker N^{\ell+1}\cap \gr^W_\ell$ consists of the classes of generators of cyclic subspaces of $V$ of dimension $\ell$ as $\C[N]$-modules for $\ell\geq 0$. For each generator $v$, we have $N^{\ell+1}v=0$ but $N^\ell v\neq 0$ and also $v$ is not a image of $N$. Therefore, we have the identification 
\[
P_N\gr^W_\ell=\frac{\ker N^{\ell+1}}{\ker N^\ell +\mathrm{im\,}N\cap \ker N^{\ell+1}}.
\]
Furthermore, we have the Lefschetz decomposition $\gr^W_\ell=\bigoplus_{k\geq 0} N^kP_NV_{\ell+2k}$. Taking $N=\sfY$, the Lefschetz structure and the grading uniquely determines the operator $X$ such that $(\sfX,\sfY,\sfH)$ is a $\fsl_2(\C)$-triple by the relation $\sfX\sfY^{k}=k(\ell-k+1)\sfY^{k-1}$ on $P_N\gr^W_\ell$. Thus $\gr^W$ naturally is a representation of $\fsl_2(\C)$. 
\end{ex}

By Hard Lefschetz theorem, for any compact K\"ahler manifold $Z$ of $\dim Z=n$, the vector space $\bigoplus_{\ell\in \Z} H^{n+\ell}(Z,\C)$ is a representation of $\fsl_2(\C)$ by setting $\sfX=2\pi\sqrt{-1}L$ the Lefschetz operator, $\sfY=(2\pi\sqrt{-1})^{-1}\Lambda$ the adjoint operator. But because of the Lefschetz operator of is of type $(1,1)$, we actually have $\sfX: H^{k}(Z,\C)\to H^{k+1}(Z,\C)(1)$ is a morphism of Hodge structures and $\sfX^\ell:H^{n-\ell}(Z,\C)\to H^{n+\ell}(Z,\C)(\ell)$ is an isomorphism of Hodge structures. This leads to the following definition: a \textit{Hodge-Lefschetz} structure of \textit{central weight} $n$ is a $\fsl_2(\C)$-representation $V=\bigoplus_{\ell\in\Z} V_{\ell}$ with two filtrations $F^\bullet V=\bigoplus_{\ell} (F^\bullet V\cap V_\ell)$ and $G^\bullet V=\bigoplus_{\ell} (G^\bullet V\cap V_\ell)$ compatible with the grading such that
\begin{enumerate}
  \item each graded piece $(V_{\ell},F^\bullet V_{\ell},G^\bullet V_{\ell})$ is a Hodge structure of weight $n+\ell$;
  \item the operator $\sfX: (V_{\ell}, F^\bullet V_{\ell}, G^\bullet V_{\ell})\to (V_{\ell+2}, F^{\bullet+1} V_{\ell+2},,G^{\bullet+1} V_{\ell+2})$ is a morphism of Hodge structures such that 
  \[
  \sfX^\ell: (V_{-\ell},F^\bullet V_{-\ell}, G^\bullet V_{-\ell})\to (V_{\ell},F^\bullet V_{\ell},G^\bullet V_{\ell})(\ell)
  \]
  is an isomorphism of Hodge structures;
  \item the operator $\sfY:(V_{\ell}, F^\bullet V_{\ell}, G^\bullet V_{\ell})\to (V_{\ell-2}, F^{\bullet-1} V_{\ell-2}, G^{\bullet-1} V_{\ell-2})$ is a morphism of Hodge structures such that 
  \[
  \sfY^\ell: (V_{\ell},F^\bullet V_{\ell},G^\bullet V_{\ell}) \to (V_{-\ell},F^\bullet V_{-\ell},G^\bullet V_{-\ell})(-\ell)
  \]
   is an isomorphism of Hodge structures.
\end{enumerate}
It follows from the definition the primitive part $P_XV_{\ell}$ is a sub-Hodge structure for each $\ell<0$.  Let $V_\ell=H^{n+\ell}(Z,\C)$ and $V=\bigoplus_{\ell\in \Z} V_\ell$. It follows that $V$ is a Hodge-Lefschetz structure of central weight $n$. Hodge-Lefschetz structure interplays well with the Hodge-Riemann bilinear relation. A \textit{polarization} on a Hodge-Lefschetz structure $V$ of central weight $n$ is a hermitian symmetric paring $S: V\otimes_\C \overline V\to \C$ such that 
\begin{enumerate}
  \item the restriction $S|_{V_{\ell}\otimes_\C \overline{V_{-k}}}$ is zero for $\ell+k\neq 0$; 
  \item $S(\sfX-,-)=S(-,\bar\sfX-)$ and $S(-,\bar\sfY-)=S(\sfY-,-)$; 
  \item $S_\ell:V_\ell\otimes\overline{V_{-\ell}}\to \C$ is a polarization on $V_\ell$ where $S_\ell:V_\ell\otimes\overline{V_{-\ell}}\to \C$ is the restriction of $S$. 
\end{enumerate}

\begin{lem}\label{lem:weil}
If $V$ is a Hodge-Lefschetz structure, then $\sfw: V_{k} \to V_{-k}(-k)$ is an isomorphism of Hodge structures of weight $n+k$.
\end{lem}
\begin{proof} 
Suppose that $a \in V_{-\ell}$ is primitive, in the sense that $\sfX^{\ell+1} a=0$ and $\ell \geq 0$. Then $\sfY a=0,$ and from $\sfw e^{-\sfX}=e^{\sfX} e^{-\sfY},$ we get $\sfw e^{-\sfX} a=e^{\sfX} a,$ and after expanding and comparing terms in degree $\ell-2 j,$ also
\begin{equation}\label{weileq}
\sfw \frac{\sfX^{j}}{j !} a=(-1)^{j} \frac{\sfX^{\ell-j}}{(\ell-j) !} a
\end{equation}
since $\sfw^{2}$ acts on $V_{-\ell+2 j}$ as $(-1)^{-\ell+2 j}=(-1)^{\ell},$ this formula is actually symmetric in $j$ and $\ell-j$.

Any $a \in V_{k}$ has a unique Lefschetz decomposition
\[
a=\sum_{j \geq \max (k, 0)} \frac{\sfX^{j}}{j !} a_{j}
\]
where $a_{j} \in V_{k-2 j}$ satisfies $Y a_{j}=0 .$ We only need to consider $j \geq k$ in the sum because $\sfX^{2 j-k+1} a_{j}=0,$ which implies that $\sfX^{j} a_{j}=0$ for $j<k$. Assuming $a \in V_{k}^{p,q}$, where $p+q=n+k$, we can show that $a_{j} \in V_{k-2j}^{p-j,q-j}$ using descending induction on $j \geq \max (k, 0)$. If $a_{j} \in V_{k-2j}^{p-j,q-j}$ for $j > \ell$ for some $\ell \geq \max(k, 0)$, then we have
\[
\sfX^{\ell-k}a=\sum_{j \geq \max (k, 0)} \frac{\sfX^{j+\ell-k}}{j !} a_{j}=\sum_{j \geq \ell} \frac{\sfX^{j+\ell-k}}{j !} a_{j} \in V_{2\ell-k}^{p+\ell-k,q+\ell-k}
\]
This implies that $\sfX^{2\ell-k}a_\ell \in V_{2\ell-k}^{p+\ell-k,q+\ell-k}$, and therefore, $a_\ell \in V_{k-2\ell}^{p-\ell,q-\ell}$. In other words, the Lefschetz decomposition holds in the category of Hodge structures.

We can now check what happens when we apply $\sfw$. Using~\eqref{weileq}, we find that
\[
\sfw a=\sum_{j \geq \max (k, 0)} \sfw \frac{\sfX^{j}}{j !} a_{j}=\sum_{j \geq \max (k, 0)}(-1)^{j} \frac{\sfX^{j-k}}{(j-k) !} a_{j} \in V_{-k}^{p-k, q-k}
\]
and so $\sfw$ is a morphism of Hodge structures. The same calculation shows that $\sfw^{-1}$ is also a morphism of Hodge structures. It follows that $\sfw$ is an isomorphism of Hodge structures. 
\end{proof}

The following theorem implies that the definition of polarized Hodge-Lefschetz structure of central weight $n$ can be simplified. 
\begin{thm}\label{thm:defredu}
A polarized Hodge-Lefschetz structure of central weight $n$ on a filtered graded vector space $(V=\bigoplus_{k}V_{k}, F^\bullet V)$ is unique determined by the following:
\begin{enumerate}[leftmargin=50 pt]
  \item [$(\mathrm{pHL1})$] for each $\ell\geq 0$, we have an operator  $\sfX: (V_{\ell},F^\bullet)\to (V_{\ell+2},F^{\bullet+1})$ such that $\sfX^\ell: F^\bullet V_{-\ell}\to F^{\bullet+\ell}V_{\ell}$ is an isomorphism;
  \item [$(\mathrm{pHL2})$] a collection of Hermitian pairings $S_\ell\colon V_\ell\otimes \overline{V_{-\ell}}$ for $\ell \in \Z$ such that 
  \[
  S_{\ell}(\sfX-,-)=S_{\ell+2}(-, \bar\sfX-);
  \]
  \item [$(\mathrm{pHL3})$] for $j\geq 0$, the triple $\left(P_{\sfX}V_{-j}, F^\bullet, S_{j}\circ\left(\sfX^j\otimes \id\right)\right)$ is a polarized Hodge structure of weight $n-j$ where $F^\bullet P_{\sfX}V_{-j}=\ker \sfX^j\cap F^\bullet V_{-j}$.
\end{enumerate}
\end{thm}
\begin{proof}
The proof is essentially a linear algebra exercise, and we will only provide a sketch of the idea. Condition $(\mathrm{pHL1})$ indicates that the Lefschetz decomposition associated to $\sfX$ respects the filtration $F^\bullet$, using a similar approach as in the proof of Theorem~\ref{lem:weil}. Thus, $\sfY$ is uniquely determined and also respects the filtration. Condition $(\mathrm{pHL2})$ implies that $S_\ell(\sfY-,-)=S_{\ell-2}(-, \overline\sfY-)$. Condition $(\mathrm{pHL3})$ implies that $S=\bigoplus_\ell S_\ell$ is non-degenerate on $F^pV_{\ell}\otimes \overline{F^pV_{\ell}}$ by the Lefschetz decomposition associated to $\sfX$ on $F^pV$. Therefore, we can define the filtration $G^qV_\ell$ as follows:
\[
\begin{aligned}
G^qV_\ell &=\{a\in V_{\ell}: S_\ell(a, \bar b)=0\text{ for all }b\in F^{n-q+1}V_{-\ell}\} \quad \text{for} \quad q\in\Z.
\end{aligned}
\]
It follows that $V^{p,q}_\ell=F^pV\ell\cap G^qV_\ell=\{a\in F^pV_\ell: S_\ell(a, \bar b)=0\text{ for all }b\in F^{p-\ell+1}V_{-\ell}\}$ for $p+q=n+\ell$ gives us the Hodge decomposition $V_\ell=\bigoplus_{p+q=n+\ell} V^{p,q}_\ell$.

Finally, we need to show that $(-1)^qS_\ell\circ (\id \otimes \bar\sfw)$ is Hermitian positive on $V_\ell^{p,q}$ by condition $(\mathrm{pHL3})$. Let $a\in V^{p,q}_\ell$ have the Lefschetz decomposition
\[
a=\sum_{j \geq \max (\ell, 0)} \frac{\sfX^{j}}{j !} a_{j}
\]
where $a_{j} \in P_\sfX V_{\ell-2 j}^{p-j,q-j}$. Then,
\[
(-1)^qS_\ell (a,  \overline{\sfw a})=\sum_{j\geq \max(\ell,0)} (-1)^{q-j} S_{\ell-2j}\left(\frac{\sfX^{2j-\ell}}{j!}a_j, \frac{1}{(j-k)!}\overline{a_j}\right)\geq 0
\]
and the equality holds if and only if $a_j=0$ for all $j\geq \max (\ell,0)$. 
\end{proof}

\begin{ex} \label{sign}
Let $Z$ be a compact K\"ahler manifold of dimension $n$, and define $V_\ell = H^{n+\ell}(Z,\C)$ and $V = \bigoplus_{\ell\in \Z} V_\ell$. Then, $V$ together with $\sfX=2\pi\sqrt{-1}L$ and $\sfY=(2\pi\sqrt{-1})^{-1}\Lambda$, along with the natural filtration, forms a Hodge-Lefschetz structure of central weight $n$. By the Hodge-Riemann bilinear relation, the sesquilinear pairing
\begin{equation}\label{eq:bphlp}
S_{\ell}(a, \bar b)=\frac{\varepsilon(n- \ell+1)}{(2\pi\sqrt{-1})^{n}}\int_Z a\wedge \bar b=\varepsilon(\ell)(-1)^{\ell n}\frac{\varepsilon(n+1)}{(2\pi\sqrt{-1})^{n}}\int_Z a\wedge \bar b
\end{equation}
for $a\in V_\ell$ and $b\in V_{-\ell}$ gives a polarization on $V$. The point of this example is to show that the polarized Hodge-Lefschetz structure $V$ is determined by the filtered $\sD_Z$-module $\omega_Z$ together with the sesquilinear pairing $S_Z$ as defined in~\eqref{np}. The graded piece $V_\ell$ is $\bH^\ell(Z,\DR_Z\omega_Z)$, with the induced filtration $F^\bullet V_\ell$ given by the image of $\bH^\ell(Z,F_{-\bullet}\DR_Z\omega_Z)$. Now we work out the relation between $S_k$ and $S_Z$ explicitly. First, the sesquilinear pairing $S_Z$ induces a map:
\[
\begin{aligned}
  S'_k\colon \bH^k(Z,\DR_Z\omega_Z)\otimes \overline{\bH^{-k}(Z,\DR_Z\omega_Z)} \to \bH^0(Z,\DR_{Z,\overline Z} \omega_Z\otimes_\C\overline{\omega_Z}) \\
  \xrightarrow{S_Z}  \bH^0(Z,\DR_{Z,\overline Z}\mathfrak C_Z)\simeq \C,
\end{aligned}
\]
and we have to make $S'_k$ clear. It is more convenient for us to use distribution in the situation by the following  commutative diagram
\[
\begin{tikzcd}
\DR_{Z,\overline Z} \omega_Z\otimes_\C \overline {\omega_Z} \arrow{r} \arrow{d}{S} & \DR_{Z,\overline Z} \sO_Z\otimes_\C \overline{\sO_Z} \arrow{d}{D} \\
\DR_{Z,\overline Z} \mathfrak{C}_Z \arrow{r} & \DR_{Z,\overline Z} \mathfrak{Db}_Z 
\end{tikzcd}
\]
where the upper horizontal arrow is the isomorphism induced by~\eqref{drlr} and similarly the lower horizontal arrow is defined on the terms in degree $-k,$
\[
\mathfrak{C}_{Z} \otimes_{\mathscr{O}_{Z, \bar{Z}}} \bigwedge^{k} \mathscr{T}_{Z, \bar{Z}} \to \Omega_{Z, \bar{Z}}^{2 n-k} \otimes_{\mathscr{O}_{Z, \bar{Z}}} \mathfrak{D} \mathfrak{b}_{Z}
\]
by the following rule: write a current locally as $D \omega \wedge \bar{\omega},$ with a distribution $D$ and denote by $\partial_J=\bigwedge_{J}\partial_j$ and $dx_{\bar J}=\bigwedge_{i\notin J} dx_i$ for an ordered index subset $J$ of $I$; then
\begin{equation}\label{cdlr}
(D \omega \wedge \bar{\omega}) \otimes \partial_{J} \wedge \bar{\partial}_{K} \mapsto(-1)^{\left(j_{1}+\cdots+j_{p}\right)+\left(k_{1}+\cdots+k_{q}\right)}(-1)^{n q} d x_{\bar J} \wedge \overline{d x}_{\bar K} \otimes D
\end{equation}
where $|J|=p$ and $|K|=q,$ and $p+q=k .$ The sign factor is explained by the number of swaps that are needed to move everything into the right place, which is $\left(2 n-j_{1}\right)+\cdots+\left(2 n-j_{p}\right)+\left(n-k_{1}\right)+\cdots+\left(n-k_{q}\right) $.
We can now derive a formula for the induced pairing
\begin{equation}\label{holodb}
\DR_Z\sO_Z  \otimes_{\mathbb{C}} \overline{\DR_Z\sO_Z } \to \mathrm{DR}_{Z, \bar{Z}}\mathfrak{D} \mathfrak{b}_{Z}.
\end{equation}
For the two local sections $\alpha=d x_{\bar J}$ and $\beta=d x_{\bar K}$, under the isomorphism $\DR_Z\sO_Z  \cong \DR_Z\omega_Z$ in~\eqref{drlr}, the $(n-p)$-form
\[
\alpha \mapsto (-1)^{n p}(-1)^{j_{1}+\cdots+j_{p}} \cdot \omega \otimes \partial_{J}.
\]
and the $(n-q)$-form 
\[
\beta \mapsto (-1)^{n q}(-1)^{k_{1}+\cdots+k_{q}} \cdot \omega \otimes \partial_{K}.
\]
The pairing $\DR S_Z$ on $\DR_Z\omega_Z$ takes those two sections to
\begin{equation}\label{eq:secpair}
(-1)^{n(p+q)}(-1)^{\left(j_{1}+\cdots+j_{p}\right)+\left(k_{1}+\cdots+k_{q}\right)} S_Z(\omega, \omega) \otimes \partial_{J} \wedge \bar{\partial}_{K}.
\end{equation}
Now $S_Z(\omega, \omega)=D_Z \omega \wedge \bar{\omega},$ where $D_Z$ is the distribution
\[
D_Z=\frac{\varepsilon(n+1)}{(2 \pi \sqrt{-1})^{n}} \int_{Z}
\]
Under the isomorphism in~\eqref{cdlr} the section~\eqref{eq:secpair} therefore goes to
\[
(-1)^{n p} d x_{\bar J} \wedge \overline{d x}_{\bar K} \otimes D_Z=(-1)^{n(\operatorname{deg} \alpha-n)} \alpha \wedge \bar{\beta} \otimes D_Z
\]
The formula we have just derived also works for smooth forms, of course. In other words, the same formula can be used to extend~\eqref{holodb} to a pairing on the de Rham complex of smooth forms. The resulting pairings on cohomology are
\begin{equation}\label{drS}
 S'_k\colon H^{n+k}(Z, \mathbb{C}) \otimes \overline{H^{n-k}(Z, \mathbb{C})} \to \mathbb{C}, \quad (\alpha, \beta) \mapsto(-1)^{n(\operatorname{deg} \alpha-n)} \frac{\varepsilon(n+1)}{(2 \pi \sqrt{-1})^{n}} \int_{Z} \alpha \wedge \bar{\beta}.
\end{equation}
Comparing with the pairing~\eqref{eq:bphlp}, we see that the polarization $S_k=\varepsilon(k)S'_k$.
\end{ex}

 \subsection{Polarized bigraded Hodge-Lefschetz structures}\label{subsec:pbhl}
Heuristically, what we really consider is the degeneration of the ``variation of Hodge-Lefschetz structures'' of a family of compact K\"ahler manifolds and as it turns out the limit of the degeneration is a bigraded Hodge-Lefschetz structure. We begin to define polarized bigraded Hodge-Lefschetz structures, which is an adjustment of~\cite[\S 4]{hl}. 
Similarly to the case of $\fsl_2(\C)$-representation, a $\fsl_2(\C)\times\fsl_2(\C)$-representation is a bigraded vector space $V=\bigoplus_{\ell,k\in\Z} V_{\ell,k}$ with two sets of endomorphisms $(\sfX_1, \sfY_1, \sfH_1)$ and $(\sfX_2,\sfY_2,\sfH_2)$  satisfying the following conditions: 
\begin{enumerate}
  \item any member in the set $(\sfX_1, \sfY_1, \sfH_1)$ commutes with any member in the set $(\sfX_2,\sfY_2,\sfH_2)$;
  \item each bigraded piece $V_{\ell,k}$ is the $\ell$-th eigenspace of $\sfH_1$ and $k$-th eigenspace of $\sfH_2$;
  \item we have $\sfX_1:V_{\ell,k}\to V_{\ell+2,k}$ and $\sfX_2:V_{\ell,k}\to V_{\ell,k+2}$ for each $\ell,k\in\Z$ and  
  \[
  \sfX^{\ell}_1: V_{-\ell,k}\to V_{\ell,k} \quad \text{and} \quad \sfX^k_2 :V_{\ell,-k}\to V_{\ell,k}
  \]  
  are isomorphisms for $\ell,k\geq 0$;
  \item we have $\sfY_1: V_{\ell,k} \to V_{\ell-2,k}$ and  $\sfY_2: V_{\ell,k} \to V_{\ell,k-2}$ for each $\ell,k\in\Z$ and
  \[
  \sfY^{\ell}_1: V_{\ell,k}\to V_{-\ell,k} \quad \text{and}\quad \sfY^k_2: V_{\ell,k}\to V_{\ell,-k}
  \]
  are isomorphisms for $\ell,k\geq 0$.
\end{enumerate}

A \textit{bigraded Hodge-Lefschetz structure} of \textit{central weight} $n$ is a $\fsl_2(\C)\times \fsl_2(\C)$-representation $V=\bigoplus_{\ell,k\in\Z} V_{\ell,k}$ with two filtrations $F^\bullet V=\bigoplus_{\ell,k}F^\bullet V\cap V_{\ell,k}$ and $G^\bullet V=\bigoplus_{\ell,k}G^\bullet V\cap V_{\ell,k}$ compatible with the bigrading such that 
\begin{enumerate}
  \item the bifiltered vector space $(V_{\ell,k}, F^\bullet V_{\ell,k},G^\bullet V_{\ell,k})$ is a Hodge structure of weight $n+\ell+k$;
  \item   the two operators 
  \[
  \begin{aligned}
    &\sfX_1: (V_{\ell,k}, F^\bullet, G^\bullet)\to (V_{\ell+2,k}, F^{\bullet+1},G^{\bullet+1})  \quad \text{and}\\
    &\sfX_2: (V_{\ell,k},F^\bullet,G^\bullet )\to (V_{\ell,k+2},F^{\bullet+1} ,G^{\bullet+1})
  \end{aligned}
  \]
  are morphisms of Hodge structures such that 
  \[
  \begin{aligned}
     &\sfX^{\ell}_1: (V_{-\ell,k},F^\bullet ,G^\bullet)\to (V_{\ell,k},F^{\bullet},G^{\bullet})(\ell) \quad \text{and}  \\
     &\sfX^{k}_2: (V_{\ell,-k},F^\bullet ,G^\bullet)\to (V_{\ell,k},F^{\bullet},G^{\bullet})(k) 
  \end{aligned}
  \]
  are isomorphisms of Hodge structures for $\ell,k\geq 0$.
  \item the two operators 
  \[
  \begin{aligned}
    &\sfY_1: (V_{\ell,k},F^\bullet,G^\bullet)\to (V_{\ell-2,k},F^{\bullet-1},G^{\bullet-1}) \quad \text{and}  \\
    &\sfY_2: (V_{\ell,k},F^\bullet,G^\bullet)\to (V_{\ell,k-2},F^{\bullet-1},G^{\bullet-1})
  \end{aligned}
  \]
   are morphisms of Hodge structures such that 
  \[
  \begin{aligned}
    &\sfY^{\ell}_1: (V_{\ell,k},F^\bullet,G^\bullet)\to (V_{-\ell,k},F^{\bullet},G^{\bullet})(-\ell) \quad \text{and} \\
    &\sfY^k_2: (V_{\ell,k},F^\bullet,G^\bullet)\to (V_{\ell,-k},F^{\bullet},G^{\bullet})(-k)
  \end{aligned}
  \]
  are isomorphisms of Hodge structures for $\ell,k\geq 0$.
\end{enumerate}

A \textit{polarization} on a bigraded Hodge-Lefschetz structure $V=\bigoplus_{\ell,k\in\Z} V_{\ell,k}$ of central weight $n$ is a Hermitian symmetric pairing $S:V\otimes_\C\overline V\to\C$ such that
\begin{enumerate}
  \item the restriction $S|_{V_{\ell,k}\otimes_\C \overline{V_{i,j}}}\colon V_{\ell,k}\otimes_\C \overline{V_{i,j}} \to \C$ vanishes unless for $\ell=-i$ and $k=-j$; 
  \item $S(\sfX_1-,-)=S(-, \overline{\sfX_1}-)$ and $S(-,\overline{\sfY_2}-)=S(\sfY_2-,-)$;
  \item $S_{\ell,k}(-,\overline{\sfw_1\sfw_2}-)$ is a polarization on $V_{\ell,k}$, where $S_{\ell,k}$ is the restriction of $S$ on $V_{\ell,k}\otimes \overline{V_{-\ell,k}}$ and $\sfw_i=e^{\sfX_i}e^{-\sfY_i}e^{\sfX_i}$ for $i=1,2$.
\end{enumerate} 

This is the practical definition because in the later application, $\sfX_1$ will be the $2\pi\sqrt{-1}L$ and $\sfY_2$ will be, up to a scalar, the logarithmic of the monodromy for the degeneration. Similarly to the case of Hodge-Lefschetz structure Theorem~\ref{thm:defredu}, we have a simpler definition.
\begin{thm}\label{bphl}
A polarized bigraded Hodge-Lefschetz structure of central weight $n$ on a filtered bigraded vector space $(V=\bigoplus_{\ell,k}V_{\ell,k}, F^\bullet V)$ is uniquely determined by the following:
\begin{enumerate}[leftmargin=58 pt]
  \item [$(\mathrm{pbHL1})$] there are two commuting operators $\sfX_1:(V_{\ell,k},F^\bullet)\to (V_{\ell+2,k}, F^{\bullet+1})$ and $\sfY_2:(V_{\ell,k}, F^\bullet)\to (V_{\ell,k-2},F^{\bullet-1})$ such that, for every $\ell,k\in\Z$,  
  \[
  \sfX^\ell_1: F^{\bullet}V_{-\ell,k} \to F^{\bullet+\ell}V_{\ell,k} \quad \text{and} \quad \sfY^k_2: F^{\bullet}V_{\ell,k}\to F^{\bullet-k}V_{\ell,-k}\text{ are isomorphisms};
  \]
  \item [$(\mathrm{pbHL2})$] a collection of Hermitian pairings $S_{\ell,k}: V_{\ell,k}\otimes_\C \overline{V_{-\ell,-k}} \to \C$ such that 
   \[
   S_{\ell,k}(\sfX_1-,-)=S_{\ell+2,k}(-,\overline{\sfX_1}-) \quad \text{and}\quad S_{\ell,k}(-,\overline{\sfY_2}-)=S_{\ell,k-2}(\sfY_2-,-);
   \]
  \item [$(\mathrm{pbHL3})$] the triple $\left(P_{-\ell,k},F^\bullet P_{-\ell,k},S_{\ell,-k}\circ \left(\sfX^\ell_1\otimes \left(-\overline{\sfY_2}\right)^k\right)\right)$ is a polarized Hodge structure of weight $n-\ell+k$ where $F^\bullet P_{-\ell,k}=\ker \sfX^\ell_1\cap\ker \sfY^k_2\cap F^\bullet V_{-\ell,k}$ is the bi-primitive part.
\end{enumerate}
Then the  $G^q V_{j,k}$ can be described as 
\[
\begin{aligned}
G^qV_{j,k} &=\{a \in V_{j,k}\colon S_{j,k}(a,\bar b)=0\text{ for all } b\in F^{n-q+1}V_{-j,-k}\} \quad \text{for} \quad q\in\Z
\end{aligned} 
\]
and 
\[
V^{p,q}_{j,k} =\{a\in F^pV_{j,k}\colon S_{j,k}(a, \bar b)=0\text{ for all } b\in F^{p-j-k+1}V_{-j,-k}\}, 
 \quad \text{for} \quad p+q=n+j+k.
\]
\end{thm}

The proof is similar to that of Theorem~\ref{thm:defredu} and is left to the readers. Later when we construct the limiting mixed Hodge structure, the polarized bigraded Hodge-Lefschetz structure naturally comes up from the first page of the weight spectral sequence associated to a mixed Hodge complex. Modeled on the properties of the differential of spectral sequence we give the following definition:

A \textit{differential} of a polarized bigraded Hodge Lefschetz structure $(V,F^\bullet,\sfX_1,\sfY_2,S)$ is a linear map $d:V\to V$ such that 
\begin{enumerate}
  \item $d:(V_{j,k}, F^\bullet)\to (V_{j+1,k-1},F^\bullet)$ and $d^2=0$;
  \item $d$ is {skew-symmetric} with respect to $S$, i.e., $S(d-,-)+S(-,\bar d-)=0$;
  \item $[\sfX_1,d]=0$ and $[\sfY_2,d]=0$.
\end{enumerate}
\begin{rem}\label{dmorphism}
In fact, the above three conditions imply that $d$ is a morphism of Hodge structures $d: V^{p,q}_{j,k}\to V^{p,q}_{j+1,k-1}$. A vector $a\in G^qV_{j,k}$ means that $S(a,\bar b)=0$ for all $b\in F^{n-q+1}V_{-j,-k}$. Then $S(da,b)=S(a,db)=0$ for all $b\in F^{n-q+1}V_{-j-1,-k+1}$, indicating $da$ belongs to $G^qV_{j+1,k-1}$.
\end{rem}

The main result of this subsection is the following version of Deligne's lemma, shown by Guill\'en and Navarro Aznar.

\begin{thm}[{\cite[(4.5)]{hl}}] \label{GA90}
The cohomology $\ker d/\mathrm{im}\, d$ of a polarized differential bigraded Hodge-Lefschetz structure is again a polarized bigraded Hodge-Lefschetz structure.
\end{thm}
\begin{proof}
Let $C: V \to V$ be the operator that acts as $(-1)^{q}$ on the subspace $V_{j, k}^{p, q}$ in the Hodge decomposition of each $V_{j, k} .$ Since $d$ is a morphism of Hodge structures, we have $[d, C]=0$. The fact that $S$ is a polarization means that the Hermitian pairing
\[
h^{+}: V \otimes_{\mathbb{C}} \overline{V} \to \mathbb{C}, \quad h^+(a, \bar b)=S\left(C a, \overline{\sfw_{1} \sfw_{2} b}\right)
\]
is positive-definite on $V$. Let $d^{*}$ be the adjoint of $d$ with respect to $h^{+}$. Fix $a \in V_{j, k}$ and $b \in V_{j, k},$:
\[
\begin{aligned}
h^{+}(d a, \bar b) &=S\left(C d a, \overline{\sfw_{1} \sfw_{2} b}\right)=S\left(d C a, \overline{\sfw_{1} \sfw_{2} b } \right) \\
&=-S\left(C a, \overline{d \sfw_{1} \sfw_{2} b}\right)=-S\left(C a, \overline{\sfw_{1} \sfw_{2} \cdot \sfw_{2}^{-1} \sfw_{1}^{-1} d \sfw_{1} \sfw_{2} \cdot b} \right) =h^{+}\left(a, \overline{d^{*} b } \right),
\end{aligned}
\]
i.e. the adjoint $d^*=-\sfw_{2}^{-1} \sfw_{1}^{-1} d \sfw_{1} \sfw_{2}$. 

In addition to the two relations in the definition of differential
\[
\left[\sfX_{1}, d\right]=0 \quad \text { and } \quad\left[\sfY_{2}, d\right]=0
\]
we obtain from the grading another two relations
\[
\left[\sfH_{1}, d\right]=d \quad \text { and } \quad\left[\sfH_{2}, d\right]=-d.
\]
This means, with respect to the $\mathfrak{s l}_{2}(\mathbb{C}) \times \mathfrak{s l}_{2}(\mathbb{C})$-action on $\operatorname{End}_{\mathbb{C}}(V),$ the element $d$ has weight $(+1,-1),$ and is primitive with respect to the action by $\sfY_{1}$ and $\sfX_{2} .$ Define
\[
d_{1}=\left[\sfY_{1}, d\right] \quad \text { and } \quad d_{2}=-\left[\sfX_{2}, d\right].
\]
The reason for the minus sign is that we have $\left[\sfY_{2}, d\right]=0$. Then $d_{1}$ has weight $(-1,-1),$ and is primitive with respect to the action by $\sfX_{1}$ and $\sfX_{2}$, i.e. $\left[\sfY_{1}, d_{1}\right]=0$ and $\left[\sfY_{2}, d_{1}\right]=0$; this gives
\[
{\left[\sfH_{1}, d_{1}\right]=-d_{1},} \quad {\left[\sfH_{2}, d_{1}\right]=-d_{1},} \quad {\left[\sfX_{1},  d_{1}\right]=d}, \quad    \sfw_{1} d_{1} \sfw_{1}^{-1}=d.
\]
Similarly, $d_{2}$ has weight $(+1,+1),$ and therefore
\[
\begin{array}{ll}
{\left[\sfH_{2}, d_{2}\right]=d_{2},} & {\left[\sfX_{2}, d_{2}\right]=0, \quad\left[\sfY_{2}, d_{2}\right]=-d, \quad \sfw_{2} d_{2} \sfw_{2}^{-1}=d} \\
{\left[\sfH_{1}, d_{2}\right]=d_{2},} & {\left[\sfX_{1}, d_{2}\right]=0}.
\end{array}
\]
Therefore, $d^{*}=-\left[\sfY_{1}, d_{2}\right]=\left[\sfX_{2}, d_{1}\right] \in \operatorname{End}_{\mathbb{C}} V$. It has weight $(-1,+1),$ and is primitive with respect to $\sfX_{1}$ and $\sfY_{2}$ i.e $\left[\sfY_{1}, d^{*}\right]=0$ and $\left[\sfX_{2}, d^{*}\right]=0$. From this, and the identities we already have, we deduce the following set of relations:
\[
\begin{aligned}
&\left[\sfH_{1}, d^{*}\right]=-d^{*}, \quad & \left[\sfX_{1}, d^{*}\right]&=-d_{2}, \quad & \left[\sfY_{1}, d^{*}\right]=0, \quad & \sfw_{1} d^{*} \sfw_{1}^{-1}=-d_{2} \\
&\left[\sfH_{2}, d^{*}\right]=d^{*}, \quad & \left[\sfX_{2}, d^{*}\right]&=0, \quad & \left[\sfY_{2}, d^{*}\right]=d_{1}, \quad & \sfw_{2} d^{*} \sfw_{2}^{-1}=-d_{1}.
\end{aligned}
\] 

We can check that the (formal) Laplace operator
\[
\Delta=d d^{*}+d^{*} d \in \operatorname{End}_{\mathbb{C}}(V)
\]
is invariant under the action of $\mathfrak{s l}_{2}(\mathbb{C}) \times \mathfrak{s l}_{2}(\mathbb{C}).$ For example,
\[
\begin{array}{l}
{\left[\sfX_{1}, d d^{*}\right]=\sfX_{1} d d^{*}-d d^{*} \sfX_{1}=d \sfX_{1} d^{*}-d\left(\sfX_{1} d^{*}+d_{2}\right)=-d d_{2}} \\
{\left[\sfX_{1}, d^{*} d\right]=\sfX_{1} d^{*} d-d^{*} d \sfX_{1}=\left(d^{*} \sfX_{1}-d_{2}\right) d-d^{*} \sfX_{1} d=-d_{2} d}
\end{array}
\]
from which we conclude, using $d^{2}=0,$ that
\[
\left[\sfX_{1}, \Delta\right]=-\left(d d_{2}+d_{2} d)=-(d\left(d \sfX_{2}-\sfX_{2} d\right)+\left(d \sfX_{2}-\sfX_{2} d\right) d\right)=0
\]
The other three commutators can be checked similarly. On the other hand, $\Delta$ is also a morphism of Hodge structures: the reason is that
\[
d: V_{j, k} \to V_{j+1, k-1}, \quad \sfY_{1}: V_{j, k} \to V_{j-2, k}(-1), \quad \sfX_{2}: V_{j, k} \to V_{j, k+2}(1)
\]
are all morphisms of Hodge structures, and $\Delta$ is obtained by composing them in some order. It follows that $\ker\Delta \subseteq V$ is a bigraded Hodge-Lefschetz structure, polarized by the restriction of $S$. Because of the canonical isomorphism $\ker \Delta\simeq \ker d/\mathrm{im\,}d$ as bigraded Hodge-Lefschetz structures, the induced pairing by $S$ on $\ker d/\mathrm{im\,} d$ is also a polarization. This concludes the proof.
\end{proof}

\section{Relative log de Rham complex}\label{sec:logdr}
Let $f:X\to \Delta$ be a proper holomorphic morphism smooth away from the origin whose central fiber $Y$ is simple normal crossing but not necessarily reduced. Assume $X$ is K\"ahler of dimension $n+1$ and $Y=\sum_{i\in I}e_iY_i$ where $Y_i$'s are smooth components and $I$ a finite index set. Let $t$ be a parameter on $\Delta$ and $z_0,z_1,\dots,z_n$ a local coordinate system on $X$ such that $t=z^{e_0}_0z^{e_1}_1\cdots z^{e_k}_k$ such that $e_0,e_1,\dots,e_k\geq 1$. Then we have $\Omega_\Delta(\log 0)=\sO_\Delta \cdot\frac{dt}{t}$ and $\Omega_{X}(\log Y)$ is locally generated by 
\[
e_0\frac{dz_0}{z_0},e_1\frac{dz_1}{z_1},\dots,e_k\frac{dz_k}{z_k},dz_{k+1},dz_{k+2},\dots,dz_n
\]
over $\sO_X$. Denote by $\xi_0,\xi_1,\dots,\xi_n$ the classes of the above generators in $\Omega_{X/\Delta}(\log Y)$, respectively. As a quotient of $\Omega_{X}(\log Y)$, the sheaf $\Omega_{X/\Delta}(\log Y)$ is generated by $\xi_0,\xi_1,\dots,\xi_n$,  but under the relation $\xi_0+\xi_1+\cdots+\xi_n=0$ because
\[
f^*\frac{dt}{t}=e_0\frac{dz_0}{z_0}+e_1\frac{dz_1}{dz_1}+\cdots+e_k\frac{dz_k}{z_k}
\]
is in $f^*\Omega_\Delta(\log 0)$. Let $\sT_{X/\Delta}(\log Y)$ be the dual bundle of $\Omega_{X/\Delta}(\log Y)$. Then $\sT_{X/\Delta}(\log Y)$ is a subsheaf of $\sT_X$, generated by 
\begin{equation}
D_i=\left\{
\begin{aligned}
\frac{1}{e_i}z_i\partial_i-\frac{1}{e_0}z_0\partial_0 &, & 1\leq i \leq k\\
\partial_i &, &  i>k,
\end{aligned}
\right.
\end{equation}
where $\partial_i$ is the local section of $\sT_X$ dual to $dz_i$ in $\Omega_X$. It follows that $D_1,D_2,\dots,D_n$ is the dual frame of $\xi_1,\xi_2,\dots,\xi_n$. 

\subsection{A ``log connection''}
We shall construct an endomorphism of $\bR f_*\Omega^{\bullet+n}_{X/\Delta}(\log Y)$ in $\bD^b(\Delta,\C)$, which should be regarded a ``log connection''. The short exact sequence of $\sO_X$-modules
\[
0\to f^*\Omega_\Delta(\log 0)\otimes \Omega^{\bullet+n}_{X/\Delta}(\log Y) \to \Omega^{\bullet+n+1}_{X}(\log Y) \to \Omega^{\bullet+n+1}_{X/\Delta}(\log Y)\to 0,
\]
under the identification $\frac{dt}{t}\wedge:\sO_X \to f^*\Omega_{\Delta}(\log 0)$, can be expressed as
\[
\begin{tikzcd}[sep=small]
0\arrow{r} & \Omega^{\bullet+n}_{X/\Delta}(\log Y) \arrow{r}{\frac{dt}{t}\wedge} & \Omega^{\bullet+n+1}_{X}(\log Y) \arrow{r} & \Omega^{\bullet+n+1}_{X/\Delta}(\log Y) \arrow{r} & 0.
\end{tikzcd}
\]
Here, the morphism $\frac{dt}{t}\wedge:\Omega^k_{X/\Delta}(\log Y)\to \Omega^{k+1}_X(\log Y)$, $[\alpha]\mapsto \frac{dt}{t}\wedge \alpha$ which does not depend on the representative of $[\alpha]$. Let $\Cone^\bullet=\Omega^{\bullet+n}_X(\log Y)\oplus \Omega^{\bullet+n}_{X/\Delta}(\log Y)$ be the mapping cone of $\frac{dt}{t}\wedge:\Omega^{\bullet+n-1}_{X/\Delta}(\log Y)\to \Omega^{\bullet+n}_{X}(\log Y)$. In our convention, the differential $\delta$ of the mapping cone works as $\delta(\alpha,[\beta])=\left((-1)^{n}d\alpha+\frac{dt}{t}\wedge\beta,(-1)^nd[\beta]\right)$, where $d$ is the usual exterior derivative on $\lome{\bullet}$ and by abuse of notation, we also use $d$ to denote the induced differential on $\relome{\bullet}$. Then we have the following diagram:
\begin{equation}\label{res}
\begin{tikzcd}
\Cone^\bullet \arrow{r}{q} \arrow{d}{p} & \Omega^{\bullet+n}_{X/\Delta}(\log Y) \arrow[dotted]{dl}{p\circ q^{-1}} \\
\Omega^{\bullet+n}_{X/\Delta}(\log Y) &
\end{tikzcd}
\end{equation}
where $q:\Cone^\bullet\to\Omega^{\bullet+n}_{X/\Delta}(\log Y),(\alpha,[\beta])\mapsto [\alpha]$ is a quasi-isomorphism and $p$ is the second projection. Therefore we have the morphism $p\circ q^{-1}$ in $\End_{\bD^b(X,\C)}\left(\relome{n+\bullet}\right)$. The multiplication by $g$ is an endomorphism of $\Omega^{\bullet+n}_{X/\Delta}(\log Y)$ for any local section $g$ of $\sO_\Delta$, because the complex is $f^{-1}\sO_\Delta$-linear. 

\begin{lem}~\label{nabla}
The operator $\nabla=(-1)^{n-1} p\circ q^{-1}\in \End_{\bD^b(X,\C)}\left(\Omega^{\bullet+n}_{X/\Delta}(\log Y)\right)$ satisfies $[\nabla, g]=tg'$, where $g\in \sO_\Delta$ and $g'$ denotes the derivative of $g$.
\end{lem}
\begin{proof}
It is equivalent to show that $[p\circ q^{-1},g]=(-1)^{n-1}tg'$. Define $g(\alpha,[\beta])=(g\alpha,g[\beta]+(-1)^{n-1}tg'[\alpha])$ for any $(\alpha,[\beta])\in \Cone^\bullet$ and $g$ a local section of $f^{-1}\sO_\Delta$. We shall prove that $g$ is an endomorphism of $\Cone^\bullet$, i.e., $g \delta(\alpha,[\beta])=\delta g(\alpha,[\beta])$. This follows from that
\[
\begin{aligned}
g \delta(\alpha,[\beta]) &=g\left((-1)^{n}d\alpha+\frac{dt}{t}\wedge\beta,(-1)^{n}d[\beta]\right) \\
  &=\left((-1)^{n}gd\alpha+g\frac{dt}{t}\wedge\beta,(-1)^{n}gd[\beta]-tg'd[\alpha]\right)
\end{aligned}
\]
and
\[
\begin{aligned}
&\delta g(\alpha,[\beta]) \\
&=\delta\left(g\alpha, g[\beta]+(-1)^{n-1}tg'[\alpha]\right) \\
  &=\left((-1)^{n}d(g\alpha)+\frac{dt}{t}\wedge(g\beta+(-1)^{n-1}tg'\alpha),(-1)^nd(g[\beta]+(-1)^{n-1}tg'[\alpha])\right) \\
  &=\left((-1)^{n}gd\alpha+g\frac{dt}{t}\wedge\beta,(-1)^{n}gd[\beta]-tg'd[\alpha]\right).
\end{aligned}
\]
It is easy to see that $g\circ q=q\circ g$ and hence, $q^{-1}\circ g=g\circ q^{-1}$. Therefore,
\[
[p\circ q^{-1},g]=p\circ q^{-1}\circ g-g\circ p\circ q^{-1}=[p,g]\circ q^{-1}
\]
But $[p,g](\alpha,[\beta])=p(g\alpha,g[\beta]+(-1)^{n-1}tg'[\alpha])-g[\beta]=(-1)^{n-1}tg'[\alpha]$. It follows that 
\[
[p\circ q^{-1},g]\circ q (\alpha,[\beta])=[p,g](\alpha,[\beta])=(-1)^{n-1}tg' \circ q (\alpha,[\beta]).
\]
By inverting $q$ we prove the statement.
\end{proof}

Because of the identification $\frac{dt}{t}\wedge:\sO_\Delta\to \Omega_\Delta(\log 0)$, what we really get is a morphism in $\bD^b(X,\C)$
\[
\nabla: \relome{\bullet+n}\to f^*\Omega_\Delta(\log 0)\otimes \relome{\bullet+n}
\]
such that $\nabla g= g\nabla+\frac{dt}{t}\otimes tg'\in \End_{\bD^b(X,\C)}(\relome{\bullet+n})$ for any local section $g$ of $\sO_\Delta$. Running the similar construction, we obtain an induced $\C$-linear (indeed $f^{-1}\sO_\Delta$-linear) endomorphism $[\nabla]$ on $\Omega^{\bullet+n}_{X/\Delta}(\log Y)|_Y$ in $\bD^b(X,\C)$ satisfying the following diagram.
\[
\begin{tikzcd}
 \relome{\bullet+n} \arrow{d}{\nabla+1}\arrow{r}{t} &  \relome{\bullet+n} \arrow{d}{\nabla}\arrow{r} &  \relome{\bullet+n}|_Y \arrow{d}{[\nabla]} \arrow{r}{+1} &  {} \\
 \relome{\bullet+n}  \arrow{r}{t} &   \relome{\bullet+n} \arrow{r} &  \relome{\bullet+n}|_Y \arrow{r}{+1} &  {} 
\end{tikzcd}
\]
Since $\relome{\bullet+n}$ is $f^{-1}\sO_\Delta$-linear, each cohomology $\bR^k f_*\Omega^{\bullet+n}_{X/\Delta}(\log Y)$ is a coherent $\sO_\Delta$-module. Taking direct image, we get $\C$-linear morphisms between distinguished triangles in $\bD^b_{\mathrm{coh}}(\Delta,\sO_\Delta)$: 
\begin{equation}\label{natureres}
\begin{tikzcd}
 \bR f_*\relome{\bullet+n} \arrow{d}{\bR f_*\nabla+1}\arrow{r}{t} &  \bR f_*\relome{\bullet+n} \arrow{d}{\bR f_*\nabla}\arrow{r} &  \bR f_*\relome{\bullet+n}|_Y \arrow{d}{\bR f_*[\nabla]} \arrow{r}{+1} &  {} \\
 \bR f_*\relome{\bullet+n}  \arrow{r}{t} &   \bR f_*\relome{\bullet+n} \arrow{r} & \bR f_* \relome{\bullet+n}|_Y \arrow{r}{+1} & {}
\end{tikzcd}
\end{equation}
where the morphism
\[
\bR f_*\nabla: \bR f_*\Omega^{\bullet+n}_{X/\Delta}(\log Y)\to \bR f_*\Omega^{\bullet+n}_{X/\Delta}(\log Y)
\]
satisfies $[\bR f_*\nabla,g]=tg'\in \End_{\bD^b(\Delta,\C)}\left(\bR f_*\Omega^{\bullet+n}_{X/\Delta}(\log Y)\right)$ for any local section $g$ of $\sO_\Delta$. 

\subsection{Residue}
Heuristically, we should regard  $\bR f_*[\nabla]$ as the residue of $\bR f_*\nabla$. More generally, let $\cF^\bullet$ be a complex of $\sO_\Delta$-modules with a morphism $\nabla \in \End_{\bD^b(\Delta,\C)}(\cF^\bullet)$ such that $[\nabla,g]=tg'$ for any $g\in \sO_\Delta$.  
Let $\mathcal G^\bullet$ be the mapping cone of $t:\cF^\bullet\to \cF^\bullet$, which computes to $\cF^\bullet\otimes^{\mathbf L} \C(0)$. Then by the axioms of triangulated categories~\cite{HTT}, there exists an operator $R \in \End_{\bD^b(\Delta,\C)}(\mathcal G^\bullet)$ making the following diagram commute in $\bD^b(\Delta,\C)$.
\[
\begin{tikzcd}
\cF^\bullet \arrow{d}{\nabla+1} \arrow{r}{t} & \cF^\bullet \arrow{d}{\nabla}\arrow{r} & \mathcal G^\bullet=\cF^\bullet\otimes^{\mathbf L}\C(0) \arrow{d}{R} \arrow{r} & \cF^\bullet[1] \arrow{d}{\left(\nabla+1\right)[1]} \\
\cF^\bullet  \arrow{r}{t} &   \cF^\bullet \arrow{r} & \mathcal G^\bullet=\cF^\bullet \otimes^{\mathbf L}\C(0) \arrow{r} & \cF^\bullet[1] 
\end{tikzcd}
\]
We call the operator $R$ a \textit{residue} of $\nabla$. Note that the axioms of triangulated categories cannot guarantee that the filling is unique. However, the eigenvalues of $R_\ell$ only depends on $\nabla$, where $R_{\ell}$ denotes the induced operator on the cohomology $\sH^{\ell}\left(\cF^\bullet\otimes^{\mathbf L} \C(0)\right)$. First, every object in $\bD^b_{\mathrm{coh}}(\Delta,\sO)$ splits, meaning that $\cF^\bullet\simeq \bigoplus_{\ell\in\Z}\sH^\ell\cF^\bullet[-\ell]$,  since there are no $\Ext^i$ for  $i\geq 2$ between two coherent sheaves over a curve. It follows that the morphism $\nabla$ breaks up into sum of morphism consisting of diagonal morphism $\nabla_\ell:\sH^\ell\cF^\bullet[-\ell]\to \sH^\ell\cF^\bullet[-\ell]$ which is an actual log connection and off-diagonal morphism $\sH^\ell\cF^\bullet[-\ell]\to \sH^m\cF^\bullet[-m]$ but only for $\ell>m$. Thus the eigenvalues of $R_\ell$ are determined by $\nabla_\ell$ and $\nabla_{\ell+1}$. When $\cF^\bullet$ is a locally free sheaf centered at degree zero and $\nabla$ is the usual log connection. The above definition coincides with the usual definition of the residue of $\nabla$.

Returning to our case, the natural choice of a residue of $\bR f_*\nabla$ is $R=\bR f_*[\nabla]$ because of the diagram~\eqref{natureres}: by the projection formula, we have
\[
\begin{aligned}
  \bR f_*\Omega^{\bullet+n}_{X/\Delta}(\log Y)\tens{\sO_\Delta}{\mathbf L} \C(0) &= \bR f_*\left(\relome{\bullet+n}\tens{f^{-1}\sO_\Delta}{\mathbf L} f^{-1}\C(0)\right) \\
  &= \bR f_*\left(\relome{\bullet+n}|_Y\right).
\end{aligned}
\]
Our main result concerning the relative log de Rham complex is the following.

\begin{thm}\label{mainlogdr}
The higher direct image $\bR^{\ell}f_*\Omega^{\bullet+n}_{X/\Delta}(\log Y)$ is locally free for each $\ell\in\Z$. Moreover, there exists a canonical isomorphism for every $p\in \Delta$
\[
\bR^{\ell}f_*\Omega^{\bullet+n}_{X/\Delta}(\log Y)\otimes \C(p) \simeq \bH^{\ell}(X,\Omega^{\bullet+n}_{X/\Delta}(\log Y)|_{X_p}), 
\]
where $\C(p)$ is the residue field at $p\in \Delta$.
\end{thm}
We first present two preliminary theorems.

\begin{thm}
The operator $R_{\ell}$ has eigenvalues in $[0,1)\cap\Q$ for each $\ell\in \Z$.
\end{thm}
\begin{proof}
Later in $\S$\ref{sec:trand} (Theorem~\ref{poly}) we will show that in fact $[\nabla]$ satisfies $p([\nabla])=0$ for
\[
p(\lambda)=\prod_{i\in I}\prod^{e_i-1}_{j=0}(\lambda-\frac{j}{e_i}).
\] 
Hence $p(\bR^{\ell} f_*[\nabla])=0$ and this implies the eigenvalues of $R_\ell$ are in $ [0,1)\cap\Q$.

Alternatively, by Grothendieck spectral sequence
\[
E^{p,q}_2=R^p f_*\sH^q(\Omega^{\bullet+n}_{X/\Delta}(\log Y)|_Y) \Rightarrow R^{p+q} f_*(\Omega^{\bullet+n}_{X/\Delta}(\log Y)|_Y),
\]
it suffices to show that the induced operator $R^pf_*\sH^q[\nabla]$ has eigenvalues in $[0,1)\cap\Q$ on $R^pf_*\sH^q \left(\Omega^{\bullet+n}_{X/\Delta}(\log Y)|_Y\right)$ for each $q\in\Z$ since $E^{p,q}_\infty$ is a sub-quotient of $E^{p,q}_2$. The following is proved by Steenbrink~\cite[Proposition 1.13]{Ste76}:
\begin{lem}\label{ste113}
The stalk of $\sH^q (\Omega^{\bullet+n}_{X/\Delta}(\log Y)|_Y)$ at a point $u\in Y$ is generated by the germs $(t^{\frac{a}{e}}\xi_{i_1}\wedge\xi_{i_2}\wedge\cdots\xi_{i_{q+n}})_u$ for all $0\leq a<e$ and all $0\leq i_1,i_2,\dots,i_{q+n} \leq n$ over the ring $\C\{t^{\frac{1}{e}}\}/t\C\{t^{\frac{1}{e}}\}$ where $e$ is the $\gcd$ of $e_0,e_1,\dots,e_k$ and $\C\{t^{\frac{1}{e}}\}$ is the ring of convergent power series with the variable $t^{\frac{1}{e}}$. 
\end{lem}

Temporarily admitting the lemma, then
\[
\sH^q[\nabla]_u (t^{\frac{a}{e}}\xi_{i_1}\wedge\xi_{i_2}\wedge\cdots\xi_{i_{q+n}})_u=(\frac{a}{e}t^{\frac{a}{e}}\xi_{i_1}\wedge\xi_{i_2}\wedge\cdots \wedge \xi_{i_{q+n}})_u,
\]
which implies that the eigenvalues of $\sH^q[\nabla]$ are $0,\frac{1}{e},\frac{2}{e},\dots,\frac{e-1}{e}\in [0,1)\cap\Q$ in a neighborhood of $u$. Hence, there exists an open neighborhood $U$ containing $u$ and a polynomial $p_U(\lambda)$ whose roots are in $[0,1)\cap \Q$ such that $p_U\left(\sH^q [\nabla]\right)=0$ over $U$. Since $Y$ is proper, we can take a finite open covering $\mathcal U=\{U_i\}$ of $Y$ such that $p(\sH^q[\nabla])=\prod_i p_{U_i}(\sH^q[\nabla])=0$.  It follows from $p(R^pf_*\sH^q[\nabla])=0$ that the eigenvalues of $R^pf_*\sH^q[\nabla]$ in $ [0,1)\cap\Q$.
\end{proof}

\begin{proof}[Proof of Lemma~\ref{ste113}]
We will prove the original statement of \cite[Proposition 1.13]{Ste76} that, in the same notations as in the lemma, the stalk at a point $u$ of $\sH^q \Omega^{\bullet+n}_{X/\Delta}(\log Y)$ is generated by germs 
\[
\left(t^{\frac{a}{e}}\xi_{i_1}\wedge\xi_{i_2}\wedge\cdots\xi_{i_{q+n}} \right)_u
\]
for all $a\in \Z_{\geq 0}$ and all tuples $0\leq i_1,i_2,\dots,i_{q+n} \leq n$ over $\C\{t^{\frac{1}{e}}\}$.

The complex $\Omega^{\bullet+n}_{X/\Delta}(\log Y)_u$ can be identified with the Koszul complex of the differential operators $D_1,D_2,\dots,D_n$ on $\sO_{X,u}$ placed in degree $-n,-n+1,\dots,0$. Define $G^j\Omega^{\ell}_{X/\Delta}(\log Y)_u$ to be the submodules of $\Omega^{\ell}_{X/\Delta}(\log Y)_u$ spanned by the germs
\[
\xi_{i_1}\wedge \xi_{i_2} \wedge\cdots\wedge \xi_{i_\ell}\quad \text{for}  \quad |\{m:i_m\leq k\}|\geq j.
\] 
Then $\{G^{\ell} \Omega^{\bullet+n}_{X/\Delta}(\log Y)_u\}_{\ell\in\Z}$ is a decreasing filtration of $\Omega^{\bullet+n}_{X/\Delta}(\log Y)_u$. The associated spectral sequence has $E^{r,\bullet}_0=\gr^{r}_G \Omega^{\bullet+n+r}_{X/\Delta}(\log Y)_u$. As $\gr^r_G \Omega^{\bullet+n+r}_{X/\Delta}(\log Y)_u$ can be identified with direct sums of Koszul complex of operators $D_{k+1},D_{k+2},\dots,D_n$ on $\sO_{X,u}$, $E^{r,\ell}_1=\sH^{r+\ell}(\gr^r_G \Omega^{\bullet+n}_{X/\Delta}(\log Y))$ vanishes for $\ell \neq -n$. Thanks to the usual Poincar\'e lemma, $E^{r,0}_1$ is spanned by germs
\[
\xi_{i_1}\wedge\xi_{i_2}\wedge\cdots\wedge \xi_{i_\ell}  \text{ such that }  |\{i_m\leq k\}|=j
\] 
over $\C\{z_0,z_1,\dots,z_k\}$. Consequently, the spectral sequence degenerates at $E_2$ with $E^{r,-n}_2=\sH^{r-n}(\Omega^{\bullet+n}_{X/\Delta}(\log Y))_u$. Note that $E^{\bullet,-n}_1$ is the Koszul complex of the differential operators $D_1,D_2,\dots,D_k$ on $\C\{z_0,z_1,\dots,z_k\}$. Because each $D_i$ for $0\leq i\leq k$ is a homogeneous differential operator, $E_2$ can be computed monomial by monomial.

For simplicity, let $\xi_{i_1,i_2,\dots,i_r}=\xi_{i_1}\wedge \xi_{i_2}\wedge\cdots \wedge \xi_{i_r}$. The claim is that a cocycle 
\[
v=\sum_{i_1<i_2,\dots<i_r} c_{i_1,i_2,\dots,i_r}z^{a_0}_0z^{a_1}_1\cdots z^{a_k}_k\xi_{i_1,i_2,..,i_r} \in E^{r,0}_1
\]
is cohomologous to zero if and only if $A_{j}\colon = {a_j}/{e_j}-{a_0}/{e_0}\neq 0$ for some $1\leq j\leq k$. Note that $D_j(z^{a_0}_0z^{a_1}_1\cdots z^{a_k}_k)=A_jz^{a_0}_0z^{a_1}_1\cdots z^{a_k}_k$ for every $1 \leq j\leq k$. Since $v$ is a cocycle, the coefficients satisfy
\begin{equation}\label{cocycle}
\sum^{r}_{\ell=1} (-1)^{\ell}c_{i_1,i_2,\dots,\hat{i_\ell},\dots,i_{r+1}}A_{i_\ell}=0.
\end{equation}
Assume that not all $A_j$'s are zero for $1\leq j\leq k$ then $A=\sum A^2_i$ is non-zero. Then the number 
\[
d_{i_1,i_2,\dots,i_{r-1}}=\sum^{k}_{\alpha=1}\frac{A_\alpha}{A} c_{\alpha,i_1,i_2,\dots,i_{r-1}}.
\]
is well-defined. Here, we define that 
$
c_{\sigma(i_1),\sigma(i_2),..,\sigma(i_r)}=\sign(\sigma)c_{i_1,i_2,\dots,i_r}
$
for any permutation $\sigma$. Then the element  
\[
\sum_{i_1<i_2<\dots<i_{r-1}}d_{i_1,i_2,\dots,i_{r-1}}z^{a_0}_0z^{a_1}_1\cdots z^{a_k}_k\xi_{i_1,i_2,\dots,i_{r-1}}
\] 
in $E^{r-1,0}_1$ has coboundary 
\[
\begin{aligned}
  & \sum^k_{\alpha=1}\sum_{i_1<\dots<i_{r-1}}A_\alpha d_{i_1,i_2,\dots,i_{r-1}}z^{a_0}_0z^{a_1}_1\cdots z^{a_k}_k\xi_{\alpha,i_1,i_2,\dots,i_{r-1}} \\
  =& \sum_{i_1<\dots<i_r}\sum^r_{\ell=1}(-1)^{\ell} A_{i_\ell}d_{i_1,i_2,\dots,\hat{i_\ell},\dots,i_r}z^{a_0}_0z^{a_1}_1\cdots z^{a_k}_k\xi_{i_1,i_2,\dots,i_r} \\
  =&\sum_{i_1<\dots<i_r}\sum^k_{\alpha=1}\sum^r_{\ell=1} (-1)^{\ell} \frac{A_{i_\ell}A_\alpha}{A}c_{\alpha,i_1,i_2,\dots,\hat{i_\ell},\dots,i_r}z^{a_0}_0z^{a_1}_1\cdots z^{a_k}_k\xi_{i_1,i_2,\dots,i_r}  \\
   =& \sum_{i_1<\dots<i_r}\sum^k_{\alpha=1} \frac{A^2_\alpha}{A}c_{,i_1,i_2,\dots,i_r}z^{a_0}_0z^{a_1}_1\cdots z^{a_k}_k\xi_{i_1,i_2,\dots,i_r} =v \quad \text{by applying~\eqref{cocycle} }.
\end{aligned}
\]
We have concluded the claim. Therefore, $E^{r,-n}_2$ is generated over $\C$ by the elements $z^{a_0}_0z^{a_1}_1\cdots z^{a_k}_k\xi_{i_1,i_2,..,i_r}$ such that 
\[
D_i(z^{a_0}_0z^{a_1}_1\cdots z^{a_k}_k)=0 \quad \Leftrightarrow \quad z^{a_0}_0z^{a_1}_1\cdots z^{a_k}_k=t^{a/e}
\]
for some $a$. We have concluded the lemma.
\end{proof}

\begin{thm}\label{eigen01}
Let $\cF^\bullet$ be a complex of $\sO_\Delta$-modules with coherent cohomology sheaves, equipped with a log connection, i.e an operator
\[
\nabla\in \End_{\bD^b(\Delta,\C)}\left(\cF^\bullet\right) \quad \text{such that } [\nabla,g]=tg' 
\]
for any local holomorphic function $g$ on $\Delta$ where $g'$ is the derivative of $g$. Assume that the residue $R_{\ell}$ of $\nabla$  defined at the beginning of this subsection acting on each cohomology $\sH^{\ell}\left(\cF^\bullet\otimes^{\mathbf L} \C(0)\right)$ has eigenvalues in $[0,1)$. Then every $\sH^{\ell}(\cF^\bullet)$ is locally free.
\end{thm}
\begin{proof}
 By the definition of residue, we have the morphism of distinguished triangles
\[
\begin{tikzcd}
\cF^\bullet \arrow{d}{\nabla+1} \arrow{r}{t} & \cF^\bullet \arrow{d}{\nabla}\arrow{r} & \mathcal \cF^\bullet\otimes^{\mathbf L}\C(0) \arrow{d}{R} \arrow{r} & \cF^\bullet[1] \arrow{d}{\left(\nabla+1 \right)[1]} \\
\cF^\bullet  \arrow{r}{t} &   \cF^\bullet \arrow{r} & \mathcal \cF^\bullet \otimes^{\mathbf L}\C(0) \arrow{r} & \cF^\bullet[1] 
\end{tikzcd}
\]
in $\bD^b(\Delta,\C)$. Then we obtain the following commutative diagram.
\begin{equation}\label{diag:exact}
\begin{tikzcd}[column sep=small]
{}\arrow{r} & \sH^{\ell-1}\left(\cF^\bullet\tens{}{\mathbf L}\C(0) \right) \arrow{r}\arrow{d}{R_\ell } & \sH^\ell\left(\cF^\bullet\right)\arrow{r}{t}\arrow{d}{\nabla+1} & \sH^\ell(\cF^\bullet) \arrow{r}\arrow{d}{\nabla} & \sH^{\ell}\left(\cF^\bullet\tens{}{\mathbf L}\C(0)\right) \arrow{r}\arrow{d}{R_{\ell+1}} & {} \\
{}\arrow{r} & \sH^{\ell-1}\left(\cF^\bullet\tens{}{\mathbf L}\C(0)\right) \arrow{r} & \sH^\ell(\cF^\bullet) \arrow{r}{t} & \sH^\ell(\cF^\bullet)\arrow{r} & \sH^{\ell}\left(\cF^\bullet\tens{}{\mathbf L}\C(0)\right) \arrow{r} & {}
\end{tikzcd}
\end{equation}
For simplicity, fix $\ell$ and let $\sH=\sH^\ell(\cF^\bullet)$ and denote by $\ker t$ the kernel of the morphism $t:\sH\to \sH$. It suffices to prove that $\ker t$ is trivial on $\sH$. We are going to show that $\ker t$ is a subset of $t^k\sH$ for all $k\geq0$ and thus, by Krull's theorem $\ker t$ is zero. 

It follows from the diagram~\eqref{diag:exact} that $\nabla+1$ on $\ker t$ and $\nabla$ on $\sH/t\sH$ have eigenvalues in $[0,1)$ and hence, there exists a polynomial $b_1(s)\in\C[s]$ with roots in $[0,1)$ such that 
\[
b_1(\nabla) \sH \subset t\sH,
\]
and another polynomial $b_2(s)\in \C[s]$ with eigenvalues in $[0,1)$ such that 
\[
b_2(\nabla+1) \ker t=0.
\]
Suppose $v$ is an element in $\ker t\cap t^k\sH$ for some $k\geq 0$. It follows that $v=t^k v_1$ for some $v_1\in \sH$.  Because the roots of $b_1(s-k)$ are bigger than the roots of $b_2(s+1)$, the two polynomials $b_1(s-k)$ and $b_2(s+1)$ are relative prime. We deduce that there exist $p(s), q(s)\in\C[s]$ such that
\[
1=p(s)b_1(s-k)+q(s)b_2(s+1).
\]
Therefore, combining the fact that $b_2(\nabla+1)v$ vanishes,
\[
v=p(\nabla)b_1(\nabla-k)v+q(\nabla)b_2(\nabla+1)v=p(\nabla)b_1(\nabla-k)t^kv_1.
\]
Because of  $(\nabla-k)t^k=t^k\nabla$ and $b_1(\nabla)v_1=tv_2$ for some $v_2\in\sH$,,  
\[
v=t^kp(\nabla+k)b_1(\nabla)v_1=t^kp(\nabla+k)b_1(\nabla)tv_2=t^{k+1}p(\nabla+k+1)b_1(\nabla+1)v_2\in t^{k+1}\sH.
\]
We have proved that $v$ is also an element in $t^{k+1}\sH$ and by induction and Krull's theorem, we have concluded the proof.
\end{proof}

Now we can immediately finish 
\begin{proof}[Proof of Theorem~\ref{mainlogdr}]
The complex $\bR f_*\Omega^{\bullet+n}_{X/\Delta}(\log Y)$ with $\bR f_*\nabla$ satisfies the condition of Theorem~\ref{eigen01}. Therefore, each cohomology $\bR^\ell f_*\Omega^{\bullet+n}_{X/\Delta}(\log Y)$ is locally free. The proof is completed by Grauert's base change theorem.
\end{proof}

\section{Transfer to $\sD-$modules}\label{sec:trand}
Lemma~\ref{ste113} implies that $\relome{\bullet+n}|_Y$ is a semi-perverse sheaf. In fact, it is even a perverse sheaf, as shown in~\cite[$\S$2]{Ste76}. According to the Riemann-Hilbert correspondence established by Kashiwara~\cite{KM84} and Mebkhout~\cite{MZ84}, there should exist a regular holonomic $\sD$-module whose de Rham complex is isomorphic to $\relome{\bullet+n}|_Y$. From a Hodge-theoretic perspective, the stupid filtration on the relative logarithmic de Rham complex should correspond to a coherent filtration. This allows us to capture the endomorphism $[\nabla]$ in the derived category by an endomorphism of a $\sD$-module. Consequently, we can study the relationship between the filtration and $[\nabla]$ in a simpler and cleaner manner. In this section, we will construct the filtered $\sD$-module and the endomorphism.

\subsection{Construction of filtered holonomic $\sD_X$-modules}
Since $\sT_{X/\Delta}(\log Y)$ is a subsheaf of $\sT_X$, the multiplication by sections in $\sT_{X/\Delta}(\log Y)$ induces a morphism $\sD_X\to \Omega_{X/\Delta}(\log Y)\otimes\sD_X$, with $P\mapsto \sum^n_{i=1}\xi_i\otimes D_i P$ locally. The morphism extends to a filtered complex of $\sD_X$-modules
\begin{equation}\label{eq:relatived}
\Omega^{n+\bullet}_{X/\Delta}(\log Y)\otimes\sD_X=\{\sD_X \to \Omega_{X/\Delta}(\log Y)\otimes\sD_X \to \cdots \to \Omega^n_{X/\Delta}(\log Y)\otimes\sD_X\}[n]
\end{equation}
with the filtration $F_{\ell}\left(\Omega^{n+\bullet}_{X/\Delta}(\log Y)\otimes\sD_X\right)$ given by
\[
\begin{aligned}
  &\relome{n+\bullet}\otimes F_{\ell+n+\bullet}\sD_X  \\
  &=\{F_{\ell}\sD_X\to \Omega_{X/\Delta}(\log Y)\otimes F_{\ell+1}\sD_X\to \cdots \to \Omega^n_{X/\Delta}(\log Y)\otimes F_{\ell+n}\sD_X\}[n].
\end{aligned}
\]
Let $\tilde \cM$ be the $0$-th cohomology of $\Omega^{n+\bullet}_{X/\Delta}(\log Y)\otimes\sD_X$ and $F_\ell\cM$ be the $\sO_X$-submodule induced by the the filtration $F_{\ell}\left(\Omega^{n+\bullet}_{X/\Delta}(\log Y)\otimes\sD_X\right)$. 
\begin{thm}~\label{tilM}
The complex $\Omega^{n+\bullet}_{X/\Delta}(\log Y)\otimes \sD_X$ is a filtered resolution of a filtered $\sD_X$-module $(\tilde\cM,F_\bullet\tilde\cM)$.
\end{thm}
\begin{proof}
As $\gr^F_{\ell}\left(\Omega^{n+\bullet}_{X/\Delta}(\log Y)\otimes \sD_X\right)=\Omega^{n+\bullet}_{X/\Delta}(\log Y)\otimes\gr^F_{\ell+n+\bullet} \sD_X$ can be identified locally with the Koszul complex associated to the regular sequence $D_1,D_2,\dots,D_n$ over the ring $\gr^F\sD_X$, the complex $\Omega^{n+\bullet}_{X/\Delta}(\log Y)\otimes\gr^F\sD_X$ is a resolution and thus, each graded peace $\gr^F_{\ell}\left(\Omega^{n+\bullet}_{X/\Delta}(\log Y)\otimes \sD_X\right)$ is also a resolution. We deduce inductively that 
$F_{\ell}\left(\Omega^{n+\bullet}_{X/\Delta}(\log Y)\otimes\sD_X\right)$ is also a resolution by using the long exact sequence associated to the short exact sequence
\[
0\to F_{\ell-1} \to F_{\ell} \to \gr^F_{\ell} \to 0.
\]
Taking direct limit, we conclude that $\Omega^{n+\bullet}_{X/\Delta}(\log Y)\otimes \sD_X$ is a resolution of $\tilde\cM$. The long exact sequence also implies the $0$-th cohomology of $F_{\ell}\left(\Omega^{n+\bullet}_{X/\Delta}(\log Y)\otimes\sD_X\right)$ is isomorphic to $F_\ell\tilde\cM$. This completes the proof.
\end{proof}

\begin{rem} \label{rem:sepdmod}
Note that $\Omega^{n+\bullet}_{X/\Delta}(\log Y)\otimes \sD_X$ is a complex of $(f^{-1}\sO_\Delta,\sD_X)$-bimodules because $\Omega^{n+\bullet}_{X/\Delta}(\log Y)$ is $f^{-1}\sO_\Delta$-linear. It follows that $\tilde \cM$ is also a $(f^{-1}\sO_\Delta,\sD_X)$-bimodule. Note we have two different actions of $t$ on $\tilde \cM$ due to the bimodule structure. We usually use the left multiplication by $t$. One can think of $\tilde \cM$ as a flat family assembling the $\sD$-module ${i_{X_p}}_+\omega_{X_p}$ of the smooth fibers $X_p$ for $p\in \Delta^*$ and a specialization $\cM=\tilde \cM/t\tilde \cM$ because using the left $f^{-1}\sO_\Delta$ structure, we have filtered isomorphisms
\[
\begin{aligned}
  \C(p)\tens{}{}\tilde\cM &\simeq \C(p)\tens{}{}\Omega^{n+\bullet}_{X/\Delta}(\log Y)\otimes\sD_X \\
  &\simeq \Omega^{n+\bullet}_{X/\Delta}(\log Y)|_{X_p}\otimes \sD_X \simeq {i_{X_p}}_*\Omega^{n+\bullet}_{X_p}\otimes\sD_X \simeq {i_{X_p}}_+\omega_{X_p},
\end{aligned}
\] 
where $i_{X_p}:X_p\to X$ is the closed embedding of the smooth fiber $X_p$.
\end{rem}

\begin{rem}
The theorem also says by choosing the local trivialization $\xi_1\wedge \xi_2\wedge\cdots\wedge \xi_n$ of $\Omega^n_{X/\Delta}(\log Y)$, $\tilde\cM$ can be identified locally with $\sD_X/(D_1,D_2,\dots,D_n)\sD_X$ and $\gr^F_\bullet\tilde\cM$ can be identified locally with $\gr^F_{\bullet+n}\sD_X/(D_1,D_2,\dots,D_n)\gr^F_{\bullet+n}\sD_X$.
\end{rem}

\begin{thm}~\label{M}
The complex $\Omega^{n+\bullet}_{X/\Delta}(\log Y)|_Y\otimes \sD_X$ is a filtered resolution of a filtered holonomic $\sD_X$-module $(\cM, F_\bullet\cM)$.
\end{thm}
\begin{proof}
Since $\Omega^{n+\bullet}_{X/\Delta}(\log Y)|_Y\otimes \sD_X$ is the cokernel of the left multiplication by $t$ on $\Omega^{n+\bullet}_{X/\Delta}(\log Y)\otimes \sD_X$, the statement that the complex $\Omega^{n+\bullet}_{X/\Delta}(\log Y)|_Y\otimes \sD_X$ is a filtered resolution is equivalent to the assertion that $t\colon \tilde\cM\to \tilde \cM$ is filtered injective. It suffices to prove that $t: \gr^F\tilde\cM\to \gr^F\tilde\cM$ is injective because the multiplication by $t$ is a filtered morphism. But this follows from $t,D_1,D_2,\dots,D_n$ is a regular sequence over the ring $\gr^F\sD_X$. We also find that $\gr^F\cM$ is isomorphic locally to $\gr^F\sD_X/(t,D_1,D_2,\dots,D_n)\gr^F\sD_X$. This means the characteristic variety $\textit{char}(\cM)$ is cut out by the local sections $t,D_1,D_2,\dots,D_n$ of $\sO_{T^*X}$ and thus, $\textit{char}(\cM)$ is of dimension $n+1$. 
\end{proof}

\begin{rem}
Let $\sD_{X/\Delta}(\log Y)$ be the subalgebra of $\sD_X$ generated by $\sT_{X/\Delta}(\log Y)$. One can show that $\tilde \cM$ is nothing but 
\[
\omega_{X/\Delta}(\log Y)\tens{\sD_{X/\Delta}(\log Y)}{} \sD_X.
\]
And the filtration $F_\bullet \tilde \cM$ is induced from $F_\bullet \omega_{X/\Delta}(\log Y)$, where $F_\ell \omega_{X/\Delta}(\log Y)$ is $\omega_{X/\Delta}(\log Y)$ for $\ell\geq -n$ and is zero otherwise. Similarly to the case of $\tilde \cM$, the $\sD_X$-module $\cM$ is just 
\[
\omega_{X/\Delta}(\log Y)|_Y \tens{\sD_{X/\Delta}(\log Y)}{} \sD_X
\]
with the filtration $F_\bullet\cM$ induced by $\left(F_\bullet \omega_{X/\Delta}(\log Y) \right)|_Y$. To keep the proof elementary, we avoid talking about $\sD_{X/\Delta}(\log Y)$-modules.
\end{rem}

\subsection{Properties of $\cM$} We first calculate the characteristic cycle of $\cM$ which is important for later when we identify the primitive part of $\gr^W\cM$. Then we prove that $\DR_X\cM$ with the induced filtration recovers $\relome{\bullet+n}|_Y$ with the stupid filtration. Lastly, we translate the operator $[\nabla]\in \End_{\bD^b(X,\C)}(\relome{\bullet+n})|_Y$ to an operator $R$ on $\cM$

\begin{thm}\label{cycle}
The characteristic cycle of $\cM$ is 
\[
cc(\cM)=\sum_{J\subset I}\sum_{j\in J}e_j \left[T^*_{Y^J}X\right],
\]
where $[T^*_{Y^J}X]$ is the cycle of the conormal bundle of $Y^J$ in $T^*X$ and $e_i$ is the multiplicity of $Y$ along each component $Y_i$ for $i\in I$.
\end{thm}\begin{proof}
The statement is local and we identify $\cM$ with $\sD_X/(t,D_1,D_2,\dots,D_n)$. The support of $\gr^F\cM$ as a sheaf on $T^*X$ is defined by the radical of the ideal generated by $(t,D_1,D_2,\dots,D_n)$ in $\gr^F\sD_X$. In fact, $z_i\partial_i$ for $0\leq i\leq k$ is in the radical because
\[
\begin{aligned}
  \left(z_i\partial_i \right)^{e_0+e_1+\cdots+e_k} \equiv & (z_0\partial_0)^{e_0} (z_1\partial_1)^{e_1}\cdots (z_k\partial_k)^{e_k} & \\
  \equiv & t\partial^{e_0}_0\partial^{e_1}_1\cdots\partial^{e_k}_k \\
  \equiv & 0  \quad  \mathrm{mod}\, (t, D_1,D_2,\dots,D_n)\gr^F\sD_X .
\end{aligned}
\]
Therefore, $char(\cM)$ is cut out by $t_{\mathrm{red}},z_0\partial_0,z_1\partial_1,\dots,z_k\partial_k,\partial_{k+1},\dots,\partial_n$, where $t_{\mathrm{red}}=z_0z_1\cdots z_k$. It follows that $char(\cM)=\bigcup_{J\subset I} T^*_{Y^J}X$. 

Denote by $\mathfrak{p}(Z)$ the prime ideal defining an integral subvariety $Z$. 
Let 
\[
cc(\cM)=\sum_{J\subset I} m_J \left[T^*_{Y^J}X \right]
\]
where $m_J$ is the length of $\gr^F\cM_{\mathfrak{p}(T^*_{Y^J}X)}$ as an Artinian $\gr^F\sD_{X,\mathfrak p(T^*_{Y^J}X)}$-module. Without loss of generality, let us assume $J=\{0,1,2,..,\mu\}$ and by abuse of notation we also the prime ideal $\mathfrak{p}=\mathfrak{p}(T^*_{Y^J}X)$ of the variety $T^*_{Y^J}X$ is locally generated by $z_0,z_1,\dots,z_\mu,\partial_{\mu+1},\partial_{\mu+2},\dots,\partial_n$ over $\gr^F\sD_X$ in some local coordinate system. Then 
\[
\gr^F\sD_{X,\mathfrak p}/(t,D_1,D_2,\dots,D_n)\gr^F\sD_{X,\mathfrak{p}}=\gr^F\sD_{X,\mathfrak p}/(D'_0,D'_1,\dots,D'_n)\gr^F\sD_{X,\mathfrak p}
\] 
where 
\begin{equation}
D'_i=\left\{
\begin{aligned}
z^{e_0+e_1+\cdots+e_\mu}_0 &, & \text{ for }i=0\\
 \frac{1}{e_i}z_i-\frac{1}{e_0}z_0\frac{\partial_0}{\partial_i} &, & \text{ for }1\leq i \leq \mu\\
\frac{1}{e_i}\partial_i-\frac{1}{e_0}z_0\frac{\partial_0}{z_i} &, & \text{ for } \mu+1 \leq i \leq k \\
\partial_i &, & \text{ for } i>k,
\end{aligned}
\right.
\end{equation}
because $\partial_0,\partial_1,\dots,\partial_\mu,z_{\mu+1},z_{\mu+2},\dots,z_k$ are invertible in $\gr^F\sD_{X,\mathfrak{p}}$. We observe that $\gr^F\cM_{\mathfrak{p}}$ is generated by $z_0^\alpha$ as a $\gr^F\sD_{X,\mathfrak{p}}$-module for $0 \leq \alpha< e_0+e_1+\cdots+e_\mu$ and that $\mathfrak{p}\cdot \gr^F\sD_{X,\mathfrak{p}}$ is generated by $z_0, D'_1,\dots,D'_n$ over $\gr^F\sD_{X,\mathfrak{p}}$. It will immediately follow that $m_J=\sum_{j\in J} e_j$ once we prove that 
\[
0= z^{e_0+e_1+\cdots+e_\mu}_0 \cdot \gr^F\cM_{\mathfrak{p}} \subset \cdots \subset z_0\cdot \gr^F\cM_{\mathfrak{p}} \subset \gr^F\cM_{\mathfrak{p}}
\] 
is a composition series of $\gr^F\cM_{\mathfrak{p}}$. It is clear that $P \mapsto z^\alpha_0 P$ induces a well-defined surjection from  $\gr^F\sD_{X,\mathfrak{p}}/\mathfrak{p}\cdot \gr^F\sD_{X,\mathfrak{p}}$ to $z_0^\alpha \cdot \gr^F\cM_{\mathfrak{p}}/z_0^{\alpha+1} \cdot \gr^F\cM_{\mathfrak{p}}$. Then we must show that $z_0^\alpha \cdot \gr^F\cM_{\mathfrak{p}}/z_0^{\alpha+1} \cdot \gr^F\cM_{\mathfrak{p}}$ is non-zero when $0 \leq \alpha< e_0+e_1+\cdots+e_\mu$. Suppose that the class in $z_0^\alpha \cdot \gr^F\cM_{\mathfrak{p}}/z_0^{\alpha+1} \cdot \gr^F\cM_{\mathfrak{p}}$ represented by $z_0^\alpha$ vanishes. Then there exist $P,Q_1, Q_2,\dots ,Q_n$ in $\gr^F\sD_X$ with $P\notin \mathfrak p$ such that
\[
z_0^\alpha P = z_0^{\alpha+1}Q_0 +\sum_{i=1}^n D_i Q_i
\]
holds in $\gr^F\sD_X$. It follows from $z_0^\alpha, D_1,D_2,\dots, D_n$ is a regular sequence that $P$ is in the ideal generated by $z_0,D_1,D_2,\dots, D_n$. But this ideal is contained by $\mathfrak p$, which is a contradiction. 
\end{proof}

\begin{rem}
The above theorem verifies that 
\[
cc(\cM)=\lim_{p\to 0} cc({i_p}_+\omega_{X_p})=\lim_{p\to 0}\left[T^*_{X_p}X\right]
\] 
as cycles in algebraic cotangent space $T^*X$ for $p\in \Delta^*$ where $i_p\colon X_p\to X$ the closed embedding of the smooth fiber. In fact, one can show that $\C(p)\otimes \gr^F\tilde\cM$, using the left $f^{-1}\sO_\Delta$-module structure, is isomorphic to $\gr^F{i_p}_+\omega_{X_p}$ as in Remark~\ref{rem:sepdmod}. Refer to~\cite{Gins} for general results about the characteristic cycles of specializations of holonomic $\sD$-modules. 
\end{rem}

\begin{cor}~\label{spm}
The de Rham complex $\DR_X\cM$ together with filtration $F_\bullet\DR_X\cM$ is isomorphic to $\Omega^{n+\bullet}_{X/\Delta}(\log Y)|_Y$ with the stupid filtration in the derived category of filtered complexes of sheaves of $\C$-vector spaces.
\end{cor}
\begin{proof}
We have showed that $F_{\ell}\left(\Omega^{n+\bullet}_{X/\Delta}(\log Y)|_Y\otimes \sD_X\right)$ is a resolution of $F_{\ell}\cM$. Therefore, the total complex of $F_{\ell+*}\left(\Omega^{n+\bullet}_{X/\Delta}(\log Y)|_Y\otimes \sD_X\right)\otimes \bigwedge^{-*}\sT_X$ is quasi-isomorphic to $F_{\ell+*}\cM\otimes\bigwedge^{-*}\sT_X$, which is exactly $F_{\ell}\DR_X\cM$. It remains to show the total complex also quasi-isomorphic to $F_{\ell}\Omega^{n+\bullet}_{X/\Delta}(\log Y)|_Y$. This follows from that 
\[
\begin{aligned}
F_{\ell+*}\left(\Omega^{n+\bullet}_{X/\Delta}(\log Y)|_Y\otimes \sD_X\right)\otimes \bigwedge^{-*}\sT_X &=\Omega^{n+\bullet}_{X/\Delta}(\log Y)|_Y\otimes \left(F_{\ell+n+\bullet}\sD_X\otimes \bigwedge^{-*}\sT_X\right) \\
  &\simeq\Omega^{n+\bullet}_{X/\Delta}(\log Y)|_Y\otimes F_{\ell+n+\bullet}\sO_X\\
  &=F_{\ell}\Omega^{n+\bullet}_{X/\Delta}(\log Y)|_Y.
\end{aligned}
\]
Here, $F_\ell\sO_X=\sO_X$ for $\ell\geq 0$ and otherwise it is zero. \end{proof}

\begin{thm}~\label{endm}
The endomorphism $\nabla\in \End_{\bD^b(X,\C)}\Omega^{n+\bullet}_{X/\Delta}(\log Y)$ in Lemma~\ref{nabla} can be identified with, under the functor $\DR_X$, the filtered morphism 
\[
\nabla\colon(\tilde\cM,F_\bullet\tilde\cM)\to (\tilde\cM,F_{\bullet+1}\tilde\cM),
\quad \left[[\alpha]\otimes P\right]\mapsto (\frac{dt}{t}\wedge)^{-1}\left\{d(\alpha\otimes P)\right\}
\]
where $\alpha\in \Omega^n_X(\log Y)$ and $P\in \sD_X$ so that $[\alpha]\otimes P\in\relome{n}\otimes\sD_X$. Moreover, restriction on $Y$ gives a filtered morphism 
\[
R:(\cM,F_\bullet\cM)\to (\cM,F_{\bullet+1}\cM)
\]
such that
\begin{equation}\label{poly}
\prod_{i\in I} \prod^{e_i-1}_{j=0}(R-\frac{j}{e_i})=0.
\end{equation}
\end{thm}
\begin{proof}
The morphism $\frac{dt}{t}\wedge\colon\Omega^{n+\bullet}_{X/\Delta}(\log Y)\to \Omega^{n+1+\bullet}_{X}(\log Y)$ extends to the corresponding complexes of induced $\sD_X$-modules 
\[
\frac{dt}{t}\wedge\colon\Omega^{n+\bullet}_{X/\Delta}(\log Y)\otimes\sD_X\to \Omega^{n+1+\bullet}_{X}(\log Y)\otimes\sD_X.
\]
Let $\Cone^\bullet\otimes \sD_X$ be the mapping cone of the above morphism. We get a diagram of complexes of $\sD_X$-modules similarly to~\eqref{res} and taking $0$-th cohomology we get the following commutative diagram.
\begin{equation}
\begin{tikzcd}
\sH^0 \left(\Cone^\bullet\otimes\sD_X\right) \arrow{r}{q} \arrow{d}{p} & \tilde\cM \arrow{dl}{p\circ q^{-1}} \\
\tilde\cM &
\end{tikzcd}
\end{equation}
where abuse of notation, still denote by $p$ and $q$ the induced morphisms from diagram~\eqref{res}. Now $q$ is an isomorphism of $\sD_X$-modules. Let $[\alpha\otimes P, [\beta]\otimes Q]$ be a class in $\sH^0(\Cone^\bullet\otimes\sD_X)$ for any $\alpha\otimes P\in \Omega^n_X(\log Y)\otimes \sD_X$ and $[\beta]\otimes Q\in \Omega^{n}_{X/\Delta}(\log Y)\otimes\sD_X$. Then 
\[
\delta\left(\alpha\otimes P, [\beta]\otimes Q \right)=\left((-1)^{n}d(\alpha\otimes P)+\frac{dt}{t}\wedge\beta\otimes Q,(-1)^{n}d([\beta]\otimes Q)\right)=0.
\]
Here, the sign factor $(-1)^n$ shows up because of the Koszul sign rule. Since 
\[
\frac{dt}{t}\wedge\colon \Omega^n_{X/\Delta}(\log Y)\to \Omega^{n+1}_X(\log Y)
\] 
is an isomorphism, we have
\[
[\beta]\otimes Q=(-1)^{n-1}(\frac{dt}{t}\wedge)^{-1}\{d(\alpha\otimes P)\}.
\]
Therefore, $q^{-1}\colon \tilde \cM\to \sH^0(\Cone^\bullet\otimes\sD_X)$ is given by 
\[
[[\alpha]\otimes P]\mapsto [\alpha\otimes P, (-1)^n(\frac{dt}{t}\wedge)^{-1}\{d(\alpha\otimes P)\}],
\] 
and hence,
\[
\nabla=(-1)^{n-1}p\circ q^{-1}\colon[[\alpha]\otimes P]\mapsto (\frac{dt}{t}\wedge)^{-1}\{d(\alpha\otimes P)\}.
\]
Let $R$ be the endomorphism of $\cM$ obtained by restriction $\nabla$ on $Y$. If $\alpha=\xi_1\wedge\xi_2\wedge\cdots\wedge\xi_n$, 
\[
\begin{aligned}
R[\xi_1\wedge\xi_2\wedge\cdots\wedge\xi_n\otimes P] &=(\frac{dt}{t}\wedge)^{-1}\left(d \left(e_1\frac{dz_1}{z_1}\wedge e_2\frac{dz_2}{z_2}\wedge\cdots\wedge dz_n\otimes P\right)\right) \\
  &=(\frac{dt}{t}\wedge)^{-1}\left(e_0\frac{dz_0}{z_0}\wedge e_1\frac{dz_1}{z_1}\wedge e_2\frac{dz_2}{z_2}\wedge\cdots\wedge dz_n\otimes \frac{1}{e_0}z_0\partial_0 P\right) \\
  &= \left[\xi_1\wedge\xi_2\wedge\cdots\wedge\xi_n\otimes \frac{1}{e_0}z_0\partial_0 P\right].
\end{aligned}
\]
We see that $R F_\bullet\cM\subset F_{\bullet+1}\cM$. The reason for $\nabla F_\bullet\tilde\cM\subset F_{\bullet+1}\tilde\cM$ is similar. 

To get the characteristic polynomial of $R$, we work locally and identify $\cM$ with $\sD_X/(t,D_1,\dots,D_n)$ via the local trivialization $\xi_1\wedge\xi_2\wedge\cdots\wedge\xi_n$ of $\Omega^n_{X/\Delta}(\log Y)$. Then for $P\in\sD_X$, $R[P]=[\frac{1}{e_0}z_0\partial_0 P]$. Modulo the relations $D_1,D_2,\dots,D_n$, the left multiplication by $\frac{1}{e_0}z_0\partial_0$ on $\cM$ is the same as the multiplication by $\frac{1}{e_i}z_i\partial_i$ for $1\leq i \leq k$. It follows from the identity
\[
(z\partial)(z\partial-1)\cdots(z\partial-\ell)=z^{\ell+1}\partial^{\ell+1}
\]
for any $\ell\geq 0$ that 
\[
\begin{aligned}
\prod_{i\in I} \prod^{e_i-1}_{j=0}(R-\frac{j}{e_i})[P] &=\prod_{i\in I} \prod^{e_i-1}_{j=0}(\frac{1}{e_i}z_i\partial_i-\frac{j}{e_i})[P]=\prod_{i\in I}\frac{1}{e^{e_i}_i}z^{e_i}_i\partial^{e_i}_i[P]=t\prod_{i\in I}\frac{1}{e_i^{e_i}}\partial^{e_i}_i[P] \\
  &=0\in {\sD_X}/{(D_1,D_2,\dots,D_n,t)\sD_X}.
\end{aligned}
\]
This completes the proof.
\end{proof}

\begin{rem}\label{logcntm}
Note that $\nabla\colon\tilde\cM\to\tilde\cM$ is also can be identified with the left multiplication by $\frac{1}{e_i}z_i\partial_i$ for $i\leq k$, because of the relations $D_i=\frac{1}{e_i}z_i\partial_i-\frac{1}{e_0}z_0\partial_0$ for $1\leq i\leq k$. This means for any local section $g$ of $f^{-1}\sO_\Delta$, we have $[\nabla,g]=tg'$ where $t$ and $g$ are local sections of $f^{-1}\sO_\Delta$ acting on the left of $\tilde \cM$. This makes $\tilde\cM$ a $(f^{-1}\sD_\Delta(\log 0),\sD_X)$-bimodule. Using Godement resolution, the direct image $\bR f_*\DR_X\tilde \cM$ is a complex of left $\sD_\Delta(\log 0)$-modules. Similarly, as we already saw in the proof, locally the morphism $R\colon\cM\to \cM$ can be identified with left multiplication by $\frac{1}{e_i}z_i\partial_i$ for $0\leq i\leq k$, meaning $[R,g]=tg'=0$ for $g$ local sections of $f^{-1}\sO_\Delta$ acting left on $\cM$. 
\end{rem}

\begin{rem}
We point out that the $\sD_X$-module $\cM$ is even regular holonomic although we will not use this observation. Later we will see that $\cM$ is an extension of regular holonomic $\sD_X$-modules which will again prove that $\cM$ is regular (see Theorem~\ref{iden} for the reduced case and Theorem~\ref{idenN} for the general case).
\end{rem}

\section{Reduced case: Strictness and the weight filtration}\label{sec:redstrict}
We begin to study the weight filtration $W_\bullet\cM$ induced $R$ on $\cM$. To motivate the results and the ideas, we assume $Y$ is reduced in $\S$\ref{sec:redstrict} and $\S$\ref{sec:redses}. The general case will be treated in  $\S$\ref{sec:malpha} and $\S$\ref{sec:sesquil}. Hence, the multiplicity $e_i$ of irreducible component $Y_i$ is $1$ and $R$ is nilpotent. Recall that the weight filtration of the nilpotent operator $R$ is uniquely characterized by the following two properties: 
\begin{itemize}
  \item for $\ell\in\Z$, $R\colon W_\ell\cM\to W_{\ell-2}\cM$;
  \item the induced operator $R^\ell\colon \gr^W_{\ell}\cM \to \gr^W_{-\ell}\cM$ is an isomorphism for $\ell\geq 0$.
\end{itemize}

\subsection{Strictness of $R$}
Let $F_\bullet W_r\cM=F_\bullet\cM\cap W_r\cM$ be the induced filtration for every integer $r$. The good filtration and the weight filtration interact nicely because of the following theorem.

\begin{thm}~\label{strict}
Any power of $R$ is strict on $(\cM,F_\bullet\cM)$:  
\[
R^aF_\bullet \cM=F_{a+\bullet}R^a\cM \quad {for any} \quad a\geq 0.
\]
\end{thm}

\begin{proof} 
Since strictness is a local property, we can assume that
\[
(\cM,F_\bullet)=(\sD_X/(t,D_1,D_2,\dots,D_n)\sD_X,F_{\bullet+n})
\] 
and $R$ is left multiplication by $z_0\partial_0$ on it, recalling that $D_i=z_i\partial_i-z_0\partial_0$ for $1\leq i\leq k$ and $D_i=\partial_i$ for $k+1\leq i\leq n$. It is clear that $R^a F_b\cM$ is contained in $F_{a+b}R^a\cM$. It suffices to show that for every $R^aP\in F_{a+b}\cM$, we can find an element $Q\in F_{b}\cM$ such that $R^aP=R^aQ$. Suppose that $P\in F_\ell\cM$. We can assume that $\ell> b$; otherwise the statement is empty. Then the class of $R^aP$ vanishes in $\gr^F_{a+\ell}\cM$. We have the following lemma:

\begin{lem}\label{locgen}
Denote by $[R]$ the induced operator on $\gr^F\cM$. Then $\ker [R]^{r+1}$ is locally generated by the classes of all degree $k-r$ monomials dividing $t=z_0z_1\cdots z_k$.
\end{lem}

We can easily check that monomials of degree $k-r$ dividing $t$ are in $\ker [R]^{r+1}$. Indeed, it is already true that monomials of degree $k-r$ dividing $t$ is in $\ker R^{r+1}$. Without loss of generality, we only need to check this for the monomial $z_{r+1}z_{r+2}\cdots z_k$:
\[
R^{r+1} z_{r+1}z_{r+2}\cdots z_k= z_0\partial_0 z_1\partial_1\cdots z_r\partial_rz_{r+1}z_{r+2}\cdots z_k=t\partial_0\cdots\partial_k =0\in \cM.
\]
We will prove the opposite direction after finishing the proof of the theorem. Going back to the proof of the theorem, by Lemma~\ref{locgen}, there exist $z_J=\prod_{j\in J}z_j$, $Q_{J}\in F_\ell\cM$ and $Q_{\ell-1}\in F_{\ell-1}\cM$ such that 
\begin{equation}\label{eq:iter}
  P=\sum_{\substack{J\subset I, \\ |J|=k-a+1}} z_J Q_{J} + Q_{\ell-1}.
\end{equation}
It follows from $R^a z_J=0$ for $z_J$ of degree $k-a+1$ dividing $t$ that $R^aP=R^aQ_{\ell-1}$. Iterating~\eqref{eq:iter}, we eventually find an element $Q\in F_b\cM$ such that 
$R^aP=R^aQ$ with $Q\in F_b\cM$.
\end{proof}

\begin{proof}[Proof of Lemma~\ref{locgen}]
We work on the commutative ring $\gr^F\sD_X$. We proceed by induction on $r$. Let $P\in \gr^F\sD_X$ be a representative of an element in $\ker [R]^{r+1}$. When $r=0$, we have
\[
z_0\partial_0 P= tQ_0+\sum^n_{i=1}D_iQ_i. 
\]
Then $tQ_0\in (\partial_0,\partial_1,\dots,\partial_n)\gr^F\sD_X$. Notice that $t,\partial_0,\partial_1,\dots,\partial_n$ is a regular sequence over $\gr^F\sD_X$. We have $Q_0=\sum^n_{i=0}\partial_i Q'_i$. This implies 
\[
\begin{aligned}
z_0\partial_0 P &=\sum^k_{i=0}\frac{t}{z_i} z_i\partial_i Q'_i + \sum^n_{j=k+1}t\partial_{j}Q'_j +\sum^n_{i=1}D_iQ_i \\
  &=\sum^k_{i=0}\frac{t}{z_i} z_0\partial_0 Q'_i +\sum^k_{i=1} D_i (Q_i+\frac{t}{z_i}Q'_i)+\sum^n_{j=k+1}D_j(Q_j+tQ'_j),
\end{aligned}
\]
from which we conclude that $z_0\partial_0 (P-\sum^k_{i=0}\frac{t}{z_i}Q'_i)\in(D_1,D_2,\dots,D_n)\gr^F\sD_X$. Because $z_0\partial_0, D_1,D_2,\dots,D_n$ is again a regular sequence, we see that $P-\sum^k_{i=0}\frac{t}{z_i}Q'_i \in (D_1,D_2,\dots,D_n)\gr^F\sD_X$. This concludes the base case for the induction.

Assume the statement is true for the cases when the exponent is less than $r+1$. For $[P]\in \ker [R]^{r+1}$, we have $[R][P]$ is in $\ker [R]^r$. By induction hypothesis,
\begin{equation}\label{strictred}
z_0\partial_0 P= \sum_{\substack{|J|=k-r+1,\\ J\subset I}} z_J Q_J+\sum^n_{i=1}D_i Q_i,
\end{equation}
where $z_J=\prod_{j\in J} z_j$. Fix an index subset $J$ of $I$ such that $|J|=k-r+1$. Then $z_J Q_J$ is in the submodule generated by $z_i $ for $i\in I\setminus J$ and $\partial_j$ for $j\in J$ and $k<j\leq n$ over $\gr^F\sD_X$. Since $z_i $ for $i\in I\setminus J$, $\partial_j$ for $j\in J$ and $k<j\leq n$ together with $z_J$ form a regular sequence, we have 
\[
Q_J=\sum_{i\in I\setminus J} z_i Q'_i +\sum_{j\in J} \partial_j Q'_j +\sum_{k<\ell \leq n} \partial_\ell Q'_\ell.
\]
Therefore, it follows that 
\[
z_J Q_J=\sum_{i\in I\setminus J} z_Jz_i Q'_i+\sum_{j\in J} \left(\frac{z_J}{z_j}z_0\partial_0 Q'_j+  D_j\frac{z_J}{z_j}Q'_j \right)+\sum_{k<\ell\leq n} D_\ell z_J Q'_\ell.   
\]
Then substituting in~\eqref{strictred}, we deduce that
\[
z_0\partial_0\left(P-\sum_{j\in J} \frac{z_J}{z_j} Q'_j \right)-\sum_{i\in I\setminus J} z_Jz_i Q'_i
\]
is in the submodule generated by degree $k-r+1$ monomials dividing $t$ except $z_J$, and $D_1,D_2,\dots,D_n$ over $\gr^F\sD_X$. It follows that we can reduce the monomials of degree $k-r+1$ dividing $t$ in the right-hand side equation~\eqref{strictred} one by one and at the last step, we get 
$z_0\partial_0 \left( P- P'\right)-Q'$, where $P'$ is a linear combination of degree $k-r$ monomials dividing $t$ and $Q'$ is a linear combination of $k-r+2$ monomials dividing $t$, is in the submodule generated by $D_1,\dots,D_n$ over $\gr^F\sD_X$. But $\ker [R]^{r-1}$ is generated by classes represented by degree $k-r+2$ monomials dividing $t$ by induction hypothesis. It says that the class of $P-P'$ is in $\ker [R]^r$ and by induction it is generated by degree $k-r+1$ monomials dividing $t$. Therefore, $P$ is a linear combination of degree $k-r$ monomials dividing $t$. This completes the proof.
\end{proof}

\begin{cor}\label{localgenR}
The $\ker R^{r+1}$ is also generated by degree $k-r$ monomials dividing $t$ if one identifies $\cM$ locally with $\sD_X/(t,D_1,D_2,\dots,D_n)\sD_X$.
\end{cor}
\begin{proof}
It suffices to show that $\gr^F\ker R^{r+1}$ is generated by degree $k-r$ monomials dividing $t$. Notice that $\gr^F\ker R^{r+1}$ is contained in $\ker [R]^{r+1}$, since $[R]^{r+1}$ vanishes on $\gr^F\ker R^{r+1}$. In fact, we have $\gr^F\ker R^{r+1}=\ker [R]^{r+1}$ because degree $k-r$ monomials dividing $t$ are also in $\gr^F\ker R^{r+1}$. 
\end{proof}

\subsection{The weight filtration}
The results concerning the weight filtration and Lefschetz decomposition are formal and we will work on the abstract setting.

\begin{thm}\label{thm:genlef}
Let $N\colon (\mathcal G, F_\bullet)\to (\mathcal G, F_{\bullet+1})$ be a nilpotent operator on a filtered $\sD$-module $(\cG,F_\bullet)$. Assume that every power of $N$ satisfies strictness, i.e., $N^a F_b\cG=F_{a+b}N^a\cG$ for $a\geq 0$ and $b\in\Z$. Then the induced operator $N^r\colon F_\ell \gr^W_r\cG \to F_{\ell+r}\gr^W_{-r}\cG$ is an isomorphism for $r\geq 0$, where $W_\bullet$ is the monodromy filtration induced by $N$.
\end{thm}
\begin{proof}
It suffices to prove that for any $b\in F_{\ell+r}W_{-r}\cG$, we could find $a'\in F_\ell W_{r}\cG$ such that $a= N^r a'$.  Because $W_{-r}\cG\subset N^{r}\cG$, let $N^{r}a =b$ for some $a$. Then by strictness, there exists $a'\in F_\ell\cG$ such that $N^{r}a'=N^ra\in W_{-r}\cG$. It follows that $a'\in W_r\cG$. Indeed, if $a'\in W_{r+k}\cG$ for some $k>0$ such that $a'\neq 0\in \gr^W_{r+k}\cG$. Then $N^{r+k}a'=0\in \gr^W_{-r-k}\cG$ because $N^{r}a'=0\in \gr^W_{-r+k}\cG$, from which we conclude that $a'\in F_\ell W_{r+k-1}\cG$. Thus, iterating the procedure, $a'$ is actually in $F_\ell W_r\cG$. We conclude the proof.
\end{proof}

Let $\cP_r\colon = \ker \left(N^{r+1}\colon \gr^W_r\cG\to \gr^W_{-r-2}\cG\right)$ be the primitive part of $\gr^W\cG$, which can be identified with 
\[
\frac{\ker N^{r+1}}{\ker N^r+N\ker N^{r+2}}.
\]
See Example~\ref{defwtf}. Recall the Lefschetz decomposition:
\[
\gr^W_r\cG=\bigoplus_{\substack{\ell\geq 0,-\frac{r}{2}}} N^{\ell}\cP_{r+2\ell}\text{ \,for any\,  }r\in\Z.
\]
There are two possible ways to define the filtration on $\cP_r$: first, we have the natural filtration $F_\ell\cP_r$ induced from the inclusion $\cP_r \to \gr^W_r\cG$ and second we can also define the filtration using 
\[
\frac{F_\ell\ker N^{r+1}+\ker N^r+N\ker N^{r+2}}{\ker N^r+N\ker N^{r+2}}.
\]
But indeed, the two different methods result in the same filtration because of the strictness. Let $m\in F_\ell W_r+W_{r-1}$ such that $N^{r+1}m\in W_{-r-3}$ so that represents a class in $F_\ell\cP_r$. It suffices to find an element in $F_\ell\ker N^{r+1}$ representing the same class as $m$ in $F_\ell\cP_r$. Let $m=m_1+m_2$ for $m_1\in F_\ell W_r$ and $m_2\in W_{r-1}$. It follows that $N^{r+1}m_1\in F_{\ell+r+1}W_{-r-3}$ because both $N^{r+1}m, N^{r+1}m_2\in W_{-r-3}$ and $m_1\in F_\ell W_r$. Since $N^{r+3}\colon F_{\ell-2}W_{r+3}\to F_{\ell+r+1}W_{-r-3}$ is surjective, there exists $x\in F_{\ell-2}W_{r+3}$ such that $N^{r+3}x=N^{r+1}m_1\in F_{\ell+r+1}W_{-r-3}$. See the proof of Theorem~\ref{thm:genlef}. It follows that $m_1-N^2x\in F_\ell \ker N^{r+1}$ represents the same element as $m$ in $F_\ell\cP_r\subset F_\ell \gr^W_r$. 

\begin{cor}
The Lefschetz decomposition of $\gr^W\cG$ respects filtrations, i.e.
\[
F_\bullet\gr^W_r\cG=\bigoplus_{\substack{\ell\geq 0,-\frac{r}{2}}} N^{\ell}F_{\bullet-\ell}\cP_{r+2\ell}\text{ \,for any\,  }r\in\Z.
\]
\end{cor}

Returning to our situation, it follows that:
\begin{thm}\label{FW}
The induced operator $R^r\colon F_\ell \gr^W_r\cM \to F_{\ell+r}\gr^W_{-r}\cM$ is an isomorphism and the Lefschetz decomposition of $\gr^W\cM$ respects filtrations, i.e.
\[
F_\bullet\gr^W_r\cM=\bigoplus_{\substack{\ell\geq 0,-\frac{r}{2}}} R^{\ell}F_{\bullet-\ell}\cP_{r+2\ell}\text{ \,for any\,  }r\in\Z.
\]
\end{thm}

\subsection{Identifying the primitive part $\cP_r$}

Recall that $Y^J=\cap_{j\in J}Y_j$ for a subset $J$ of the index set $I$ and $\tilde Y^{(r+1)}$ is the disjoint union of $Y^J$ such that the cardinality of $J$ is $r+1$. The morphism $\tau^{(r+1)}\colon \tilde Y^{(r+1)}\to X$ is the natural morphism induced by the closed embeddings $\tau^J\colon Y^J\to X$.

\begin{thm} \label{iden}
There exists a canonical filtered isomorphism 
\[
\phi_r\colon(\cP_r, F_\bullet \cP_r) \to  \tau^{(r+1)}_+\omega_{\tilde Y^{(r+1)}}(-r).
\]
\end{thm}

\begin{proof}
Denote by $D^J$ the normal crossing divisor $Y^J\cap Y_{I\setminus J}$ on $Y^J$. 
The residue map 
\[
\Res_{\tilde Y^{(r+1)}}\colon\Omega^{\bullet+n+1}_{X}(\log Y)|_Y \to \bigoplus_{|J|=r+1}\Omega^{\bullet+n-r}_{ Y^{J}}(\log D^J)
\]
extends to a morphism of complexes of filtered induced $\sD_X$-modules
\[
\Res_{\tilde Y^{(r+1)}}\colon\Omega^{\bullet+n+1}_{X}(\log Y)|_Y\otimes\sD_X \to \bigoplus_{|J|=r+1}\Omega^{\bullet+n-r}_{Y^{J}}(\log D^J)\otimes\sD_X
.\]
Let $\cH^k\colon=\sH^k \left(\Omega^{\bullet+n+1}_{X}(\log Y)|_Y\otimes\sD_X \right)$. Looking at the $0$-th cohomology, by Example~\ref{starex} we obtain that
\[
\Res_{\tilde Y^{(r+1)}}\colon \cH^0 \to \bigoplus_{|J|=r+1}\tau^J_+\omega_{Y^J}(* D^J)(-r).
\] 
Since the morphism $\frac{dt}{t}\wedge\colon \Omega^{\bullet+n}_{X/\Delta}(\log Y)\to \Omega^{\bullet+n+1}_X(\log Y)$ also extends to a morphism of the complexes of induced $\sD_X$-modules, there is a short exact sequence 
\[
0 \to \Omega^{\bullet+n}_{X/\Delta}(\log Y)|_Y\otimes\sD_X \xrightarrow{\frac{dt}{t}\wedge}   \Omega^{\bullet+n+1}_X(\log Y)|_Y\otimes\sD_X \to \Omega^{\bullet+n+1}_{X/\Delta}(\log Y)|_Y\otimes\sD_X \to 0.
\]
The associated long exact sequence
\begin{equation}\label{dttr}
\begin{tikzcd}
  0 \arrow[r]
    & \cH^{-1} \arrow[r]
        \arrow[d, phantom, ""{coordinate, name=Z}]
      & \cM \arrow[dll,
                 "R",
    rounded corners,
    to path={ -- ([xshift=2ex]\tikztostart.east)
        |- (Z) [near end]\tikztonodes
      -| ([xshift=-2ex]\tikztotarget.west) 
      -- (\tikztotarget)}] \\
         \cM \arrow[r,"\frac{dt}{t}\wedge"]
    & \cH^0 \arrow[r]
& 0
\end{tikzcd}
\end{equation}
implies the morphism $\frac{dt}{t}\wedge\colon \cM\to \cH^0$ vanishes on the image of $R$. 
To motivate the proof, let me do some local calculations. Let $\zeta=\frac{dz_1}{z_1}\wedge \frac{dz_2}{z_2}\wedge\cdots \wedge\frac{dz_k}{z_k}\wedge\cdots\wedge dz_n$ represent a local frame of $\Omega^n_{X/\Delta}(\log Y)|_Y$. Then a local section of $\cM$ is represented by $\zeta\otimes P$ for $P$ a local section $\sD_X$ and, $\Res_{\tilde Y^{(r+1)}} \frac{dt}{t}\wedge\zeta\otimes P$ is a section of $\bigoplus_{|J|=r+1}\Omega^{n-r}_{ Y^{J}}(\log D^J)\otimes\sD_X$. Post-composing with the projection 
\[
\bigoplus_{|J|=r+1}\Omega^{n-r}_{ Y^{J}}(\log D^J)\otimes\sD_X \to \bigoplus_{|J|=r+1}\tau^J_+\omega_{Y^J}(* D^J)(-r),
\]
we make the morphism explicit:
\[
\Res_{\tilde Y^{(r+1)}}\circ \frac{dt}{t}\wedge\colon\cM \to \bigoplus_{|J|=r+1}\tau^J_+\omega_{Y^J}(* D^J)(-r), \quad [\zeta\otimes P] \mapsto [\Res_{\tilde Y^{(r+1)}} \frac{dt}{t}\wedge\zeta\otimes P].
\]
Let $\zeta\otimes z_{\overline J}P$ represent a class in $\ker R^{r+1}$ for some fixed ordered index subset $J$ with $|J|=r+1$, where $z_{\overline J}= \prod_{j\in I\setminus J}z_j$ (Corollary~\ref{localgenR}). Its image under the above morphism is only contained in the component $\tau^J_+\omega_{Y^J}(* D^J)(-r)$ because $z_{\overline J}$ vanishes on other components. Thus, the image is the class represented by
\begin{equation}\label{eq:resimage}
\begin{aligned}
  &\Res_{\tilde {Y}^{r+1}}\frac{dz_0}{z_0}\wedge\frac{dz_1}{z_1}\wedge\cdots \frac{dz_k}{z_k}\wedge\cdots\wedge dz_n\otimes z_{\overline J}P \\
  &= \pm \frac{dz_{\overline J}}{z_{\overline J}} \wedge dz_{k+1}\wedge\cdots dz_n \otimes z_{\overline J}P \in \Omega^{n-r}_{Y^J}(\log D^J)\otimes \sD_X,
\end{aligned}
\end{equation}
where $\frac{dz_{\overline J}}{z_{\overline J}}=\bigwedge_{j\in I\setminus J}\frac{dz_j}{z_j}$ and the sign depends on the order of $J$. In fact, from the calculation, we see that the image does not have any pole along $D^J$, so it is contained in the subsheaf consisting of classes represented by $\Omega^{n-r}_{Y^J}\otimes \sD_X$. This means that the class of~\eqref{eq:resimage} in $\tau^J_+\omega_{Y^J}(* D^J)(-r)$ is also contained in the image of the inclusion
\[
\begin{aligned}
  \tau^J_+\omega_{Y^J}(-r) &\to \tau^J_+\omega_{Y^J}(* D^J)(-r), \\
   [{dz_{\overline J}} \wedge dz_{k+1}\wedge\cdots dz_n \otimes P] &\mapsto [\frac{dz_{\overline J}}{z_{\overline J}} \wedge dz_{k+1}\wedge\cdots dz_n \otimes z_{\overline J}P].
\end{aligned}
\]
See Example~\ref{starex}. It follows that there is a morphism $\rho_r$ of filtered $\sD_X$-modules making the following diagram commute. 
\[ 
\begin{tikzcd}%[sep=small]
\ker R^{r+1} \arrow[dashed]{rr}{\rho_r}\arrow[hook]{d} &  & \tau^{(r+1)}_+\omega_{Y^{(r+1)}}(-r)\arrow[hook]{d} \\
\cM \arrow{r}{\frac{dt}{t}\wedge} & \cH^0 \arrow{r}{\Res} & \bigoplus_{|J|=r+1}\tau^J_+\omega_{Y^J}(* D^J)(-r)
\end{tikzcd}
\]
For a local section $\zeta\otimes z_{ K}P$ where $z_{ K}=\prod_{i\in  K}z_i$ a monomial of degree $k-r+1$, representing a class in $\ker R^r$, its image under $\rho_r$ is indeed zero because $z_{K}$ annihilates all $\Omega^{n-r}_{Y^J}(\log D^J)$ for index subset $J$ such that $|J|=r+1$. This implies the morphism $\rho_r$ kills $\ker R^r$. The morphism $\rho_r$ also kills $R\ker R^{r+2}$, because by~\eqref{dttr} $\frac{dt}{t}\wedge$ vanishes on the image of $R$. Thus it factors through
\[
\phi_r\colon \cP_r=\frac{\ker R^{r+1}}{\ker R^r+R\ker R^{r+2}} \to \tau^{(r+1)}_+\omega_{Y^{(r+1)}}(-r).
\]
The morphism $\phi_r$ is filtered surjective because for $dz_{\bar J}\wedge dz_{k+1}\wedge\cdots \wedge dz_n\otimes P\in \Omega^{n-r}_{Y^J}\otimes F_{\ell}\sD_X$ representing a class in $F_\ell \tau^{J}_+\omega_{Y^{J}}(-r)$ with $|J|=r+1$, we can find a lifting class represented by $\zeta\otimes z_{\overline J} P$ in $F_\ell\ker R^{r+1}$. It follows that 
\[
cc(\cP_r)\geq cc(\tau^{(r+1)}_+\omega_{Y^{(r+1)}})=\sum_{|J|=r+1} \left[T^*_{Y^J}X\right].
\]
Summing up the inequalities gives
\[
\sum_{r\geq 0}(r+1)cc(\cP_r) \geq \sum_{r\geq 0}(r+1)\sum_{|J|=r+1} \left[T^*_{Y^J}X \right]=\sum_{J\subset I} |J|\cdot \left[T^*_{Y^J}X\right].
\]
On the other hand, by the Lefschetz decomposition and Theorem~\ref{cycle}, we have
\[
\sum_{J\subset I} |J|\cdot\left[T^*_{Y^J}X\right]=cc(\cM)=cc(\gr^W\cM)=\sum_{r\geq 0} (r+1)cc(\cP_r).
\]
It follows that all inequalities must be equalities, i.e. $cc(\cP_r)=cc(\tau^{(r+1)}_+\omega_{\tilde Y^{(r+1)}})$ and therefore, that $\phi_r$ is a filtered isomorphism~\cite[Proposition 3.1.2]{HTT}.
\end{proof}

\section{Reduced case: Sesquilinear pairing on $\cM$ and limiting mixed Hodge structure}\label{sec:redses}
  \subsection{Sesquilinear pairing} 
We begin to construct the last data we need for the limiting mixed Hodge structure -- Sesquilinear pairing. In the sense that $\cM$ is the specialization of ${i_{X_t}}_+\omega_{X_t}$ for $t\neq 0$, the sesquilinear $S\colon\cM\otimes_\C \overline{\cM}\to \mathfrak{C}_X$ should also be the specialization of ${i_{X_t}}_+S_{X_t}$, where $S_{X_t}$ is defined in $\S$\ref{sec:pre}. Presumably one would expect that the pairing 
\[
\begin{aligned}
\langle S(\left[\zeta_1\otimes P_1 \right], \overline{\left[\zeta_2\otimes P_2\right]}),\eta \rangle &=  \lim_{t\to 0}\langle {i_{X_t}}_+S_{X_t}(\zeta_1\otimes P_1,\overline{\zeta_2\otimes P_2)},\eta\rangle \\
&= \lim_{t\to 0}  \frac{\varepsilon(n+1)}{(2\pi \sqrt{-1})^{n}} \int_{X_t}  P_1\overline{P_2}\eta  \wedge \zeta_1\wedge \overline{\zeta_2}
\end{aligned} 
\] 
should work on $\cM$ for $\zeta_i\otimes P_i$, $i=1,2$ sections of $\Omega^n_{X/\Delta}\otimes \sD_X$ over local chart $U$ representing classes of $\cM$, and $\eta$ is a test function over $U$. But one could check that the integral $\int_{X_t} P_1\overline{P_2}\left(\eta\right)   \zeta_1\overline{\wedge \zeta_2}$ could have order $(-\log |t|^{2})^{k}$ near the origin where $k+1$ is the number of components that intersect in $U$, so the limit may not exist. To avoid the issue, we use a Mellin transform device following the ideas in~\cite{SC02}: locally,
\[
\begin{aligned}
&\langle S(\left[\zeta_1\otimes P_1\right],\overline{\left[\zeta_2\otimes P_2\right]}),\eta \rangle \\
&\colon = \Res_{s=0}\frac{\varepsilon(n+2)}{(2\pi \sqrt{-1})^{n+1}} \int_X |t|^{2s} P_1\overline{P_2}\eta  \frac{dt}{t}\wedge \zeta_1\wedge\overline{\frac{dt}{t}\wedge \zeta_2} \\
&=\Res_{s=0} \frac{\varepsilon(2)}{2\pi\sqrt{-1}}\int_\Delta |t|^{2s}\frac{dt}{t}\wedge \overline{\frac{dt}{t}} \left(  \frac{\varepsilon(n+1)}{(2\pi \sqrt{-1})^{n}} \int_{X_t}  P_1\overline{P_2}\eta  \wedge \zeta_1\wedge \overline{\zeta_2} \right) \\
&= \Res_{s=0}  \frac{\varepsilon(2)}{2\pi\sqrt{-1}}\int_\Delta |t|^{2s}\frac{dt}{t}\wedge \overline{\frac{dt}{t}}  \langle {i_{X_t}}_+S_{X_t}(\zeta_1\otimes P_1,\overline{\zeta_2\otimes P_2)},\eta\rangle.
\end{aligned}
\]
The last expression in the definition to some extent explains that $S$ is the specialization of ${i_{X_t}}_+S_{X_t}$ and the $0$-current $\Res_{s=0}  \frac{\varepsilon(2)}{2\pi\sqrt{-1}}\int_\Delta |t|^{2s}\frac{dt}{t}\wedge \overline{\frac{dt}{t}}$ is doing the job of renormalization of $\ {i_{X_t}}_+S_{X_t}$ for $t\neq 0$. In fact, for any test function $g$ on $\Delta$, we have
\[
\Res_{s=0}  \frac{\varepsilon(2)}{2\pi\sqrt{-1}}\int_\Delta |t|^{2s}\frac{dt}{t}\wedge \overline{\frac{dt}{t}} g=g(0).
\]
We have not checked that $S$ is well-defined, but let us do an example to see how the Mellin transform works.

\begin{ex}
Suppose $Y$ that is smooth. Then $R$ is identical zero and $\cM\simeq {i_Y}_+\omega_Y$, by Theorem~\ref{iden}. Thus, the pairing $S$ should recover the natural pairing $S_{Y}$. In local coordinates $t=z_0$ and for any local sections $\zeta_i\otimes P_i=dz_1\wedge dz_2\wedge\cdots\wedge dz_n\otimes P_i$ of $\Omega^n_{X/\Delta}(\log Y)\otimes\sD_X$, $i=1,2$ over local chart $U$, 
\[
\begin{aligned}
&\langle S(\left[\zeta_1\otimes P_1\right], \overline{\left[\zeta_2\otimes P_2\right])},\eta\rangle  \\
&=\Res_{s=0}\frac{\varepsilon(n+2)}{(2\pi\sqrt{-1})^{n+1}}\int_X |t|^{2s} P_1\overline{P_2}(\eta) \frac{dt}{t}\wedge \zeta_1\wedge\overline{\frac{dt}{t}\wedge \zeta_2} \\
&=\Res_{s=0} \int_X |t|^{2s-2}  P_1\overline{P_2}(\eta) \bigwedge^n_{i=0} \frac{\sqrt{-1}}{2\pi} dz_i\wedge \overline{dz_i} \\
& \quad \text{ integration by parts on }t\text{ and }\bar t \\
&= \Res_{s=0}  \int_{X} \frac{|t|^{2s}}{s^2} \partial_0\overline{\partial_0} \left( P_1\overline{P_2}(\eta) \right)  \bigwedge^n_{i=0} \frac{\sqrt{-1}}{2\pi} dz_i\wedge \overline{dz_i} .
\end{aligned}
\]
Because the Laurent expansion of $s^{-2}|t|^{2s}$ is $\sum^\infty_{\ell=0} \left(\log |t|^2 \right)^\ell s^{\ell-2}$, the above  continuously equals to, by Poincar\'e-Lelong equation~\cite[Page 388]{GH}
\[
\begin{aligned}
&\int_{X} \log |t|^2 \partial_0\overline{\partial_0}\left( P_1\overline{P_2}(\eta) \right) \bigwedge^n_{i=0} \frac{\sqrt{-1}}{2\pi} dz_i\wedge \overline{dz_i} \\
&= \int_{Y}  P_1\overline{P_2}(\eta)    \bigwedge^n_{i=1} \frac{\sqrt{-1}}{2\pi} dz_i\wedge \overline{dz_i} \\
&=\frac{\varepsilon(n+1)}{(2\pi\sqrt{1})^n}\int_Y P_1\overline{P_2}(\eta) \zeta_1\wedge\overline{\zeta_2} \\
&=\langle {i_Y}_+S_Y([\zeta_1\otimes P_1],\overline{[\zeta_2\otimes P_2]}),\eta\rangle .
\end{aligned}
\]
We can take a cleaner point of view. In the case $Y$ is smooth, the form $P_1\overline{P_2}(\eta)   \zeta_1\overline{\wedge \zeta_2}$ is smooth in the neighborhood of $Y$. It follows that  ${i_{X_t}}_+S_{X_t}$ extends smoothly to $t=0$ and the limit of ${i_{X_t}}_+S_{X_t}$ is exactly ${i_Y}_+S_Y$.
\end{ex}

When $Y$ has several smooth irreducible components, the idea of computation is similar to the above. Now we begin to establish the statements needed to ensure $S$ is well-defined. For any test function $\eta$ over an arbitrary open subset $U$ of $X$ and two sections $m_1,m_2 \in \bH^0\left(U,\Omega^n_{X/\Delta}(\log Y)\otimes\sD_X\right)$, the $(2n+2)$-form $\frac{dt}{t}\wedge m_1 \wedge\overline{\frac{dt}{t}\wedge m_2}(\eta)$ is smooth away from $Y$ but with poles along $Y$ supported in $U$. Locally, say $m_i=\zeta \otimes P_i$ for $\zeta=\frac{dz_1}{z_1}\wedge \frac{dz_2}{z_2}\wedge\cdots \frac{dz_k}{z_k}\wedge dz_{k+1}\wedge\cdots \wedge dz_n$ and $i=1,2$, the $(2n+2)$-form $\frac{dt}{t}\wedge m_1 \wedge\overline{\frac{dt}{t}\wedge m_2}(\eta)$ is just $P_1\overline {P_2}(\eta)\frac{dt}{t}\wedge\zeta \wedge \overline{ \frac{dt}{t}\wedge \zeta}$. Let  $F(s)=F(s,m_1,m_2,\eta)$ be the meromorphic continuation via integration by parts of the following function
\[
\frac{\varepsilon(n+2)}{(2\pi\sqrt{-1})^{n+1}}\int_X |t|^{2s} \frac{dt}{t}\wedge m_1 \wedge\overline{\frac{dt}{t}\wedge m_2}(\eta).
\]
The function $F(s)$ is holomorphic when $\mathrm{Re\,} s>0$ and has potential poles at non-positive integers. Note that $F(s)$ is independent of local coordinates. We are only interested in the polar part of the function $F(s)$ at $s=0$. 

\begin{thm}
The polar part of $F(s)$ at $s=0$ only depends on the classes of $m_1$ and $m_2$ in $\cM$.
\end{thm}

\begin{proof}
Let $\left \{\rho_\lambda\right\}$ be a partition of unity of an open covering $\left\{U_\lambda\right\}$ by local charts. Then 
\[
F(s)=\sum_{\lambda} \frac{\varepsilon(n+2)}{(2\pi\sqrt{-1})^{n+1}}\int_{U_\lambda} |t|^{2s} \frac{dt}{t}\wedge m_1 \wedge\overline{\frac{dt}{t}\wedge m_2}(\rho_\lambda\eta).
\]
Since $\rho_\lambda \eta$ is a test function over $U_\lambda$, without loss of generality, we can assume $U$ itself is a local chart. It follows that we can assume that $m_i=\zeta\otimes P_i$ for $i=1,2$ and $\zeta=\frac{dz_1}{z_1}\wedge \frac{dz_2}{z_2}\wedge\cdots\frac{dz_k}{z_k}\wedge dz_{k+1}\wedge\cdots\wedge dz_n$. We begin with some properties of $F(s)$.

\begin{lem} %\label{func}
Under the assumption that $m_i=\zeta \otimes P_i$ for $\zeta=\frac{dz_1}{z_1}\wedge\frac{dz_2}{z_2}\wedge\cdots\wedge\frac{dz_k}{z_k}\wedge\cdots\wedge dz_n$ and for $i=1,2$, the followings are valid.
\begin{enumerate}
  \item the order of the pole of $F(s)$ at $s=0$ is at most $k+1$;
  \item if $P_i=tP'_i$ for one of $i=1,2$, then $F(s)$ is holomorphic at $s=0$;
  \item for $0\leq j\leq k$ we have, 
\[
F(s,\zeta_1\otimes P_1, \overline{\zeta_2\otimes {z_j\partial_j}P_2},\eta)= F(s, \zeta_1\otimes z_j\partial_jP_1,\overline{\zeta_2\otimes P_2},\eta) = -sF(s, \zeta_1\otimes P_1,\overline{\zeta_2\otimes P_2},\eta).
\] 
\end{enumerate}
\end{lem}

\begin{proof}[Proof of the lemma] 
The Laurent expansion of $F(s)$ at $s=0$ is
\[
\begin{aligned}
F(s) &=  \int_{X}   |z_I|^{2s-2}  P_1\overline{P_2}(\eta) \bigwedge^n_{i=0}(\frac{\sqrt{-1}}{2\pi}dz_i\wedge{d\overline{z_i}}),\quad \text{where } z_I=\prod_{i\in I}z_i \\
  &=  \int_{X}  \frac{|z_I|^{2s}}{s^{2k+2}} \partial_I\overline{\partial_I}  P_1\overline{P_2}(\eta) \bigwedge^n_{i=0}(\frac{\sqrt{-1}}{2\pi}dz_i\wedge{d\overline{z_i}}),\quad \text{where }\partial_I=\prod^k_{i=0}\partial_i \\
  &=  \sum^{\infty}_{\ell=0}\frac{s^{\ell-(2k+2)}}{\ell!}\int_{X}  \left( \log |z_I|^2 \right)^{\ell}\partial_I\overline{\partial_I} P_1\overline{P_2}(\eta) \bigwedge^n_{i=0}(\frac{\sqrt{-1}}{2\pi}dz_i\wedge{d\overline{z_i}}).
\end{aligned}
\]
The order of the pole at $s=0$ is at most $k+1$: if $\ell<k+1$, the form 
\[
\left(\log |z_I|^2 \right)^{\ell}\partial_I\overline{\partial_I}P_1\overline{P_2}(\eta) \bigwedge^n_{i=0}(\frac{\sqrt{-1}}{2\pi}dz_i\wedge{d\overline{z_i}})
\]
is actually exact because one of $a_i$'s must be $0$ in the expansion of $\left( \log |z_I|^2 \right)^{\ell}$ into a linear combination of $\prod^k_{i=0}\left(\log |z_i|^2 \right)^{a_i}$ with $\sum^k_{i=0}a_i=\ell<k+1$. This proves $(1)$.

Suppose that $P_1=tP'_1$. Then the function
\[
\begin{aligned}
F(s) &=  \int_{X}   |z_I|^{2s-2}  tP'_1\overline{P_2}(\eta) \bigwedge^n_{i=0}(\frac{\sqrt{-1}}{2\pi}dz_i\wedge{d\overline{z_i}}). \\
\end{aligned}
\]
is well-defined at $s=0$ because the form 
\[
 \frac{1}{\overline{z_I}} P'_1\overline{P_2}(\eta) \bigwedge^n_{i=0}(\frac{\sqrt{-1}}{2\pi}dz_i\wedge{\overline{dz_i}})
\]
is integrable. The same argument works for the case when $P_2=tP'_2$. This proves $(2)$.

Now we turn to the last statement
\[
\begin{aligned}
& F(s,\zeta\otimes P_1, \overline{\zeta\otimes{z_j\partial_j}P_2},\eta) \\
=& \frac{\varepsilon(n+2)}{(2\pi\sqrt{-1})^{n+1}}\int_X |t|^{2s}\overline{z_j\partial_j}(P_1\overline{P_2}\eta)\frac{dz_0}{z_0}\wedge\frac{dz_1}{z_1}\wedge\cdots\wedge dz_n\wedge \overline{\frac{dz_0}{z_0}\wedge\frac{dz_1}{z_1}\wedge\cdots\wedge dz_n} \\
  =&\int_X  |z_{I\setminus \{j\}}|^{2s-2}z^{s-1}_j\overline{z^{s}_j \partial_j}P_1\overline{P_2}\eta\bigwedge^n_{i=0}(\frac{\sqrt{-1}}{2\pi}dz_i\wedge{\overline{dz_i}}) \\
  =&-s\int_X   |z_I|^{2s-2} P_1\overline{P_2}\eta \bigwedge^n_{i=0}(\frac{\sqrt{-1}}{2\pi}dz_i\wedge{d\overline{z_i}})  \quad \text{by integration by part on }  dz_j \\
  =&-sF(s,\zeta\otimes P_1,\overline{\zeta\otimes P_2},\eta).
\end{aligned}
\]
The same argument works for $F(s,\zeta\otimes z_j\partial_j P_1, \overline{\zeta\otimes P_2},\eta) = -sF(s,\zeta\otimes P_1,\overline{\zeta\otimes P_2},\eta).$ This proves $(3)$.
\end{proof}
Returning to the proof of the theorem, if one of $\zeta\otimes P_i$ is $\frac{dz_1}{z_1}\wedge\frac{dz_2}{dz_2}\wedge\cdots\frac{dz_k}{z_k}\wedge dz_{k+1 }\wedge\cdots\wedge dz_n\otimes tP'_i$, the above lemma $(2)$ says $F(s)$ is holomorphic. Otherwise, one of $\zeta \otimes P_i$ is $\frac{dz_1}{z_1}\wedge\frac{dz_2}{dz_2}\wedge\cdots\frac{dz_k}{z_k}\wedge dz_{k+1}\wedge\cdots\wedge dz_n\otimes D_iP$, then the above lemma $(3)$ says $F(s)$ is $0$.
\end{proof}

For any sections $\alpha,\beta \in\cM$, let $\left \{\rho_\lambda\right\}$ be a partition of unity of the open covering $\left\{U_\lambda\right\}$ by local charts such that $\alpha, \beta$ lifts to $\tilde \alpha_\lambda,\tilde \beta_\lambda$  over $U_\lambda$ in $\Omega^n_{X/\Delta}(\log Y)\otimes \sD_X$. The above theorem guarantees that the pairing $S\colon\cM\otimes_\C\overline\cM\to \mathfrak C_X$ defined by
\[
\langle S(\alpha, \bar\beta),\eta\rangle \colon =\Res_{s=0}\sum_\lambda F(s, \tilde \alpha_\lambda, \overline{\tilde \beta_\lambda} ,\rho_\lambda\eta)
\]
is well-defined and does not depend on the choice of a partition of unity. The following is an immediate corollary of the lemma.
\begin{cor}\label{cor:sr}
The operator $R$ is self-adjoint with respect to $S$, i.e. $S\circ (R\otimes_\C \id) =S\circ (\id\otimes_\C \bar R)$.
\end{cor}

The self-adjointness implies there are induced pairings on the graded quotient $S_r\colon \gr^W_r\cM\otimes_\C\overline{\gr^W_{-r}\cM}\to \mathfrak C_X$ for every integer $r$. Denote by $P_R S_r$ the pairing 
\[
S_r\circ (\id\otimes_\C R^{r})\colon \cP_r\otimes_\C \overline{\cP_r}\to \mathfrak C_X.
\]

\begin{thm}\label{main}
For any local sections $\alpha,\beta\in \cP_r$, we have 
\[
P_R S_r(\alpha, \bar \beta)= \frac{(-1)^r}{(r+1)!}\tau^{(r+1)}_+S_{{\tilde Y^{(r+1)}}} (\phi_r \alpha,\overline{\phi_r\beta})
\]
where $\phi_r \colon (\cP_r, F_\bullet \cP_r)\to \tau^{(r+1)}_+\omega_{\tilde Y^{(r+1)}}(-r)$ is the isomorphism in Theorem~\ref{iden}.
\end{thm}

\begin{proof}
Because the problem is local, it suffices to prove the theorem for $\alpha$ and $\beta$ are represented by 
\[
\frac{dz_1}{z_1}\wedge\frac{dz_2}{z_2}\wedge\cdots\frac{dz_k}{z_k}\wedge dz_{k+1} \wedge\cdots\wedge dz_n\otimes z_{K_i}
\]
and $|K_i|=k-r$ for $i=1,2$ over a local chart $U$ respectively. Recall that $z_K=\prod_{j\in K}z_j$. Let $\eta$ be a test function over $U$. We have 
\[
\begin{aligned}
  \langle P_R S_r(\alpha, \bar \beta),\eta\rangle &=\langle S(\alpha, \overline{R^{r}\beta}),\eta\rangle \\
   &=\Res_{s=0}(-s)^{r}\int_X   |z_I|^{2s-2} z_{K_1}\overline{z_{K_2}}(\eta) \bigwedge^n_{i=0}(\frac{\sqrt{-1}}{2\pi}dz_i\wedge{d\overline{z_i}}).
\end{aligned}
\]
If $\alpha\neq\beta$, the above is zero. Indeed, for $v \in K_2\setminus K_1$, by choosing $R^r=\prod_{i\in I\setminus {K_1\setminus \{v\}}}{z_i\partial_i}$,
\[
\begin{aligned}
\langle P_R S_r(\alpha, \bar \beta),\eta\rangle  &=\langle S(R^r\alpha,\bar \beta),\eta\rangle =\Res_{s=0} \int_X   |z_I|^{2s-2} \frac{t}{z_{v}} \overline{z_{v}} \tilde\eta\bigwedge^n_{i=0}(\frac{\sqrt{-1}}{2\pi}dz_i\wedge{d\overline{z_i}}), \\
\end{aligned}
\] 
where $\tilde \eta= \partial_{I\setminus (K_1\setminus \{v\})} {\overline{z_{K_2}}}{\left(\overline{z_{v}}\right)^{-1}}\eta$ is a smooth test function. The function 
\[
 \int_X  |z_I|^{2s-2} \frac{t}{z_{v}} \overline{z_{v}} \tilde\eta\bigwedge^n_{i=0}(\frac{\sqrt{-1}}{2\pi}dz_i\wedge{d\overline{z_i}})
\]
is holomorphic at $s=0$ because 
\[
 \frac{1}{\overline{z_I}} \frac{\overline{z_{v}}}{z_{v}}\tilde\eta\bigwedge^n_{i=0}(\frac{\sqrt{-1}}{2\pi}dz_i\wedge{d\overline{z_i}})
\]
is integrable. Therefore, we have reduced the proof to the case when both $\alpha$ and $\beta$ are represented by 
\[
\frac{dz_1}{z_1}\wedge\frac{dz_2}{z_2}\wedge\cdots\wedge\frac{dz_k}{z_k}\wedge\cdots\wedge dz_n\otimes z_{K}.
\]
We shall prove that, if $\overline K$ is the complement of $K$ in $I$
\[
P_R S_r(\alpha, \overline \alpha)=\frac{(-1)^r}{(r+1)!}\tau^{\overline K}_+S_{Y^{\overline K}}(\phi_r \alpha, \overline{\phi_r\alpha}).
\]
Without loss of generality, we can assume that $K=\{r+1,r+2,\dots,k\}$. Let $\partial_{\bar K}=\prod^{r}_{i=0}\partial_i$. Then 
\[
\begin{aligned}
&P_R S_r(\alpha,\overline{\alpha})= \\
&= \Res_{s=0} (-s)^{r}\int_X |z_{\overline K}|^{2s-2} \prod^k_{j=r+1}|z_j|^{2s} \eta \bigwedge^n_{i=0}(\frac{\sqrt{-1}}{2\pi}dz_i\wedge{d\overline{z_i}}) \\ 
  &=  (-1)^r\Res_{s=0} s^{-(r+2)}\int_X  \prod^{k}_{i=0} |z_i|^{2s} \partial_{\bar K}\overline{\partial_{\bar K}}(\eta) \bigwedge^n_{i=0}(\frac{\sqrt{-1}}{2\pi}dz_i\wedge{d\overline{z_i}}),   \\
  &= \frac{(-1)^r}{(r+1)!} \int_X  \left(\log\prod^{k}_{i=0} |z_i|^2\right)^{r+1}\partial_{\bar K}\overline{\partial_{\bar K}}(\eta) \bigwedge^n_{i=0}(\frac{\sqrt{-1}}{2\pi}dz_i\wedge{d\overline{z_i}}) \\
   &=  \frac{(-1)^r}{(r+1)!} \int_X  \prod^{r}_{i=0} \log |z_i|^2 \partial_{\bar K}\overline{\partial_{\bar K}}(\eta) \bigwedge^n_{i=0}(\frac{\sqrt{-1}}{2\pi}dz_i\wedge{d\overline{z_i}})  \quad  (\star) \\
  &= \frac{(-1)^r}{(r+1)!}\int_{Y^{\overline K}} \eta \bigwedge^n_{i=r+1}(\frac{\sqrt{-1}}{2\pi}dz_i\wedge{d\overline{z_i}})   \\
  & \quad \quad \text{(Poincar\'e-Lelong equation~\cite[Page 388]{GH})} \\
  &= \frac{(-1)^r}{(r+1)!}\tau^{{\overline K}}_+S_{Y^{\overline K}} (\Res_{Y^{\overline K}}\frac{dt}{t}\wedge\alpha, \overline{\Res_{Y^{\overline K}}\frac{dt}{t}\wedge\alpha}).
\end{aligned}
\]
The equality $(\star)$ holds because if we expand $\left(\log\prod^{k}_{i=0} |z_i|^2\right)^{r+1}$ as a linear combination of $\prod^k_{i=0}\left(\log |z_i|^2 \right)^{a_i}$ with $\sum^k_{i=0}a_i=r+1$, the only possible non-exact form among
\[
  \prod^{k}_{i=0} \left(\log |z_i|^2\right)^{a_i}\partial_{\bar K}\overline{\partial_{\bar K}}(\eta) \bigwedge^n_{i=0}(\frac{\sqrt{-1}}{2\pi}dz_i\wedge{d\overline{z_i}}),
\] 
is $\left( \prod^{r}_{i=0} \log |z_i|^2 \right)\partial_{\bar K}\overline{\partial_{\bar K}}(\eta) \bigwedge^n_{i=0}(\frac{\sqrt{-1}}{2\pi}dz_i\wedge{d\overline{z_i}})$. Note that while $\Res_{Y^{\overline{K}}}$ depends on the order of the index sets $K$ and $I$, the pairing  
\[
\frac{(-1)^r}{(r+1)!}\tau^{(r+1)}_+S_{{\tilde Y^{(r+1)}}} (\phi_r \alpha,\overline{\phi_r\beta})= \frac{(-1)^r}{(r+1)!}\tau^{{\overline K}}_+S_{Y^{\overline K}} (\Res_{Y^{\overline K}}\frac{dt}{t}\wedge\alpha, \overline{\Res_{Y^{\overline K}}\frac{dt}{t}\wedge\alpha})
\]
does not because the sign cancels out. We complete the proof.
\end{proof}

\subsection{Construction of the limiting mixed Hodge structure}
We are going to show that the triple $\left(\DR_X\cM, F, W\right)$ gives a mixed Hodge complex. Unlike the $\Q$-mixed Hodge complex considered by Deligne~\cite{Hodge2}, where the rational structure is a required input, we do not have this piece of information in our situation. We will redo Deligne's argument on mixed Hodge complex by sesquilinear pairings. It is also worth noting that the sesquilinear pairing simplifies the process of verifying that the first page weight spectral sequence of $\DR_X\cM$ is a polarized bigraded Hodge-Lefschetz structure, making it easier than the situation in \cite{hl}, where they need to decompose the differential $d_1$ on the first page into a combinatorial differential and a sum of Gysin morphisms. 

We first set up the pairing on each page of the weight spectral sequence abstractly. Let $\cN$ be a holonomic $\sD_Z$-module with compact support equipped with a sesquilinear pairing $S\colon\cN\otimes_\C \overline{\cN}\to \mathfrak{C}_Z$ on a complex manifold $Z$. Let $N$ be a nilpotent operator on $\cN$ such that $S\circ (\id\otimes_\C \bar N)=S\circ(N\otimes_\C \id)$. Let $W_\bullet\cN$ be the monodromy filtration associated to $N$ on $\cN$. Let $E^{i,j}_r$ be the weight spectral sequence convergent to $\gr^W_{-i}\bH^{i+j}(Z,\DR_Z\cN)$ with $E^{i,j}_1=\bH^{i+j}(Z,\gr^W_{-i}\DR_Z\cN)$. By abuse of notation, denote by $S_k=\varepsilon(k)S_k'$ where
\[
\begin{aligned}
  S_k' &\colon \bH^k(Z,\DR_Z\cN) \otimes_\C \overline{\bH^{-k}(Z,\DR_Z\cN)} \to \\
  &\bH^0(Z,\DR_{Z,\overline Z}\cN\otimes_\C\overline \cN) \to \bH^0_c(Z,\DR_{Z,\overline Z}\mathfrak C_Z)\simeq \C.
\end{aligned}
\]
Let $a$ be a local section of $\left(\DR_Z\cN\right)^{-j-1}$ and $b$ be a local section of $\left(\DR_Z\cN\right)^i$. Then 
\[
D(a\otimes_\C \bar b)= da\otimes_\C \bar b+(-1)^{-j-1}a\otimes_\C \overline{db}
\]
for $D$ a differential on $\DR_{Z,\overline Z}\cN\otimes_\C \overline \cN$. Applying $S$, we find that 
\begin{equation}\label{eq:dsign}
DS(a, b)=S(da, b)+(-1)^{-j-1} S(a, \overline{db}).
\end{equation}
Since the differential $d$ is compatible with the weight filtration, we have an induced pairing $E_1(S)_k=\varepsilon(k) E_1(S')_k$ on the first page $E^{i,j}_1$ of the weight spectral sequence, where
\[
\begin{aligned}
  E_1(S')_k &\colon \bH^k(Z,\gr^W_{-i}\DR_Z\cN) \otimes_\C \overline{\bH^{-k}(Z,\gr^W_{i}\DR_Z\cN)} \to \\
   &\bH^0(Z, \DR_{Z,\overline Z}\gr^W_{-i}\cN\otimes_\C \overline{ \gr^W_i \cN}) \to \bH^0(Z,\DR_{Z,\overline Z}\mathfrak C_Z)\simeq \C.
\end{aligned}
\]
Then by equation~\eqref{eq:dsign} we obtain 
\[
0=\varepsilon(-j)E_1(S)_{-j}(d_1a, \bar b)+\varepsilon(-j-1)(-1)^{-j-1}E_1(S)_{-j-1}(a,\overline{d_1b}),
\]
due to $DS(a,\bar b)$ is cohomologous to zero. Working out the sign, we find that  
\[
E_1(S)_{-j}(d_1a,\bar b)+E_1(S)_{-j-1}(a,\overline{d_1b})=0,
\]
i.e. the differential $d_1$ is skew-symmetric with respect to $E_1(S)$. It follows that we have an induced pairing on the second page: $E_2(S)_k: E^{i,k-i}_2\otimes \overline{ E^{-i,-k+i}_2}\to \C$ since $E_2=\ker d_1/\mathrm{Im\,} d_1$. It also follows from that the equation~\eqref{eq:dsign}, the differential $d_2$ is skew-symmetric with respect to $E_2(S)$. By an inductive argument, we get the induced pairing $E_r(S): E_r\otimes \overline{E_r}\to \C$ on the $r$-th page of the weight spectral sequence $E_r\otimes \overline{E_r}\to \C$ such that $d_r$ is skew-symmetric with respect to $E_r(S)$ for every $r\geq 1$.

Next, let $L=[\omega]\wedge$ be a Lefschetz operator for a K\"ahler class $[\omega]\in H^1(Z,\Omega_Z)\cap H^2(Z,\R)$ and $\sfX=2\pi\sqrt{-1}L$ on $Z$. We can also regard $\sfX$ as a morphism $\DR_Z\cN\to \DR_Z\cN[2]$. Let us work out the relation between the sesquilinear pairing $S_k$ and the operator $\sfX$. By functoriality, we have the following commutative diagram in $\bD^b(Z,\C)$.
\[
\begin{tikzcd}[sep=small]
\DR_{Z,\overline Z}\cN\otimes_\C \overline \cN \arrow{r}{S}\arrow{d}{\sfX\otimes_\C \id} & \DR_{Z,\overline Z}\mathfrak C_Z \arrow{r}{\simeq}\arrow{d}{\sfX} & \DR_{Z,\overline Z}\mathfrak{Db}_Z \arrow{r}{\simeq}\arrow{d}{\sfX} & \cA^{\bullet+2n}_Z\otimes\mathfrak{Db}_Z \arrow{d}{\sfX}  \\
\DR_{Z,\overline Z}\cN\otimes_\C \overline \cN \left[2\right] \arrow{r}{S[2]} & \DR_{Z,\overline Z}\mathfrak C_Z \left[2\right] \arrow{r}{\simeq} & \DR_{Z,\overline Z}\mathfrak{Db}_Z\left[2\right] \arrow{r}{\simeq} &  \cA^{\bullet+2n+2}_Z\otimes\mathfrak{Db}_Z 
\end{tikzcd}
\]
Similarly, we have $S[2]\circ (\id \otimes_\C \overline{\sfX})={\overline\sfX} S$. It follows from $\sfX+\overline \sfX=0$ on $\cA^{\bullet+2n}_Z\otimes\mathfrak{Db}$ that 
\begin{equation}\label{eq:sx}
\varepsilon(k)S_k(\sfX-,-)+\varepsilon(k-2)S_{k-2}(-, \bar\sfX-)=0,\quad \text{i.e. }S_k(\sfX-,-)=S_{k-2}(-, \bar \sfX-).
\end{equation}

Returning to our situation, we begin to construct a polarized bigraded Hodge-Lefschetz structure on 
\[
\gr^W \bH^\bullet(X,\DR_X\cM).
\]
Fix a K\"ahler class $[\omega]$ on $X$ and let $L=[\omega]\wedge\colon\DR_X\cM\to\DR_X\cM[2]$ be the Lefschetz operator and $\sfX_1=2\pi\sqrt{-1}L$ as the discussion above. Relabel the first page of the weight spectral sequence by
\[
V_{\ell,k}=\bH^{\ell}(X,\gr^W_k\DR_X\cM)=\prescript{W}{}{E^{-k,\ell+k}_1}.
\]
Let $V=\bigoplus_{\ell,k\in\Z}V_{\ell,k}$ with filtration $F_\bullet V$ induced by $F_\bullet\cM$. Denote by $E_i(R)$ the induced operator by $R$ on $\prescript{W}{}{E_i}$ and let $\sfY_2=E_1(R)$. Let $S_{\ell,k}=\varepsilon(\ell)S_{\ell,k}'$ for $\ell,k\in\Z$ be the induced pairing on $V_{\ell,k}\otimes \overline{V_{-\ell,-k}}$  where 
\[
\begin{aligned}
  S_{\ell,k}' &\colon \bH^{\ell}(X,\gr^W_k\DR_X\cM)\otimes \overline {\bH^{-\ell}(X,\gr^W_{-k}\DR_X\cM}) \to \\
  & \bH^0(X,\DR_{X,\overline X}\gr^W_k\cM\otimes_\C \overline{\gr^W_{-k}\cM}) \to \bH^0_c(X, \DR_{X,\overline X}\mathfrak {C}_X) \simeq \C.
\end{aligned}
\]
Let $d_1$ be the differential of $E_1$. In terms of relabeling, we have
\[
d_1\colon (V_{\ell,k},F_\bullet V_{\ell,k}) \to (V_{\ell+1,k-1},F_\bullet V_{\ell+1,k-1}).
\] 

\begin{thm}\label{rd1pbhl}
The tuple $(V,\sfX_1,\sfY_2,F_\bullet V, \bigoplus S_{j,k})$ is a differential polarized bigraded Hodge-Lefschetz structure of central weight $n$ and $d_1$ is a differential of the bigraded Hodge-Lefschetz structure. 
\end{thm}
\begin{proof}
Let us first check the conditions in Theorem~\ref{bphl} one by one. It is clear that two operators $\sfX_1,\sfY_2$ are commute. Moreover, we have $\sfY_2\colon(V_{\ell,k}, F_\bullet V_{\ell,k})\to (V_{\ell,k-2}, F_{\bullet+1} V_{\ell,k-2})$ such that
\[
\sfY^k_2\colon F_\bullet V_{\ell,k}\to F_{\bullet+k} V_{\ell,-k},
\]
is an isomorphism by Theorem~\ref{FW}. Denote by $P_{\sfY_2}V_{-j,r}$ the ${\sfY_2}$-primitive part $\ker {\sfY^{r+1}_2}\cap V_{-j,r}= \bH^{-j}(X,\DR_X\cP_r)$. It follows from Theorem~\ref{iden} that 
\[
\phi_r\colon (P_{\sfY_2}V_{-j,r}, F_\bullet P_{\sfY_2}V_{-j,r}) \simeq H^{-j}\left(\tilde Y^{(r+1)},\DR_{\tilde Y^{(r+1)}}\omega_{{\tilde Y^{(r+1)}}}\right)(-r).
\]  
Therefore, $\sfX_1 \colon F_\bullet P_{\sfY_2}V_{-j,r}\to F_{\bullet-1}P_{\sfY_2}V_{-j+2,r}$ and by Hard Lefschetz,
\[
\sfX^{j}_1\colon F_\bullet P_{\sfY_2}V_{-j,r}\to F_{\bullet-j}P_{\sfY_2}V_{j,r}
\] 
is an isomorphism. It follows from the Lefschetz decomposition of $\sfY_2$ that 
\[
\sfX^j_1\colon F_\bullet V_{-j,r}\to F_{\bullet-j}V_{j,r}
\] 
is an isomorphism. This proves $(\mathrm{pbHL1})$ in Theorem~\ref{bphl}. And (pbHL2) for $\sfX_1$ follows from the equation~\eqref{eq:sx}.

Because the operator $R$ self-adjoint with respect to $S$ by Corollary~\ref{cor:sr}, we have $S_{j,r} (-, \overline{\sfY_2}-)=S_{j,r+2}(\sfY_2-,-)$. By Theorem~\ref{main}, the morphism $\phi_r$ identifies $P_{\sfY_2}S_{-j,r}\colon =S_{-j,r}(-,\overline{\sfY^r_2}-)$ with $\frac{(-1)^r}{(r+1)!}S_{\tilde Y^{(r+1)},-j}$. Recall that 
\[
S_{\tilde Y^{(r+1)},j}(a, \bar b)={\frac{\varepsilon(n-r+j+1)}{(2\pi\sqrt{-1})^{n-r}}}\int_{\tilde Y^{(r+1)}} a\wedge \bar b,
\]
for $ a\in H^{n-r+j}(\tilde Y^{(r+1)})$ and $b\in H^{n-r-j}(\tilde Y^{(r+1)})$, and that $S_{\tilde Y^{(r+1)},j}(\sfX^j_1-,-)$ is a polarization on $H^{n-r-j}_{\mathrm{prim}}(\tilde Y^{(r+1)},\C)$. The bi-primitive part $P_{-j,r}=\ker \sfX^j_1\cap \ker \sfY^r_2\cap V_{-j,r}$ together with the induced filtration $F_{\bullet}P_{-j,r}$ and the pairing $S_{j,r}(\sfX^j_1-,(-\overline{\sfY_2})^r-)$ is identified with the polarized Hodge structure $H^{n-r-j}_{\mathrm{prim}}(\tilde Y^{(r+1)},\C)(-r)$ via $\phi_r$. This proves (pbHL3).

It remains to prove that $d_1$ is a differential of the bigraded Hodge-Lefschetz structure $V$. Clearly, we have 
\[
[d_1,\sfX_1]=[d_1,\sfY_2]=0
\]
because $d_1$ is induced by the differential of $\DR_X\cM$ and $d_1$ preserves $F_\bullet$. The differential $d_1$ is skew-symmetric with respect to $\bigoplus_{j,r}S_{j,r}$ formally follows the discussion at the beginning of this subsection. Thus, we finished checking that $d_1$ is a differential. 
\end{proof}

\begin{cor}\label{rmhc}
We have the following 
\begin{enumerate}
  \item the Hodge spectral sequence degenerates at $\prescript{}{F}{E_1}$, 
  \item the weight spectral sequence degenerates at $\prescript{W}{}{E_2}$, 
  \item The tuple $\left(\bigoplus_{\ell\in \Z}\gr^W\bH^\ell(X,\DR_X\cM),F,\sfX_1,\sfY_2\right)$ together with the pairing induced by $\bigoplus S_{j,k}$ is a polarized bigraded Hodge-Lefschetz structure of central weight $n$.
\end{enumerate}
\end{cor}
\begin{proof}
We slightly modify the idea of cohomological mixed Hodge complex in~\cite{Hodge2} for statements $(1)$ and $(2)$. Let $V^k_{\ell,r}\colon =\prescript{W}{}{E_k}^{-r,\ell+r}$ be the relabeling of the $k$-th page of the weight spectral sequence. I claim that $V^k_{\ell,r}$ together with the induced filtration $F_\bullet$ and the induced pairing $S^k_{\ell,r}\circ \left(\id\otimes \bar\sfw \right)\colon V^k_{\ell,r}\otimes \overline{V^k_{\ell,r}}\to \C$ is a polarized Hodge structure of weight $n+\ell+r$ and the differential $d_k\colon V^k_{\ell,r}\to V^k_{\ell+1,r-k}$ is a morphism of Hodge structures. Indeed, the differential $d_k$ is skew-symmetric with respect to the sesquilinear pairing, i.e. $S^k_{\ell,r}(d_k-,-)+S^k_{\ell+1,r-k}(-,\overline{d_k}-)=0$. Therefore, if $(-1)^{q} S^k_{\ell,r}\circ \left(\id\otimes \bar\sfw \right)$ for $q=n+\ell+r-p$ is a Hermitian inner product on 
\[
{\left(V^{k}_{\ell,r}\right)^{p,q}=\{a\in F^pV^k_{\ell,r}\colon S^k_{\ell,r}(a,\bar b)=0\text{ for all } b\in F^{p-\ell-r+1}V^k_{-\ell,-r}\}}
\]
then $(-1)^{q} S^{k+1}_{\ell,r}\circ \left(\id\otimes \overline\sfw\right)$ is also a Hermitian inner product on 
\[
\left(V^{k+1}_{\ell,r}\right)^{p,q}=\{a\in F^pV^{k+1}_{\ell,r}\colon S^{k+1}_{\ell,r}(a,\bar b)=0\text{ for all } b\in F^{p-\ell-r+1}V^{k+1}_{-\ell,-r}\}.
\]
In particular, we have the decomposition 
\[
V^{k+1}_{\ell,r}=\bigoplus_{p+q=n+\ell+r}\left(V^{k+1}_{\ell,r}\right)^{p,q}
\]
and the morphism $d_k\colon \left(V^{k}_{\ell,r}\right)^{p,q} \to \left(V^{k+1}_{\ell,r}\right)^{p,q}$ is compatible with the decomposition. See Remark~\ref{dmorphism}. By induction, the claim is proved. It follows that $d_k$ vanishes for $k\geq 2$ by it is a morphism of Hodge structures of different weights, which proves $(2)$.

We turn to the proof of $(1)$. The two vector spaces $\bH^\ell\left(X,\gr^F\gr^W_r\DR_X\cM \right)$ and $\gr^F V_{\ell,r}$ are isomorphic because the Hodge spectral sequence on $\bH^\ell(X,\gr^W_r \DR_X\cM)$ degenerates at the first page. Specifically, we have 
\[
(\gr^W_r\DR_X\cM,F) \simeq \bigoplus_{k\geq 0,-\frac{r}{2}} \DR_X(R^k\cP_{r+2k},F)(k),
\] 
which is isomorphic to $\bigoplus_{k\geq 0,-\frac{r}{2}} \left(\Omega^{\bullet+n-r-2k}_{\tilde Y^{(r+2k+1)}},F\right)(k+r)$ by Theorem~\ref{iden}. Meanwhile, the Hodge spectral sequence on $\bH^\ell\left(\tilde Y^{r+2k+1},\Omega^{\bullet+n-r-2k}_{\tilde Y^{(r+2k+1)}}\right)$ degenerates at the first page by the usual Hodge theory on compact K\"ahler manifolds.  Consider another weight spectral sequence:
\[
V_{\ell,r}(\gr^F\cM)=\bH^\ell\left(X, \gr^W_r \gr^F \DR_X\cM \right) \Rightarrow \bH^\ell\left(X,  \gr^F \DR_X\cM \right) 
\]
Here, we use $V^\bullet_{\ell,r}(\gr^F\cM)=E_{\bullet}^{-r,\ell+r}(\gr^F\cM)$ to denote the spectral sequences and  set $V^1_{\ell,r}(\gr^F\cM)=V_{\ell,r}(\gr^F\cM)$. We have $V_{\ell,r}(\gr^F\cM)\simeq \gr^F V_{\ell,r}$ and the differentials of the $E_1$-pages are compatible: 
\[
\begin{tikzcd}
  V_{\ell,r}(\gr^F\cM) \arrow{r}{\simeq} \arrow{d}{d_1} & \gr^F V_{\ell,r} \arrow{d}{d_1}\\
  V_{\ell+1,r-1}(\gr^F\cM) \arrow{r}{\simeq} & \gr^F V_{\ell+1,r-1}
\end{tikzcd}
\]
as $\bH^\ell\left(X, \gr^W_r \gr^F \DR_X\cM \right) \simeq \gr^F  \bH^\ell\left(X, \gr^W_r \DR_X\cM \right)$, and 
\[
\gr^W_r\gr^F\DR_X\cM=\gr^F\gr^W_r\DR_X\cM,
\] 
This implies that we have the isomorphism on the second page as well: 
\[
V_{\ell,r}^2(\gr^F\cM) \simeq V^2_{\ell,r}=\gr^F\gr^W_r\bH^\ell\left(X, \DR_X\cM \right).
\]
It follows that 
\[
\begin{aligned}
  &\dim \bH^\ell\left(X, \gr^F\DR_X\cM \right) \leq \dim \bigoplus_r V_{\ell,r}^2(\gr^F\cM) \\
  &=\dim \bigoplus_r \gr^F\gr^W_r \bH^\ell\left(X, \DR_X\cM \right) = \dim \bH^\ell\left(X, \DR_X\cM \right).
\end{aligned}
\]
On the other hand, considering the Hodge spectral sequence, we also have 
\[
\dim  \bH^\ell\left(X, \DR_X\cM \right) \leq \dim \bH^\ell\left(X, \gr^F\DR_X\cM \right),
\]
which forces that $\dim  \bH^\ell\left(X, \DR_X\cM \right) = \dim \bH^\ell\left(X, \gr^F\DR_X\cM \right)$ and therefore, 
\[
\bH^\ell\left(X, \gr^F\DR_X\cM \right) \simeq \gr^F \bH^\ell\left(X, \DR_X\cM \right).
\] 
This proves that the Hodge spectral sequence on $\bH^\ell(X, \DR_X\cM)$ degenerates at  $\prescript{}{F}{E_1}$.

The statement $(3)$ follows from Theorem~\ref{GA90}.
\end{proof}

The third statement in the above corollary ensures that the weight filtration on $\bH^\ell(X,\DR_X\cM)$ is the monodromy weight filtration of the nilpotent operator $R$, i.e. $R\colon W_\bullet \bH^\ell(X,\DR_X\cM)\to W_{\bullet-2}\bH^\ell(X,\DR_X\cM)(-1)$
and $R^r\colon\gr^W_{r}\bH^\ell(X,\DR_X\cM)\to \gr^W_{-r}\bH^\ell(X,\DR_X\cM)(-r)$ is a filtered isomorphism. We have proved Theorem~\ref{thm:main} for the case when $Y$ is reduced.

\section{Non-reduced case: Generalized eigenspace $\cM_\alpha$ and the weight filtration}\label{sec:malpha}

Now we move to the general situation. Let $I$ be the index set consisting of indices of irreducible components of $Y$ and $e_i$ is the multiplicity of $Y$ along the component $Y_i$.

\subsection{The generalized eigen-modules $\cM_\alpha$}

We start by studying the generalized eigen-modules $\ker(R-\alpha)^\infty$ of the morphism $R$ in the category of filtered $\sD_X$-modules. These generalized eigen-modules naturally forms submodules of $\cM$ but the induced filtration on does not align with the expected weight of the mixed Hodge structure and is difficult to calculate. Instead, we adopt the idea of Saito in~\cite{Sai90}: we consider the generalized eigen-module as a sub-quotient of $\cM$ and impose the induced filtration on it. It turns out this filtration has desirable properties. 

We proceed to introduce some notation. Let
\[
\cM_{\geq\alpha}=\ker\left(\prod_{\lambda\geq\alpha}(R-\lambda)^\infty\right), \quad \cM_{>\alpha}=\ker\left(\prod_{\lambda>\alpha}(R-\lambda)^\infty\right) \quad \text{and} \quad \cM_\alpha=\cM_{\geq\alpha}/\cM_{>\alpha}.
\] 
Then $\cM_\alpha$  is canonically isomorphic to the generalized eigen-module $\ker \left(R-\alpha\right)^\infty$. We equip $\cM_\alpha$ with the filtration $F_\bullet\cM_\alpha$ induced from $(\cM,F_\bullet\cM)$ given by
\[
F_\bullet\cM_\alpha=\frac{\cM_{\geq \alpha}\cap F_\bullet\cM}{\cM_{>\alpha}\cap F_\bullet\cM}.
\]
There are parallel definitions of the relative log de Rham complex. Denote by $C^\bullet=\Omega^{\bullet+n}_{X/\Delta}(\log Y)\otimes \sO_Y$ for simplicity. Define sub-complexes of $C^\bullet$ by 
\[
C^\bullet_{\geq\alpha}= C^\bullet\otimes \sO_X(-\ceil*{\alpha Y}),\quad C^\bullet_{>\alpha}=C^\bullet\otimes \sO_X(-\floor*{\alpha Y}-Y_{\Red})\quad \text{ and }C^\bullet_\alpha=C^\bullet_{\geq\alpha}/C^\bullet_{>\alpha},
\] 
where $Y_{\Red}$ is the associated reduced divisor of $Y$. If we let \(I_\alpha\) be the subset of \(I\) consisting of all \(i\) such that \(\alpha e_i\) is an integer, then
\[
C^\bullet_\alpha=C^\bullet_{\geq \alpha}\otimes \sO_{Y_{I_\alpha}}, \quad \text{where }Y_{I_\alpha}=\sum_{i\in I_\alpha}Y_i.
\]
One can check $C^\bullet_\alpha$ is a generalized eigen-perverse sheaves of the residue $[\nabla]$. Since $\sO_X(-\ceil{\alpha Y})$ is preserved by relative log differential $\sT_{X/\Delta}(-\log Y)$, the multiplication by relative log differentials gives a morphism, recalling that $D_1,D_2,\dots,D_n$ are local generators of $\sT_{X/\Delta}(-\log Y)$ dual to the local generators $\xi_1,\xi_2,\dots,\xi_n$ of $\Omega_{X/\Delta}(\log Y)$,
\begin{equation}\label{eq:ca}
\begin{aligned}
  \sO_X(-\ceil{\alpha Y})\otimes\sD_X &\to \Omega_{X/\Delta}(\log Y)(-\ceil{\alpha Y})\otimes\sD_X, \\
    z^{\ceil{\alpha \mathbf{e}}}_I\otimes P &\mapsto \sum_j \xi_j\otimes D_j z^{\ceil{\alpha \mathbf{e} }}_I\otimes P= \sum_j\xi_j\otimes  z^{\ceil{\alpha \mathbf{e} }}_I (D_j+\alpha_j)\otimes P,
\end{aligned}
\end{equation}
where, using the multi-index notation, $z^{\ceil{\alpha \mathbf{e}}}_I =\prod_{i\in I}z^{\ceil{\alpha e_i}}_i$ denotes the local generator of $\sO_X(-\ceil{\alpha Y})$ and define $\alpha_i=[D_i, z^{\ceil{\alpha \mathbf{e} }}_I]/{z^{\ceil{\alpha \mathbf{e} }}_I}={\ceil{\alpha e_i}}/{e_i}-{\ceil{\alpha e_0}}/{e_0}.$ The morphism extends to a complex $\relome{n+\bullet}(-\ceil{\alpha Y})\otimes\sD_X$, which is a subcomplex of $\relome{n+\bullet}\otimes\sD_X$ (see~\eqref{eq:relatived}). Tensoring  $\sO_Y$ on the left gives $C^\bullet_{\geq \alpha}\otimes\sD_X$ by the above definition. Further tensoring $\sO_{Y_{I_\alpha}}$ on the left, we obtain the complex of induced $\sD_X$-modules $C^\bullet_\alpha\otimes\sD_X$ with the filtration defined by 
\[
F_\ell \left(C^\bullet_\alpha\otimes\sD_X\right)= C^\bullet_\alpha \otimes F_{\ell+n+\bullet}\sD_X.
\]
The following two theorems describe the generalized eigen-modules in terms of complexes of the induced $\sD_X$-modules. 

\begin{thm}\label{thm:eigenperv}
The complex $C^\bullet_\alpha\otimes\sD_X$ is a filtered resolution and the characteristic cycle of the $0$-th cohomology is 
\[
cc\left(\sH^0\left(C^\bullet_\alpha\otimes\sD_X\right)\right)=\sum_{J\subset I} |I_\alpha\cap J| \cdot \left[ T^*_{Y^J}X \right].
\] 
\end{thm}
\begin{proof}
Similarly to the proof of Theorem~\ref{tilM} and Theorem~\ref{M}, the associated graded $\gr^F\left(C^\bullet_\alpha\otimes \sD_X\right)$ can be identified locally with the Koszul complex induced by the regular sequence $(t_\alpha, D_1,D_2,\dots,D_n)$ over $\gr^F\sD_X$, where $t_\alpha=\prod_{i\in I_\alpha}z_i$ is the defining equation of $Y_{I_\alpha}$. It follows that $\gr^F(C^\bullet_{\alpha}\otimes\sD_X)$ is a resolution and therefore, $C^\bullet_\alpha\otimes\sD_X$ is a filtered resolution. We also get that $\gr^F\sH^0(C^\bullet_\alpha\otimes\sD_X)$ is locally represented by 
\begin{equation}\label{eq:locre}
\zeta_\alpha\otimes \gr^F\sD/(t_\alpha,D_1,D_2,\dots,D_n)\gr^F\sD_X, 
\end{equation}\label{eq:locrep}
where  
\[
\zeta_\alpha=z^{\ceil{\alpha \mathbf{e} }}_I\frac{dz_1}{z_1}\wedge \frac{dz_2}{z_2}\wedge\cdots\wedge\frac{dz_k}{z_k}\wedge dz_{k+1}\wedge\cdots\wedge dz_n.
\]
Similar to the calculations in Theorem~\ref{cycle}, we get the characteristic cycle is $\sum_{J\subset I}  |I_\alpha\cap J|\cdot \left[ T^*_{Y^J}X \right]$.
\end{proof}

\begin{thm}\label{Alpha}
There exists a canonical filtered isomorphism
\begin{equation}\label{eq:aiden}
\begin{tikzcd}[column sep=small]
\psi_\alpha:\left(\sH^0\left(C^\bullet_\alpha\otimes\sD_X\right),F_\bullet\sH^0\left(C^\bullet_\alpha\otimes\sD_X\right)\right) \arrow{r}{\sim} & (\cM_\alpha,F_\bullet\cM_\alpha).
\end{tikzcd}
\end{equation}
In particular, the characteristic cycle $cc(\cM_\alpha)=\sum_{J\subset I} |I_\alpha\cap J|\cdot \left[ T^*_{Y^J}X \right]$.
\end{thm}

We give the cyclic generators of $\cM_{\geq\alpha}$ and $\cM_{>\alpha}$ locally. In principle, this always can be done because every holonomic $\sD_X$-module locally is cyclic.

\begin{lem}\label{genR}
Locally, $\cM_{\geq\alpha}$ is generated by $z^{\ceil{\alpha \mathbf{e} }}_I$, and $\cM_{>\alpha}$ is generated by $z^{\floor{\alpha \mathbf{e} }+\mathbf{1}}_I$ where $\mathbf{1}=(1,1,\dots,1)\in \Z^I$.
\end{lem}
\begin{proof}
Let us first check that the class of $z^{\ceil{\alpha \mathbf{e} }}_I$ is in $\cM_{\geq\alpha}$. It suffices to check that it is in 
\[
\ker\prod_{i\in I}\prod^{e_i-1}_{j=\ceil{\alpha e_i}}(R-\frac{j}{e_i}).
\]
This can be done by direct calculations:
\[
\begin{aligned}
\prod_{i\in I}\prod^{e_i-1}_{j=\ceil{\alpha e_i}}(R-\frac{j}{e_i})z^{\ceil{\alpha \mathbf{e} }}_I &=\prod_{i\in I}\prod^{e_i-1}_{j=\ceil{\alpha e_i}}(R-\frac{j}{e_i})z^{\ceil{\alpha e_i}}_i=\prod_{i\in I}\prod^{e_i-1}_{j=\ceil{\alpha e_i}}(\frac{1}{e_i}z_i\partial_i-\frac{j}{e_i})z^{\ceil{\alpha e_i}}_i \\
  &= \prod_{i\in I} \frac{1}{e^{e_i-\ceil{\alpha e_i}}_i}z^{e_i}_i\partial^{e_i-\ceil{\alpha e_i}}_i=t\prod_{i\in I} \frac{1}{e^{e_i-\ceil{\alpha e_i}}_i}\partial^{e_i-\ceil{\alpha e_i}}_i=0\in\cM.
\end{aligned}
\]
Because $R$ satisfies the identity~\eqref{poly}, $\cM_{\geq \alpha}$ is also equal to the image of 
\[
\prod_{i\in I}\prod^{\ceil{\alpha e_i}-1}_{j=0}(R-\frac{j}{e_i}).
\] 
We conclude from 
\[
\prod_{i\in I}\prod^{\ceil{\alpha e_i}-1}_{j=0}(R-\frac{j}{e_i})(1)=\prod_{i\in I}\prod^{\ceil{\alpha e_i}-1}_{j=0}(\frac{1}{e_i}z_i\partial_i-\frac{j}{e_i})=z^{\ceil{\alpha \mathbf{e} }}_I\prod_{i\in I}\frac{1}{e^{\ceil{\alpha e_i}}_i}\partial^{\ceil{\alpha e_i}}_i
\]
that $z^{\ceil{\alpha \mathbf{e} }}_I \prod_{i\in I}\partial^{\ceil{\alpha e_i}}_i$ generates $\cM_{\geq\alpha}$. We deduce that $z^{\ceil{\alpha \mathbf{e} }}_I$ generates $\cM_{\geq\alpha}$. 
A similar argument works for $\cM_{>\alpha}$.
\end{proof}

\begin{proof}[Proof of Theorem~\ref{Alpha}]
As the above lemma, $\cM_\alpha$ is locally isomorphic to 
\[
\zeta\otimes {\left(z^{\ceil{\alpha \mathbf{e} }}_I, D_1,D_2,\dots,D_n\right)\sD_X}/{\left(z^{\floor{\alpha \mathbf{e} }+\mathbf{1}}_I,D_1,D_2,\dots,D_n\right)\sD_X}
\]
where \(\zeta=\frac{dz_1}{z_1}\wedge\frac{dz_2}{z_2}\wedge\cdots\wedge\frac{dz_k}{z_k}\wedge\cdots\wedge dz_n\). Put $\zeta_\alpha=z^{\ceil{\alpha \mathbf{e} }}_I \zeta $. Since $\sH^0(C^\bullet_\alpha\otimes\sD_X)$ by~\eqref{eq:ca} is locally isomorphic to  
\[
\zeta_\alpha \otimes\sD_X/(t_\alpha, D_1+\alpha_1,D_2+\alpha_2,\dots,D_n+\alpha_n)\sD_X,
\]
the multiplication $\sH^0(C^\bullet_\alpha\otimes\sD_X)\to \cM_\alpha,\, \zeta_\alpha \otimes P \mapsto \zeta \otimes z^{\ceil{\alpha \mathbf{e} }}_I P$ is well-defined, does not depend on the coordinate and therefore, gives a filtered morphism
\[
\begin{tikzcd}[column sep=small]
\psi_\alpha: (\sH^0(C^\bullet_\alpha\otimes\sD_X),F_\bullet\sH^0(C^\bullet_\alpha\otimes\sD_X)) \arrow{r} & (\cM_\alpha,F_\bullet\cM_\alpha). 
\end{tikzcd}
\] 
It is clear from the local description that $\phi_\alpha$ is surjective and hence, 
\[
cc\left(\sH^0\left(C^\bullet_\alpha\otimes\sD_X\right)\right)\geq cc(\cM_\alpha). 
\]
Summing over all the rational numbers $\alpha$ in $[0,1)$ gives
\[
\sum_\alpha cc\left(\sH^0\left(C^\bullet_\alpha\otimes\sD_X\right)\right)\geq \sum_\alpha cc(\cM_\alpha)=cc(\cM).
\]
On the other hand, by Theorem~\ref{M} and Theorem~\ref{thm:eigenperv}, the $\sD_X$-module $\cM$ is also successive extensions of $\sH^0\left(C^\bullet_\alpha\otimes\sD_X\right)$ for $\alpha\in \Q\cap[0,1)$. Thus, 
\[
\sum_\alpha cc\left(\sH^0\left(C^\bullet_\alpha\otimes\sD_X\right)\right)=cc(\cM).
\]
This forces that $\psi_\alpha$ must be (filtered) isomorphism.

It remains to show that 
\begin{equation}
\begin{tikzcd}[column sep=small]
F_\ell\psi_\alpha\colon F_\ell\sH^0(C^\bullet_\alpha\otimes\sD_X) \arrow{r} & F_\ell\cM_\alpha,
\end{tikzcd}
\end{equation}
is surjective. Suppose that $z^{\ceil{\alpha \mathbf{e}}}_I P\in\sD_X$ represents a class in $F_\ell\cM_\alpha$. Then we can write
\[
z^{\ceil{\alpha \mathbf{e}}}_I P= P'+\sum^n_{i=1} D_i Q_i+ z^{\floor{\alpha \mathbf{e} }+\mathbf{1}}_I T
\]
for $P' \in F_{\ell+n}\sD_X$ and $T,Q_i\in \sD_X$. Regrouping gives 
\[
z^{\ceil{\alpha \mathbf{e}}}_I (P-t_\alpha T)= P'+\sum^n_{i=1} D_i Q_i
\]
By a similar regular sequence argument as in the proof of Theorem~\ref{M}, we can assume that $P-t_\alpha T$ is in $F_{\ell +n}\sD_X$. Then the class in $\sH^0(C^\bullet_\alpha\otimes\sD_X)$ represented by $P-t_\alpha T$ is actually in $F_\ell \sH^0(C^\bullet_\alpha\otimes\sD_X)$ by the local formula. Therefore, we have found a lifting in $F_\ell \sH^0(C^\bullet_\alpha\otimes\sD_X)$ (represented by $P$) of the class of $z^{\ceil{\alpha \mathbf{e}}}_I P$ in $F_\ell\cM_\alpha$. 
\end{proof}

Let $R_\alpha\colon=(R-\alpha)$ be the nilpotent endomorphism of $\cM_\alpha$. One easily gets a nice local formula of $R_\alpha$:

\begin{cor}
The endomorphism $R_\alpha$ of $\cM_\alpha$ acts locally as $\psi_\alpha \circ (\id\otimes \frac{1}{e_j}z_j\partial_j)\circ (\psi_\alpha)^{-1}$ for any $j\in I_\alpha$. In particular, $R_\alpha\colon (\cM_\alpha,F_\bullet\cM_\alpha) \to (\cM_\alpha,F_{\bullet+1}\cM_\alpha)$ is a filtered morphism.
\end{cor}
\begin{proof}
Because $R-\alpha$ acts on the right hand side of the identification~\eqref{eq:aiden} by the left multiplication by $\frac{1}{e_0}z_0\partial_0-\alpha$,  a direct calculation 
\[
\begin{aligned}
 R_\alpha\left[\zeta\otimes z^{\ceil{\alpha \mathbf{e} }}_I\right]
=& \left[ \zeta \otimes \left(\frac{1}{e_j}z_j\partial_j-\alpha\right)\left(z^{\ceil{\alpha \mathbf{e} }}_I\right) \right]\\
=&\left[ \zeta\otimes \left(\left(\frac{1}{e_j}\ceil{\alpha e_j}-\alpha\right)z^{\ceil{\alpha \mathbf{e} }}_I+z^{\ceil{\alpha \mathbf{e} }}_I \left(\frac{1}{e_j}z_j\partial_j\right)\right)\right]\\
=& \psi_\alpha{\left[\zeta z^{\ceil{\alpha \mathbf{e} }}_I\otimes \left(\frac{1}{e_j}z_j\partial_j\right)\right]}=\psi_\alpha\circ\left(\id\otimes\frac{1}{e_j}z_j\partial_j\right)\circ\psi_\alpha^{-1}\left[\zeta_\alpha\otimes 1\right]
\end{aligned}
\]
completes the proof.
\end{proof}

\subsection{Strictness of $R_\alpha$} Similar to the reduced case, every power of $R_\alpha$ is strict.
\begin{thm}\label{strictN}
Any power of $R_\alpha$ on $(\cM_\alpha,F_\bullet\cM_\alpha)$ is strict:
\begin{equation}\label{rl}
R_\alpha^{a}F_\bullet\cM_\alpha=F_{a+\bullet}R_\alpha^a\cM_\alpha, \quad \text{for any }a\in \Z_{\geq 0}.
\end{equation}
\end{thm}

Let $[R_\alpha]$ be the endomorphism on $\gr^F\cM_\alpha$ induced by $R_\alpha$. 

\begin{lem}\label{generator}
$\ker {[ R_\alpha]}^{r+1}$ is locally generated by monomials of degree $\mu-r$ that divide $t_\alpha$. 
\end{lem}

\begin{proof}[Proof of Theorem~\ref{strictN}]
Temporarily admitting this lemma, let $R_\alpha^{r+1}m$ be an element in $F_{\ell+r+1}\cM_\alpha$. Assume that $m\in F_k\cM_\alpha$. If $k>\ell$ then the projection of $R_\alpha^{r+1}m$ vanishes in $\gr^F_{k+r+1}\cM_\alpha$. It follows from the lemma that there are $m'\in F_{k-1}\cM_\alpha$ and $z_J=\prod_{j\in J} z_j$ such that 
\[
m=\sum_{{\substack{|J|=\mu-r,\\ J\subset I_\alpha}}} z_J m_{J}+m'.
\]
Because for every $J\subset I_\alpha$ of cardinality $r+1$, we can arrange 
\[
R^{r+1}_\alpha z_J=\prod_{j\in I_\alpha\setminus J} \frac{1}{e_j}z_j\partial_j z_{J}= t_\alpha \prod_{j\in I_\alpha}\frac{1}{e_j}\partial_{j}=0\in \cM_\alpha.
\]
It follows that  
\begin{equation}\label{eq:Niter}
  \begin{aligned}
R_\alpha^{r+1}m =& \sum_{{\substack{|J|=\mu-r,\\ J\subset I_\alpha}}} R_\alpha^{r+1}z_Jm_{J}+R_\alpha^{r+1}m' \\
  =&\sum_{{\substack{|J|=\mu-r,\\ J\subset I_\alpha}}} t_\alpha m'_{J}+R_\alpha^{r+1}m' \\
  =& R_\alpha^{r+1}m'\in \cM_\alpha.
\end{aligned}
\end{equation}
But now $m'\in F_{k-1}\cM_\alpha$. Iterating~\eqref{eq:Niter}, we can find $\tilde m\in F_\ell\cM_\alpha$ such that 
\[
R_\alpha^{r+1}m=R_\alpha^{r+1}\tilde m,
\]
which completes the proof of the theorem.
\end{proof}

\begin{proof}[Proof of the lemma]
The proof is essentially the same as the reduced case. We are now working over the commutative ring $\gr^F\sD_X$. We prove by induction on $r$. Without loss of generality, we can assume by abuse of notation that locally $I_\alpha=\{0,1,\dots,\mu\}$ and $t_\alpha=z_0z_1\cdots z_\mu$. Let $P\in\gr^F\sD_X$ represent a class in $\ker [R_\alpha]^{r+1}$. 

When $r=0$, we have 
\begin{equation}\label{baseN}
\frac{1}{e_0}z_0\partial_0P=t_\alpha Q_0+\sum^n_{i=1}D_iQ_i.
\end{equation}
Then $t_\alpha Q_0$ belongs to the ideal generated by 
\[
\partial_0,\partial_1,\dots,\partial_\mu,z_{\mu+1}\partial_{\mu+1},z_{\mu+2}\partial_{\mu+2},\dots,z_k\partial_k,\partial_{k+1},\dots\partial_n
\] 
over $\gr^F\sD_X$. Since $t_\alpha$ together with the above generators form a regular sequence in $\gr^F \sD_X$, we have $Q_0$ 
\[
Q_0=\sum^\mu_{a=0}\partial_aQ_{a}+\sum^k_{b=\mu+1}z_{b}\partial_{b}Q_{b}+\sum^n_{c=k+1}\partial_{c}Q_{c}.
\]
Substituting this in~\eqref{baseN} and regrouping, we find that
\[
\frac{1}{e_0}z_0\partial_0\left(P-\sum^\mu_{a=0}e_a\frac{t_\alpha}{z_a}Q_{a}-\sum^k_{b=\mu+1}e_b t_\alpha Q_{b}\right)\in (D_1,D_2,\dots,D_n)\gr^F\sD_X.
\]
Then because $(z_0\partial_0,D_1,D_2,\dots,D_n)$ is a regular sequence in $\gr^F\sD_X$, $P$ is a linear combination of ${t_\alpha}/{z_a}$ for $a\in\{0,1,\dots,\mu\}$ and $D_1,D_2,\dots,D_n$ over $\gr^F \sD_X$. This concludes the case when $r=0$.

Assume the statement is true for the case when the exponent is less than $r$. Because $[R_\alpha]$ sends the class of $P$ to $\ker [R_\alpha]^r$, by the induction hypothesis we have
\begin{equation}\label{reqc1}
\frac{1}{e_0}z_0\partial_0 P=\sum_{{\substack{|J|=\mu-r+1,\\ J\subset I_\alpha}}} z_JQ_{J}+\sum^n_{i=1} D_iQ_i.
\end{equation}
Fixing a subset $J$, we see that $z_JQ_J$ is in the submodule generated by $z_a,\partial_b,z_c\partial_c,\partial_d$ for $a\in I_\alpha\setminus J, b\in J, c\in I\setminus I_\alpha, d\notin I$ over $\gr^F\sD_X$. Because those generators together with $z_J$ form a regular sequence in $\gr^F\sD_X$, 
we deduce that
\[
Q_J=\sum_{a\in I_\alpha\setminus J}z_{a}Q_a+\sum_{b\in J}\partial_bQ_b+\sum_{c\in I\setminus I_\alpha}z_c\partial_cQ_c+\sum_{d\notin I}\partial_dQ_d. 
\]
Substituting in \eqref{reqc1}, we deduce that 
\[
\frac{1}{e_0}z_0\partial_0\left(P-\left(\sum_{b\in J}e_b\frac{z_J}{z_{b}}Q_b+\sum_{c\in I\setminus I_\alpha} e_cz_JQ_c\right)\right) -\sum_{a\in I_\alpha\setminus J}z_Jz_{a}Q_a
\]
is in the submodule generated by degree $\mu-r+1$ monomials dividing $t_\alpha$ except $z_J$ and by $D_1,D_2,\dots,D_n$ over $\gr^F\sD_X$. This means we can reduce $z_J Q_J$ one by one for each $J$ on the right-hand side of the equation~\eqref{reqc1} and at the last step we find that $\frac{1}{e_0}z_0\partial_0(P-P')$ is a linear combination of degree $\mu-r+2$ monomials dividing $t_\alpha$ and $D_1, D_2,\dots,D_n$, where $P'$ is a linear combination of degree $\mu-r$ monomials dividing $t_\alpha$.

Note that the left multiplication by $\frac{1}{e_0}z_0\partial_0$ has the same effect as applying $[R_\alpha]$  on $\gr^F\cM_\alpha$. Therefore, the class represented by $P-P'$
is in $\ker[R_\alpha]^r$ since degree $\mu-r+2$ monomials dividing $t_\alpha$ is in $\ker [R_\alpha]^{r-1}$. By the induction hypothesis, the class represented by $P-P'$ is a linear combination of degree $\mu-r+1$ monomials dividing $t_\alpha$. Therefore, the class represented by $P$ in $\gr^F\cM_\alpha$ is a linear combination of degree $\mu-r$ monomials dividing $t_\alpha$. This completes the proof.
\end{proof}

\begin{cor}\label{cor:locgenR}
The $\ker R^{r+1}_\alpha$ is also generated by degree $\mu-r$ monomials dividing $t_\alpha$ if one identifies $\cM_\alpha$ locally with $\sD_X/(t_\alpha, D_1,D_2,\dots,D_n)\sD_X$ via~\eqref{eq:locrep} and Theorem~\ref{Alpha}.
\end{cor}
\begin{proof}
  The proof is the same as the one of Corollary~\ref{localgenR} and is left to the readers.
\end{proof}

\subsection{The weight filtration}
Recall that since $R_\alpha$ is nilpotent on $\cM_\alpha$, it induces a $\Z$-indexed filtration $W_\bullet\cM_\alpha$.
We filtered the sub-module $W_r\cM_\alpha$ by the induced filtration $F_\bullet W_r\cM_\alpha=F_\bullet\cM_\alpha\cap W_r\cM_\alpha$. Let 
\[
\cP_{\alpha,r}=\frac{\ker R^{r+1}_\alpha}{\ker R^r_\alpha+R_\alpha \ker R^{r+2}_\alpha}
\]
be the $r$-th primitive part of $\gr^W\cM_\alpha$ with the filtration defined by 
\[
F_\ell \cP_{\alpha,r}=\text{im}\, \left(F_\ell\ker R^{r+1}_\alpha\to \cP_{\alpha,r}\right).
\]
As the formal proof in Theorem~\ref{FW}, we have

\begin{cor}\label{FWN}
The morphism $R^r_\alpha\colon F_\ell \gr^W_r\cM_\alpha \to F_{\ell+r}\gr^W_{-r}\cM_\alpha$ is an isomorphism for $r\geq 0$ and the Lefschetz decomposition of $\gr^W\cM_\alpha$ respects filtrations, i.e.
\[
F_\bullet\gr^W_r\cM_\alpha=\bigoplus_{\substack{\ell\geq 0,-\frac{r}{2}}} R^{\ell}_\alpha F_{\bullet-\ell}\cP_{\alpha,r+2\ell}\text{ \,for any\,  }r\in\Z.
\]
\end{cor}

\subsection{Summands of the primitive part $\cP_{\alpha,r}$} \label{subsec:sp} 
Recall that \(Y^J=\bigcap_{j\in J}Y_j\) and \(Y_J=\bigcup_{j\in J} Y_j\) for any subset $J$ of $I$ and $e_j$ is the multiplicity of $Y_j$ in $Y$. Like the reduced case that $\cP_r$ decomposes into the direct images of $\omega_{Y^J}(-r)$ for all index subsets $J$ of cardinality $r+1$ (Theorem~\ref{iden}), the primitive part $\cP_{\alpha,r}$ of the generalized $\alpha$-eigemodule also decomposes into direct images of certain filtered $\sD_{Y^J}$-modules $\cV_{\alpha,J}(-r)$ for all $J$ of cardinality $r+1$ such that $e_j\alpha$ for every $j\in J$ is an integer. The filtered $\sD_{Y^J}$-modules $\cV_{\alpha,J}$ comes from cyclic coverings of $Y^J$ and hence, $\cP_{\alpha,r}$ carries the Hodge theory of the cyclic coverings. In fact, by a well-known construction in~\cite[$\S$3]{EV} the direct image of the de Rham complex of a cyclic covering decomposes into log de Rham complexes of line bundles. A line bundle with an integrable log connection also can be viewed as a log $\sD$-module. This suggests that the $\sD$-modules $\cV_{\alpha,J}$ is generated by a certain log $\sD$-module $\sV_{\alpha,J}$. If $Y$ is reduced and $\alpha=0$, $\cV_{\alpha,J}$ is just $\omega_{Y^J}$. We shall construct auxiliary log $\sD$-modules $\sV_{\alpha,J}$ whose log de Rham complex will be used to construct the $\sD$-module $\cV_{\alpha,J}$, without using cyclic cover. The cyclic coverings are involved only when we study the Hodge theory of those $\sD$-modules. We fix a rational number $\alpha\in[0,1)$ and let $I_\alpha$ be a subset of indices consisting of $i$ such that $\alpha e_i$ is an integer. 

Denote by $\cL$ the line bundle $\sO_X\left(-\sum_{i\in I_\alpha} {\frac{e_i}{N}}Y_i\right)$, where $N$ is the greatest common divisor of $e_i$ for $i\in I_\alpha$. Then we have $\sO_X\left(-\ceil*{\alpha Y}\right)=\cL^{\alpha N}\left(-\sum_{i\in I\setminus I_\alpha} \ceil*{\alpha e_i Y_i}\right)$. As the line bundle $\sO_X(Y)$ can be trivialized by a global section, there is an isomorphism of $\sO_X$-modules:
\begin{equation}\label{conn}
\cL^{N}=\sO_X \left(-\sum_{i\in I_\alpha} e_iY_i\right) \to \sO_X\left(\sum_{i\in I\setminus I_\alpha} e_i Y_i\right).
\end{equation}
Choose a local section $l$ of $\cL$ such that $l^{N}\mapsto \prod_{i\in I\setminus I_\alpha} z^{-e_i}_i$ under~\eqref{conn}. We shall define a log connection $\nabla$ on 
\[
\sO_X(-\ceil{\alpha Y})=\cL^{\alpha {N}}\left(-\sum_{i\in I\setminus I_\alpha} \ceil*{\alpha e_i Y_i}\right).
\] 
First, we define, using the product rule 
\begin{equation}
\frac{\nabla l^{N}}{l^{N}}=N \frac{\nabla l}{l}=\sum_{i\in I\setminus I_\alpha} -e_i\frac{dz_i}{z_i}
\end{equation}
due to~\eqref{conn}. Then, let $s=l^{\alpha N}\prod_{i\in I\setminus I_\alpha} z^{\ceil*{\alpha e_i}}_i$ be the local frame of $\sO_X(-\ceil{\alpha Y})$. Noting that $\alpha N$ is a non-negative integer, the induced log connection works as
\begin{equation}
\begin{aligned}
\frac{\nabla s}{s}=\frac{\nabla (l^{\alpha N}\prod_{i\in I\setminus I_\alpha} z^{\ceil*{\alpha e_i}}_i)}{l^{\alpha N}\prod_{i\in I\setminus I_\alpha} z^{\ceil*{\alpha e_i}}_i} &=\alpha N \frac{\nabla l}{l}+\sum_{i\in I\setminus I_\alpha} {\ceil*{\alpha e_i}}\frac{dz_i}{z_i}\\ 
&= \sum_{i\in I\setminus I_\alpha} \left({\ceil*{\alpha e_i}}-\alpha e_i\right)\frac{dz_i}{z_i} =\sum_{i\in I\setminus I_\alpha} \{-\alpha e_i\}\frac{dz_i}{z_i}, 
\end{aligned}
\end{equation}
where $\{-\}$ denotes the function of taking fractional part. Putting in more standard form, 
\[
\nabla s=\sum_{i\in I\setminus I_\alpha} \{-\alpha e_i\}\frac{dz_i}{z_i}\otimes s.
\]
This log connection is integrable and has poles along $Y_i$ for $i\in I\setminus I_\alpha$ with eigenvalues $\{-\alpha e_i\}$. We endow the line bundle $\sO_X(-\ceil{\alpha Y})$ with this integrable log connection $\nabla$. 

Fix a subset $J$ of $I_\alpha$ with $|J|=r+1$ so that $\dim Y^J=n-r$. The pullback of $(\sO_X(\ceil{-\alpha Y}),\nabla)$ by the inclusion $\tau^J\colon Y^J\to X$ gives an integrable log connection $(\mathscr V,\nabla)=(\sV_{\alpha, J},\nabla)$ on $Y^J$ with poles along $E=E^{\alpha,J}$ the pullback of $Y_{I\setminus I_\alpha}$. Moreover, the log de Rham complex of $(\sV,\nabla)$
\[
\{\sV\to \Omega_{Y^J}(\log E)\otimes \sV \to \cdots \to \Omega^{{n-r}}_{Y^J}(\log E)\otimes \sV\} [{n-r}],
\]
induces a complex of $\sD_{Y^J}$-modules (indeed, the log de Rham complex of $\sV\otimes\sD_{Y^J}$):
\begin{equation}\label{eq:vj}
\{\sV\otimes \sD_{Y^J} \to \Omega_{Y^J}(\log E)\otimes \sV\otimes \sD_{Y^J}  \to \cdots \to \Omega^{{n-r}}_{Y^J}(\log E)\otimes \sV \otimes \sD_{Y^J} \} [{n-r}],
\end{equation}
with the filtration by the subcomplexes for $\ell\in\Z$:
\[
\{\sV\otimes F_\ell\sD_{Y^J} \to \Omega_{Y^J}(\log E)\otimes \sV\otimes F_{\ell+1}\sD_{Y^J}  \to \cdots \to \Omega^{{n-r}}_{Y^J}(\log E)\otimes \sV \otimes F_{\ell+{n-r}}\sD_{Y^J} \} [{n-r}].
\]
We will prove that the complex~\eqref{eq:vj} is a filtered resolution a filtered $\sD_{Y^J}$-module $(\cV,F)=(\cV_{\alpha,J},F)$ in Lemma~\ref{lem:filresolution}. For example, if $\alpha=0$, then $E$ is empty and $\sV$ is just $\sO_{Y^J}$ so that $\cV=\omega_{Y^J}$ as $\sD_{Y^J}$-modules. Since the eigenvalues of the log connection are in $(0,1)$ if poles exist, the log de Rham complex of $(\sV,\nabla)$ is the minimal extension $R_{!*}\mathbb V$ of the local system $\mathbb V$ consisting of the flat sections of $\nabla$ on $\sV$ over $Y^J\setminus Y_{I\setminus J}$ (see~\cite[1.6]{EV}). Later we will put a sesquilinear pairing on $\cV$ and all the data will yield a pure Hodge structure of the log de Rham complex of $\sV$.

\begin{lem}\label{lem:filresolution}
The de Rham complex $\DR_{Y^J}\cV$ together with the filtration $F_\bullet\DR_{Y^J}\cV$ is isomorphic to the log de Rham complex $\Omega^{n-r+\bullet}_{Y^J}(\log E)\otimes \sV$ with the stupid filtration in the derived category of filtered complexes of $\C$-vector spaces. In addition, $\cV$ is holonomic and the characteristic cycle of $\cV$ is
\[
cc(\cV)= \sum_{K\subset I\setminus I_\alpha} \left[ T^*_{Y^{K \cup {J}}} Y^{J}  \right].
\]
\end{lem}
\begin{proof}
We can choose the local frame $s$ of $\sV$ such that 
\[
\nabla s=\sum_{i\in I\setminus I_\alpha}\frac{dz_i}{z_i}\otimes \{-\alpha e_i\} s
\] 
where $z_i$ is the defining equation of $E_i$ for each $i$.  Therefore, the complex~\eqref{eq:vj} locally is the Koszul complex over $\sD_{Y^J}$ associated to the sequence
\[
x_1\partial_1+\{-\alpha e_1\},x_2\partial_2+\{-\alpha e_2\},\dots,x_p\partial_p+\{-\alpha e_p\},\partial_{p+1},\partial_{p+2},\dots,\partial_{n-r}, 
\]
for some rearrangement of coordinates and under the trivialization of $\sV$ given by $s$. It follows that the associated graded of~\eqref{eq:vj} is the Koszul complex associated to the regular sequence 
\[
x_1\partial_1,x_2\partial_2,\dots,x_p\partial_\nu,\partial_{p+1},\partial_{p+2},\dots,\partial_{n-r}
\]
over $\gr^F\sD_{Y^J}$. Thus, the complex~\eqref{eq:vj} is a filtered resolution. By a similar argument in Theorem~\ref{M}, the $\sD_{Y^J}$-module $\cV$ is holonomic and the characteristic cycle $cc(\cV)= \sum_{K\subset I\setminus I_\alpha} \left[ T^*_{Y^{K \cup {J}}} Y^{J}  \right]$. 

Moreover, we have isomorphisms in the derived category of complexes of $\C$-vector spaces:
\[
\begin{aligned}
F_\ell \DR_{Y^J} \cV=F_{\ell+*}\cV\otimes\bigwedge^{-\star} \sT_{Y^J} &\simeq \Omega^{n-r+\bullet}_{Y^J}(\log E)\otimes \sV \otimes F_{\ell+n-r+\bullet+*}\sD_{Y^J}\otimes \bigwedge^{-\star}\sT_{Y^J} \\
  & \simeq   \Omega^{n-r+\bullet}_{Y^J}(\log E)\otimes \sV\otimes F_{\ell+n-r+\bullet}\sO_{Y^J}.
\end{aligned}
\]
The complex $\Omega^{n-r+\bullet}_{Y^J}(\log E)\otimes \sV\otimes F_{\ell+n-r}\sO_{Y^J}$ is the stupid filtration on the log de Rham complex on $\sV$, since $F_{\ell}\sO_{Y^J}$ is $\sO_{Y^J}$ or vanishes if $\ell < 0$. 
\end{proof}

We also need an auxiliary $\sD_{Y^J}$-module $\cV^*_{\alpha,J}$ to identify the primitive part $\cP_{\alpha,r}$ which plays the role as $\omega_{Y^J}(* D^J)$ in the counterpart for the reduced case (Theorem~\ref{iden}). The log de Rham complex of $(\sV,\nabla)$ can be enlarged into
\[
\{\sV\to \Omega_{Y^J}(\log D)\otimes \sV \to \cdots \to \Omega^{{n-r}}_{Y^J}(\log D)\otimes \sV\} [{n-r}], 
\]
for  $D=D^J$ the pullback of the divisor $Y_{I\setminus J}$. It is quasi-isomorphic to $Rj_*\mathbb V$ for $j\colon Y^J- Y_{I\setminus J}\to Y^J$ is the open immersion. By a similar argument of the above, it induces a complex of filtered $\sD_{Y^J}$-modules
\begin{equation}\label{eq:starvj}
\{\sV\otimes \sD_{Y^J} \to \Omega_{Y^J}(\log D)\otimes \sV\otimes \sD_{Y^J}  \to \cdots \to \Omega^{{n-r}}_{Y^J}(\log D)\otimes \sV \otimes \sD_{Y^J} \} [{n-r}]
\end{equation}
which is a filtered resolution $\cV^*=\cV^*_{\alpha,J}$ with the filtration $F_\ell\cV^*=F_\ell\cV^*_{\alpha,J}$ induced by the subcomplex
\[
\{\sV\otimes F_\ell\sD_{Y^J} \to \Omega_{Y^J}(\log D)\otimes \sV\otimes F_{\ell+1}\sD_{Y^J}  \to \cdots \to \Omega^{{n-r}}_{Y^J}(\log D)\otimes \sV \otimes F_{\ell+{n-r}}\sD_{Y^J} \} [{n-r}].
\]
We naturally get an induced morphism $(\cV,F_\bullet\cV)\to (\cV^*,F_\bullet\cV^*)$ from the inclusion of the log de Rham complexes. 
\begin{lem}\label{lem:strictex}
The canonical morphism $(\cV,F_\bullet\cV)\to (\cV^*,F_\bullet\cV^*)$ is strictly injective, and its image is generated by the monomials defining $D-E$.
\end{lem}
\begin{proof}
Suppose $x_1x_2\cdots x_p$ is the local defining equation of $E$ and $x_1x_2\cdots x_q$ is the local defining equation of $D$ for $q\geq p+1$. Since $\cV$ is locally generated by the class of
\[
\bigwedge^p_{i=1} \frac{dx_i}{x_i}\wedge dx_{p+1} \wedge \cdots \wedge dx_{n-r}\otimes s\otimes 1
\]
and $\cV^*$ is locally generated by the class of
\[
\bigwedge^q_{i=1}\frac{dx_i}{x_i}\wedge dx_{q+1}\wedge \cdots \wedge dx_{n-r}\otimes s \otimes1,
\]
the image is generated by the class of $\bigwedge^q_{i=1}\frac{dx_i}{x_i}\wedge dx_{q+1}\wedge\cdots \wedge dx_{n-r}\otimes s\otimes x_{p+1}x_{p+2}\cdots x_q$. The morphism locally can be viewed as
\[
\begin{aligned}
\sD_{Y^J}/(x_1\partial_1+r_1,\dots,x_p\partial_p+r_p,\partial_{p+1},\dots,\partial_{n-r})\sD_{Y^J} \to \\
\sD_{Y^J}/(x_1\partial_1+r_1,\dots,x_q\partial_{q}+r_q,\partial_{q+1},\dots,\partial_{n-r})\sD_{Y^J}, 
\end{aligned}
\]
with $[P] \mapsto [x_{p+1}x_{p+2}\cdots x_q P]$ where $r_1,r_2,\dots,r_p$ are the eigenvalues of $\nabla$ on $\sV$ and $r_{p+1}=r_{p+2}=\cdots=r_{q}=0$. Since 
\[
\Omega^{n-r}_{Y^J}(\log E)\otimes \sV=F_{-(n-r)}\cV\to F_{-(n-r)}\cV^*=\Omega^{n-r}_{Y^J}(\log D)\otimes\sV
\]
is injective, it suffices to show that $\gr^F\cV\to \gr^F\cV^*$ is injective by induction. Due to the complexes~\eqref{eq:vj} and~\eqref{eq:starvj} are filtered resolutions, the morphism on the associated graded modules corresponds to the map, 
\[
\begin{aligned}
  \gr^F\sD_{Y^J}/(x_1\partial_1,\dots,x_p\partial_p,\partial_{p+1},\dots,\partial_{n-r})\gr^F\sD_{Y^J} \to \\
  \gr^F\sD_{Y^J}/(x_1\partial_1,\dots,x_q\partial_q,\partial_{q+1},\dots,\partial_{n-r})\gr^F\sD_{Y^J},
\end{aligned}
\]
with $[P] \mapsto [x_{p+1}x_{p+2}\cdots x_q P]$. By another induction on the number of components of $D-E$, we can assume $q=p+1$. Let $P\in \gr^F\sD_{Y^J}$ represent a class in the kernel of the above morphism. Then we have
\[
x_qP=\sum^q_{i=1} x_i\partial_i P_i  +\sum^{n-r}_{j=q+1}\partial_j P_j\in \gr^F\sD_{Y^J}.
\]
Subtracting $x_q\partial_q P_q$ on both sides gives 
\[
x_q(P-\partial_q P_q)=\sum^{q-1}_{i=1} x_i\partial_i P_i  +\sum^{n-r}_{j=q+1}\partial_j P_j\in \gr^F\sD_{Y^J}.
\]
Since $x_q,x_1\partial_1,\dots,x_{q-1}\partial_{q-1},\partial_{q+1},\dots,\partial_{n-r}$ is a regular sequence over $\gr^F\sD_{Y^J}$,  
\[
(P-\partial_q P_q)=\sum^{q-1}_{i=1} x_i\partial_i P'_i  +\sum^{n-r}_{j=q+1}\partial_j P'_j\in \gr^F\sD_{Y^J}.
\]
Hence, $P$ is a linear combination of $x_1\partial_1,x_2\partial_2,\dots,x_p\partial_p, \partial_{p+1},\dots,\partial_{n-r}$ over $\gr^F\sD_{Y^J}$, which completes the proof.
\end{proof} 
\begin{rem}
One can utilize the Riemann-Hilbert correspondence to establish that $\cV$ corresponds to the minimal extension of $\sV|_{Y^J\setminus D}$, while $\cV^*$ corresponds to the $*$-extension of $\sV|_{Y^J\setminus D}$. As a result, we observe that $\cV \to \cV^*$ is an injective map. However, in our situation, this argument is overkill and does not capture the strictness, namely $F_\ell\cV=F_\ell\cV^* \cap \cV$.
\end{rem}

Putting in more general notations and summarizing what we have proved in the above two lemmas:
\begin{thm}\label{cycinfor}
The filtered $\sD_{Y^J}$-module $(\cV_{\alpha,J},F_\bullet)$ is holonomic and its de Rham complex $\DR_{Y^J}\cV_{\alpha,J}$ together with the induced filtration is isomorphic to the log de Rham complex $\Omega^{n-r+\bullet}_{Y^J}(\log E^{\alpha,J})\otimes \sV_{\alpha,J}$ with the stupid filtration in the derived category of filtered complexes of $\C$-vector spaces and whose characteristic cycle is
\[
cc(\cV_{\alpha,J})= \sum_{K\subset I\setminus I_\alpha} \left[ T^*_{Y^{K\cup J}} Y^{J}  \right].
\]
The canonical filtered morphism $(\cV_{\alpha, J}, F_\bullet\cV_{\alpha, J})\to (\cV^*_{\alpha, J}, F_\bullet \cV^*_{\alpha,J})$ is filtered injective and the image is generated by the monomial defining the divisor $D^J-E^{\alpha,J}$.
\end{thm}

\subsection{Identifying the primitive part $\cP_{\alpha,r}$}
We are going to identify the $r$-th primitive part $(\cP_{\alpha,r},F_\bullet\cP_{\alpha,r})$ with a direct sum of $\cV_{\alpha,J}(-r)$ for $J$ ranging over subsets $ I_\alpha$ of cardinality $r+1$. The argument is parallel to the one of the reduced case (Theorem~\ref{iden}).

\begin{thm}\label{idenN}
Let $\cV_{\alpha,r}=\bigoplus_{J}\tau^J_{+}\cV_{\alpha,J}$ for $J$ running over the subsets of $I_\alpha$ of cardinality $r+1$, where $\tau^J:Y^J\hookrightarrow X$ is the closed embedding. Then there exists an isomorphism $\phi_{\alpha,r}:(\cP_{\alpha,r},F_\bullet\cP_{\alpha,r}) \to \cV_{\alpha,r}(-r)$
between filtered $\sD_X$-modules. 
\end{thm}
\begin{proof}
Because the log connection~\eqref{conn} we constructed on \(\sO_X(-\ceil{\alpha Y})\) has zero residue on \(Y_i\) for \(i\in I_\alpha\), we have the residue morphism between log de Rham complexes.
\[
\begin{aligned}
\Res_{Y^J}: \Omega^{\bullet+n+1}_X(\log Y)\otimes \sO_X(-\ceil{\alpha Y})|_{Y_{I_\alpha}} &\to \Omega^{\bullet+n-r}_{Y^J}(\log D^J)\otimes \sV_{\alpha,J},  
\end{aligned}
\]
where $D^J$ is the pullback of $Y_{I\setminus J}$ and \(J\subset I_\alpha\) of cardinality $r+1$,
up to a sign depending on the order of the indices. Denote by $B^\bullet_\alpha$ the log de Rham complex $ \Omega^{\bullet+n+1}_X(\log Y)\otimes \sO_X(-\ceil{\alpha Y})$ of $\sO_X(-\ceil{\alpha Y})$. The residue morphism $\Res_{Y^J}$ extends to a morphism of the complexes of induced $\sD_X$-modules
\[
\Res_{Y^J}: B^\bullet_\alpha|_{Y_{I_\alpha}}\otimes \sD_X\to \Omega^{\bullet+n-r}_{Y^J}(\log D^J)\otimes \sV_{\alpha,J}\otimes\sD_X.
\] 
Let $\cH^\ell_\alpha\colon=\cH^\ell\left(B^\bullet_\alpha|_{Y_{I_\alpha}}\otimes \sD_X\right)$. Then we obtain a filtered morphism $\Res_{Y^J}:\cH^0_\alpha \to \cV^*_{\alpha, J}$ by taking cohomology. Let $\Res_{\alpha,r}=\bigoplus \Res_{Y^J}:\cH^0_\alpha\to \cV^*_{\alpha,r}(-r)$ where $\cV^*_{\alpha,r}= \bigoplus_J \cV^*_{\alpha, J}$ for $J$ running over cardinality $r+1$ subsets of $I_\alpha$. Because $\frac{dt}{t}\wedge: \Omega^{\bullet+n}_{X/\Delta}(\log Y)(-\ceil{\alpha Y})\to \Omega^{\bullet+n+1}_X(\log Y)(-\ceil{\alpha Y})$ also extends to the complexes of the induced $\sD_X$-modules, we get a short exact sequence 
\[
\begin{tikzcd}[column sep=small]
0 \arrow{r} &C^{\bullet-1}_\alpha\otimes\sD_X \arrow{r}{\frac{dt}{t}\wedge}  & B^{\bullet-1}|_{Y_{I_\alpha}}\otimes \sD_X \arrow{r} & C^{\bullet}_\alpha\otimes\sD_X \arrow{r} & 0.
\end{tikzcd}
\]
The associated long exact sequence gives
\begin{equation}\label{eq:lesnr}
\begin{tikzcd}%[column sep=small]
  0 \arrow[r]
    & \cH^{-1}_\alpha \arrow[r]
        \arrow[d, phantom, ""{coordinate, name=Z}]
      & \cM_\alpha \arrow[dll,
                 "R_\alpha",
    rounded corners,
    to path={ -- ([xshift=2ex]\tikztostart.east)
        |- (Z) [near end]\tikztonodes
      -| ([xshift=-2ex]\tikztotarget.west) 
      -- (\tikztotarget)}] \\
         \cM_\alpha \arrow[r,"\frac{dt}{t}\wedge"]
    & \cH^0_\alpha \arrow[r]
& 0.
\end{tikzcd}
\end{equation}
 By pre-composing \(\frac{dt}{t}\wedge\), we get a morphism 
 \[
 \Res_{\alpha,J}\circ \frac{dt}{t}\wedge:\cM_\alpha\to \cV^*_{\alpha,r}(-r), \quad [\zeta_\alpha\otimes P]\to [\Res_{\alpha,J}\frac{dt}{t}\wedge\zeta_\alpha\otimes P].
 \] 
Recall that every element in $\cM_\alpha$ is locally represented by $\zeta_\alpha \otimes P$ for $\zeta_\alpha=z^{\ceil{\alpha\mathbf{e}}}_I\frac{dz_1}{z_1}\wedge \frac{dz_2}{z_2}\wedge \cdots\wedge \frac{dz_k}{z_k}\wedge\cdots\wedge dz_n$ given that locally $I=\{0,1,\dots, k\}$, and $P\in\sD_X$. By Corollary~\ref{cor:locgenR}, every class in $\ker R^{r+1}_\alpha$ is represented by a linear combination of $\zeta_\alpha\otimes z_{\overline J} P$ for some ordered index subset $J$ of $I_\alpha$ of cardinality $r+1$ and $\overline J$ is the complement of $J$ in $I_\alpha$ and $z_{\overline J}=\prod_{j\in \overline J}z_j$. Thus, its image under the above morphism is only contained in the component $\cV^*_{\alpha,J}(-r)$ because $z_{\overline J}$ vanishes on other components. The image is the class represented by
\begin{equation}\label{eq:imageresn}
\begin{aligned}
&\Res_{\alpha, J} \frac{dz_0}{z_0}\wedge \frac{dz_1}{z_1}\wedge\cdots\wedge\frac{dz_k}{z_k}\wedge\cdots\wedge dz_n z^{\ceil{\alpha\mathbf{e}}}_I\otimes z_{\overline J}P \\
&=\pm \frac{dz_{I\setminus J}}{z_{I\setminus J}}\wedge dz_{k+1}\wedge \cdots \wedge dz_n \otimes s_{\alpha, J}\otimes z_{\overline J}P\in \Omega^{n-r}_{Y^J}\otimes\sV_{\alpha,J}\otimes \sD_X,
\end{aligned}
\end{equation}
where $s_{\alpha,J}$ is the local frame of $\sV_{\alpha,J}$ by restricting $z^{\ceil{\alpha\mathbf{e}}}_I$ and the sign is depending on the order of $J$. It also follows from the calculation that the image does not have a pole along the pull-back of $Y_{\overline J}$. Hence, it is contained in the subsheaf consisting of classes represented by $\Omega^{n-r}_{Y^J}(\log E^{\alpha,J})\otimes \sV_{\alpha,J}\otimes\sD_X$, where $E^{\alpha,J}$ is the pull-back of $Y_{I\setminus I_\alpha}$ so that $D^J-E^{\alpha,J}$ is the pull-back of $Y_{\overline J}$. This means that the image of the class represented by~\eqref{eq:imageresn} is also in the image of the canonical inclusion:  
\[
\begin{aligned}
\tau^{J}_+\cV_{\alpha,J}(-r) \hookrightarrow \tau^{J}_+ \cV^*_{\alpha,J}&(-r),  \\
[dz_{\overline J}\wedge \frac{dz_{I\setminus I_\alpha}}{z_{I\setminus I_\alpha}}\wedge dz_{k+1}\wedge\cdots \wedge dz_n \otimes s_{\alpha, J}\otimes P] &\mapsto \\
[\frac{dz_{\overline J}}{z_{\overline J}}\wedge \frac{dz_{I\setminus I_\alpha}}{z_{I\setminus I_\alpha}}\wedge dz_{k+1}\wedge &\cdots \wedge dz_n \otimes s_{\alpha ,J}\otimes z_{\overline J}P].
\end{aligned}
\]
See Theorem~\ref{cycinfor}. Therefore, there is a filtered morphism $\rho_{\alpha,r}$ making the following diagram commute.
\[
\begin{tikzcd}%[column sep=small]
\ker R_\alpha^{r+1} \arrow[dashed]{rr}{\rho_{\alpha,r}} \arrow[hook]{d} & & \cV_{\alpha,r}(-r) \arrow[hook]{d}\\
\cM_\alpha \arrow{r}{{\frac{dt}{t}\wedge}} &   \cH^0_\alpha \arrow{r}{\Res_{\alpha,r}}  & \cV^*_{\alpha,r}(-r)
\end{tikzcd}  
\]
The kernel of $\rho_{\alpha,r}$ contains $\ker R_\alpha^{r}$: for an element in $\ker R_\alpha^{r}$ locally represented by $\zeta_\alpha \otimes z_K P$ for $K$ a subset of $I_\alpha$ such that the cardinality of $I_\alpha \setminus K$ is $r$, its image under $\rho_{\alpha,r}$ is zero because $z_K$ annihilates all $\Omega^{n-r}_{Y^J}(\log D^{J})\otimes \sV_{\alpha,J}$ for any $J\subset I_\alpha$ of cardinality $r+1$. The morphism $\rho_{\alpha, r}$ also kills $R_\alpha \ker R^{r+2}_\alpha$ because $\frac{dt}{t}\wedge$ vanishes on the image of $R_\alpha$ by~\eqref{eq:lesnr}. It follows that $\rho_{\alpha,r}$ factors through a filtered morphism
\[
\phi_{\alpha,r}:\cP_{\alpha,r} = \frac{\ker R^{r+1}_\alpha}{\ker R^r_\alpha+R_\alpha \ker R^{r+2}_\alpha} \to  \cV_{\alpha,r}(-r).
\]
For $dz_{\overline J}\wedge \frac{dz_{I\setminus I_\alpha}}{z_{I\setminus I_\alpha}}\wedge dz_{k+1}\wedge\cdots \wedge dz_n \otimes s_{\alpha, J}\otimes P\in \Omega^{n-r}_{Y^J}(\log E^{\alpha,J})\otimes \sV_{\alpha,J} \otimes F_{\ell}\sD_X$ representing a class in $F_\ell\tau^J_+\cV_{\alpha, J}(-r)$ where $J\subset I_\alpha$ of cardinality $r+1$, we can find a lifting represented by $\zeta_{\alpha}\otimes z_{\overline J} P$ in $F_\ell\ker R^{r+1}_\alpha$, which means
\[
F_\ell\ker R_\alpha^{r+1} \to  F_{\ell+r}\cV_{\alpha,r} 
\] 
is surjective, i.e. the morphism $\phi_{\alpha,r}$ is filtered surjective. We prove that $\phi_{\alpha,r}$ is a (filtered) isomorphism by counting the characteristic cycles as in Theorem~\ref{iden}. Because $\phi_{\alpha,r}$ is surjective, one gets
\[
cc(\cP_{\alpha,r})\geq cc(\cV_{\alpha,r}).
\]
It follows from Corollary~\ref{cycinfor} that
\[
cc(\cV_{\alpha,r})=\sum_{\substack{J\subset I_\alpha,\\ |J|=r+1}}cc(\tau^J_{+}\cV_{\alpha,J})=\sum_{\substack{J\subset I_\alpha,\\  |J|=r+1}}\sum_{K\subset I\setminus I_\alpha} \left[ T^*_{Y^{J\cup K}}X \right]=\sum_{\substack{J\subset I,\\ |J\cap I_\alpha|=r+1}} \left[ T^*_{Y^J}X \right].
\]
On the other hand, by the Lefschetz decomposition and Theorem~\ref{Alpha},
\[
\begin{aligned}
\sum_{J\subset I} |J\cap I_\alpha|\cdot \left[T^*_{Y^J}X \right] &=cc(\cM_\alpha) =cc(\gr^W \cM_\alpha)=\sum_{r\geq 0}(r+1)cc(\cP_{\alpha,r}) \\
&\geq\sum_{r\geq 0}(r+1)cc(\cV_{\alpha,r}) \\
  &=\sum_{r\geq 0}\sum_{\substack{J\subset I,\\   |J\cap I_\alpha|=r+1}}(r+1) \left[ T^*_{Y^J}X \right]=\sum_{J\subset I} |J\cap I_\alpha|\cdot \left[T^*_{Y^J}X \right].
  \end{aligned}
\]
It follows that all inequalities must be equalities and in particular, 
\[
cc(\cP_{\alpha,r})=cc(\cV_{\alpha,r})
\] 
from which we conclude that $\phi_{\alpha,r}$ is an isomorphism between (filtered) $\sD_X$-modules.
\end{proof}

\section{Non-reduced case: Sesquilinear pairing and limiting mixed Hodge structure}\label{sec:sesquil}

\subsection{K\"ahler package of cyclic covering}\label{subsec:kpcyc}

To accomplish our goal, we need to show that the graded vector space
\[
\bigoplus_{\ell\in\Z} \bH^\ell \left (Y^J,\Omega^{\bullet+n-r}_{Y^J}(\log E^{\alpha,J})\otimes \sV_{\alpha,J} \right)
\]
underlies a polarized Hodge-Lefschetz structure and satisfies the Hard Lefschetz property. This will allow us to establish the hypercohomology of the de Rham complex of the primitive part $\cP_{\alpha,r}$ inherits these properties by Theorem~\ref{cycinfor} and Theorem~\ref{idenN}. To achieve this, we will utilize the geometry of cyclic coverings.

We first give another description of the integrable log connection~\eqref{conn} using cyclic coverings. Fix a rational number $\alpha$ in $[0,1)$, Because the isomorphism, 
\[
\cL^{N}=\sO_X \left(-\sum_{i\in I_\alpha} e_iY_i\right) \to \sO_X\left(\sum_{i\in I\setminus I_\alpha} e_i Y_i\right),
\]
we obtain a cyclic covering $\pi_\alpha: X_\alpha \to X$ by taking the $N$-th roots out of $\sum_{i\in I\setminus I_\alpha} e_i Y_i$ and normalizing it. The direct image ${\pi_\alpha}_*\sO_{X_\alpha}$ decomposes into eigenspaces with respect the Galois action as well as the direct image of exterior differential ${\pi_\alpha}_* \sO_{X_\alpha}\to {\pi_\alpha}_* \Omega_{X_\alpha}$~\cite[Theorem 3.2]{EV}. The line bundle
\[
\cL^{\alpha {N}}\left(-\sum_{i\in I\setminus I_\alpha} \ceil*{\alpha e_i Y_i}\right),
\] 
is the {$\alpha$-eigenspaces} of ${\pi_\alpha}_* \sO_{X_\alpha}$ for some suitable choice of a generator of the Galois group. Because the decomposition respects the exterior differential, we obtained an integrable log connection with eigenvalues $\{\alpha e_i\}$ along $Y_i$ for each $i\in I_\alpha$. Note that $X_\alpha$ might be singular.

Let $J\subset I_\alpha$ be a subset of cardinality $r+1$. Since $Y^J$ is not contained in $Y_{I\setminus I_\alpha}$, the fiber product $Y^J_\alpha=X_\alpha\times_X Y^J$ is again a cyclic covering of $Y^J$  by taking the $N$-th roots out of $\sum_{i\in I\setminus I_\alpha} e_i Y_i\cap Y^J$. Let $\pi^J_\alpha: Y^J_\alpha\to Y^J$ be the second projection. 
\begin{equation}\label{nrootdiag}
\begin{tikzcd}
Y^J_\alpha \arrow{r} \arrow{d}{\pi^J_\alpha} & X_\alpha \arrow{d}{\pi_\alpha} \\
Y^J \arrow{r}{\tau^J} & X
\end{tikzcd}
\end{equation}
We conclude that $(\sV_{\alpha,J},\nabla)$ is the $\alpha$-eigenspace of ${\pi^J_\alpha}_*(\sO_{Y^J_\alpha},d)$. The log de Rham complex of $(\sV_{\alpha,J},\nabla)$ is a summand of the direct image of the de Rham complex ${\pi^J_\alpha}_*\Omega^{\bullet+n-r}_{Y^J_\alpha}$ of $Y^J_\alpha$. 

We shall work in the general setting and adopt the convention in \cite{Viehweg} and \cite{EV}. Let $\cL$ be a line bundle on a K\"ahler manifold $Z$ with a K\"ahler form $\omega$ and $D=\sum_{i}\nu_iD_i$ be a simple normal crossings divisor such that for some $N>1$ one has $\cL^N=\sO_Z(D)$. Define $\cL^{(j)}=\cL^j(-\floor{\frac{jD}{N}})$ for $1\leq j\leq N-1$. One puts an integrable logarithmic connection on $\cL^{(j)}$ with poles along $D^{(j)}$, where 
\[
D^{(j)}=\sum_{\frac{j\nu_i}{N}\notin \Z}D_i.
\] 
Let $\iota:U\hookrightarrow Z$ be the complement of $D$ and $\mathbb V$ is the underlying local system of $\cL|_U$. Let $\tau: Z'\to Z$ be the cyclic covering obtained by first taking $N$-th root out of $D$ then taking the normalization and $\pi: \tilde Z\to Z'$ be a log resolution of singularity equivariant with respect to the Galois group $\Gal(Z'/Z)=\langle \sigma \rangle$ and let $E$ be the simple normal crossing exceptional divisor. 
\[
\begin{tikzcd}
\tilde Z \arrow{r}{\pi}\arrow[bend left]{rr}{\eta} & Z' \arrow{r}{\tau} & Z
\end{tikzcd}
\]
Note that $\tilde Z$ is K\"ahler because it is a resolution of a subvariety of the geometric line bundle of $\cL$, which is K\"ahler, although the induced K\"ahler class does not relate well with $\omega$ on $X$. The pull-back $\eta^*\omega$ is only positive over $\tilde U=\eta^{-1}(U)$, but one can still cook up a K\"ahler class by adding a small multiple of the first Chern class $\Theta\in H^2(\tilde Z,\Z(1))$ of the relative ample line bundle of the projective morphism $\pi:\tilde Z\to Z'$. We can assume \(\Theta\) is invariant under $\sigma$ by averaging it. 

\begin{lem}
In the given notation, the cohomology class $[\eta^*\omega]+\lambda(2\pi\sqrt{-1})^{-1}\Theta\in H^{1,1}(Z)\cap H^2(Z,\R)$ is an invariant K\"ahler class if $\lambda$ is a sufficient small positive number.
\end{lem}
\begin{proof}
Let $\tilde D_i$ be the strict transformation of $\tau^{-1}(D_i)$ and $s_i\in H^0(\tilde Z,\sO_{\tilde Z}(\tilde D_i))$ whose zero locus is $\tilde D_i$. Let $h_i$ be a Hermitian metric on each line bundle $\sO_{\tilde Z}(\tilde D_i)$ and $\rho_i$ be sufficient small positive bump function supported in a small neighborhood of $\tilde D_i$ for each $i$. Then the $(1,1)$-form 
\[
\eta^*\omega+\sum_i \frac{\sqrt{-1}}{2\pi}\partial\bar\partial \rho_i h_i(s_i,s_i)
\]
is positive on $\tilde Z-E$ but only semi-positive over $E$. However, the class $(2\pi\sqrt{-1})^{-1}\Theta$ is positive over $E$. Therefore, for $\lambda$ sufficient small positive, the class of 
\[
\eta^*\omega+\sum_i \frac{\sqrt{-1}}{2\pi}\partial\bar\partial \rho_i h_i(s_i,s_i)+\lambda(2\pi\sqrt{-1})^{-1}\Theta
\]
is a K\"ahler class. But $\partial\bar\partial \rho_i h_i(s_i,s_i)$ is exact. The cohomology class of above just equals $[\eta^*\omega]+\lambda(2\pi\sqrt{-1})^{-1}\Theta\in H^{1,1}(\tilde Z)\cap H^2(Z,\R)$. It is invariant because both $[\eta^*\omega]$ and $\Theta$ are invariant.
\end{proof}

\begin{lem} \label{lem:subhodge}
The hypercohomology $\bH^k\left(Z,\Omega^\bullet_Z(\log D^{(j)})\otimes {\cL^{(j)}}^{-1}\right)$ is a summand of $\xi^{-j}$-eigenspace of $H^k(\tilde Z)$, and thus it is a sub-Hodge structure of weight $k$.
\end{lem}

\begin{proof}
By $(1.6)$ in \cite{Viehweg} that $R\iota_! \mathbb V^{-j}$, $R\iota_* \mathbb V^{-j}$ and $\Omega^\bullet_{Z}(\log D^{(j)})\otimes {\cL^{(j)}}^{-1}$ are all quasi-isomorphic, which implies that 
\[
\bH^k \left(Z,\Omega^\bullet_{Z}(\log D^{(j)})\otimes {\cL^{(j)}}^{-1}\right)\simeq H^k_c(U, \mathbb V^{-j})\simeq H^k(U, \mathbb V^{-j}).
\]
Because $\eta$ is \'etale over $U$, $H^k(U,\mathbb V^j)$ (resp. $H^k_c(U, \mathbb V^j)$) is a $\xi^j$-eigenspace of $H^k(\tilde U, \C)$ (resp. $H^k_c(\tilde U,\C)$) for the cyclic action $\sigma$, where $\xi$ is a $N$-th root of unity. Then the canonical morphisms of mixed Hodge structures
\begin{equation}\label{eq:mmhs}
H^k_c(\tilde U) \to  H^k(\tilde Z) \to H^k(\tilde U)
\end{equation}
respect the eigenspaces decomposition because we make $\tilde Z$ equivariant.
\end{proof}

\begin{lem} 
Let $\sfX={2\pi\sqrt{-1}}L$ where $L=[\omega]\wedge$ is the Lefschetz operator on $Z$. For $0\leq k \leq n$, the following two statements hold:
\begin{enumerate}
  \item Hard Lefschetz is valid on the hypercohomology, i.e. 
  \[
  \sfX^k: \bH^{n-k}\left(Z,\Omega^\bullet_Z(\log  D^{(j)})\otimes {\cL^{(j)}}^{-1}\right)\to \bH^{n+k}\left(Z,\Omega^\bullet_Z(\log  D^{(j)})\otimes {\cL^{(j)}}^{-1}\right)(k)
  \]
  is an isomorphism of Hodge structures.
  \item The pairing 
  \begin{equation} %\label{cyclicpairing}
  (m',\overline{m''})\mapsto \frac{\varepsilon(n-k+1)}{(2\pi\sqrt{-1})^{n}}\int_{\tilde Z} \eta^* \left(\sfX^{k} m'\wedge \overline{m''}\right)
  \end{equation}
   is a polarization on the primitive part of $\bH^{n-k}\left(Z,\Omega^\bullet_Z(\log  D^{(j)})\otimes {\cL^{(j)}}^{-1}\right)$, where $\eta^* \left(\sfX^{k} \alpha\wedge \overline{\beta}\right)$ is the top form determined by the inclusion 
   \[
   \eta^*\Omega^{n}_Z(\log  D^{(j)})\otimes {\cL^{(j)}}^{-1}\subset \omega_{\tilde Z}.
   \]
\end{enumerate}
\end{lem}

\begin{proof}
Let $\tilde L=[\eta^*\omega+\lambda\Theta]\wedge$ be the Lefschetz operator on $\tilde Z$. Then the Hard Lefschetz on $\tilde Z$ says 
\[
\tilde \sfX^k :H^{n-k}(\tilde Z)\to H^{n+k}(\tilde Z)(k)
\]
is an isomorphism, where $\tilde \sfX\colon =2\pi\sqrt{-1}\tilde L$. Because $\tilde L$ is invariant and respects the morphisms in~\eqref{eq:mmhs}, the above isomorphism is compatible with eigenspaces decomposition, it follows that
\begin{equation}\label{eq:kpeg}
{\tilde \sfX}^k: \bH^{n-k}\left(Z,\Omega^\bullet_Z(\log  D^{(j)})\otimes {\cL^{(j)}}^{-1}\right)\to \bH^{n+k}\left(Z,\Omega^\bullet_Z(\log  D^{(j)})\otimes {\cL^{(j)}}^{-1}\right)\left(k\right)
\end{equation}
is injective by Lemma~\ref{lem:subhodge}. In fact, the $\xi^i$-eigenspace of $H^k_c(\tilde U)$ is orthogonal to the $\xi^j$-eigenspace of $H^{2n-k}(\tilde U)$ with respect to Poincar\'e pairing unless $i+j\equiv 0\, (\mathrm{mod}\, N)$: for $a$ in the $\xi^i$-eigenspace of $H^k_c(\tilde U)$ and $b$ in the $\xi^j$-eigenspace of $H^{2n-k}(\tilde U)$ then
\[
\xi^i \int_{\tilde U} a\wedge b=\int_{\tilde U} \sigma^* a\wedge b = \int_{\tilde U} a\wedge (\sigma^{-1})^* b=\xi^{-j} \int_{\tilde U} a\wedge b, 
\]
which means $\int_{\tilde U} a\wedge b$ is zero unless $i+j\equiv 0\, (\mathrm{mod}\, N)$. Hence, the $\xi^i$-eigenspace of $H^k_c(\tilde U)$ is Poincar\'e dual to the $\xi^{-i}$-eigenspace of $H^{2n-k}(\tilde U)$. On the other hand, since the $\xi^i$-eigenspace is complex {conjugate} to the $\xi^{-i}$-eigenspace,  the $\xi^i$-eigenspace of $H^k_c(\tilde U)$ and the $\xi^i$-eigenspace of $H^{2n-k}(\tilde U)$ have the same dimension. We deduce that the morphism~\eqref{eq:kpeg} must be an isomorphism. 

We have $\tilde L=\eta^*L$ on $H^\bullet_c(\tilde U)$, because $\Theta$ is supported on $E$. Therefore, 
\[
\sfX^k: \bH^{n-k}\left(Z,\Omega^\bullet_Z(\log  D^{(j)})\otimes {\cL^{(j)}}^{-1}\right)\to \bH^{n+k}\left(Z,\Omega^\bullet_Z(\log  D^{(j)})\otimes {\cL^{(j)}}^{-1}\right)(k)
\]
is an isomorphism. We conclude $(1)$. It also follows that $\eta^*$ identifies the primitive part of $\sfX$
\[
\bH^{n-k}_{\prim}\left(Z,\Omega^\bullet_Z(\log  D^{(j)})\otimes {\cL^{(j)}}^{-1}\right)
\]
with the primitive part of $\tilde \sfX$
\[
\ker \left({\tilde \sfX}^{k+1}:\bH^{n-k}\left(Z,\Omega^\bullet_Z(\log  D^{(j)})\otimes {\cL^{(j)}}^{-1}\right)\to \bH^{n+k+2}\left(Z,\Omega^\bullet_Z(\log D^{(j)})\otimes {\cL^{(j)}}^{-1}\right) \right).
\]
Thus, $\bH^{n-k}_{\prim}\left(Z,\Omega^\bullet_Z(\log D^{(j)})\otimes {\cL^{(j)}}^{-1}\right)$ is a sub-Hodge structure of $H^{n-k}_{\prim}(\tilde Z)$. As the restriction of the polarization to a sub-Hodge structure still a polarization. We have proved $(2)$.
\end{proof}

The above two lemmas indicate that the graded vector space
\[
\bigoplus_{k\in\Z}\bH^k\left(Z,\Omega^\bullet_Z(\log D^{(j)})\otimes {\cL^{(j)}}^{-1}\right)
\]
is a polarized sub-Hodge-Lefschetz structure of $\bigoplus_{k\in\Z}H^k(\tilde Z,\C)$. In practice, it is more convenient to make the polarization independent of the resolution of singularities and intrinsic on $Z$. Heuristically, the local system $\mathbb V^{-j}$ over $U$ inherits a pairing from $\C_{\tilde U}$ and it has a Hodge theoretic extension on its canonical extension. First, we can resolve $\Omega^\bullet_Z(\log D^{(j)})$ using $\cA^\bullet_Z(\log D^{(j)})$, the complex of $\mathscr C^\infty$-forms with log poles along $D^{(j)}$. Note that we have the inclusion of sheaves
\[
\cA^{n+k}_Z(\log D^{(j)})\otimes {\cL^{(j)}}^{-1} \wedge \overline{\cA^{n-k}_Z(\log D^{(j)})\otimes {\cL^{(j)}}^{-1}}\subset \cA^{2n}_Z\otimes {\cL^{(j)}}^{-1} (D^{(j)}) \tens{\C}{} \overline{ {\cL^{(j)}}^{-1}(D^{(j)}) }. 
\]
Since $\cL^N\simeq \sO_Z(D)$, picking local section of $l$ such that $l^N=\prod_i x^{-\nu_i}_i$ we can put a canonical singular Hermitian metric on $\cL$ by setting the weight function as 
\[
|l|_h=\prod_{i}|x_i|^{{-\nu_i}/{N}}, \quad \text{where }x_i \text{ is the local defining equation of }D_i.
\]
The induced singular Hermitian metric on ${\cL^{(j)}}^{-1}(D^{(j)})=\cL^{-j}(\floor{\frac{jD}{N}}+D^{(j)})$ locally is
\[
\left|l^{-j} \prod_{i}x^{-\floor{j\mu_i/N}}_i \prod_{j\nu_i/N\notin\Z} x^{-1}_i \right|_h=\prod_{i}|x_i|^{{j\nu_i}/{N}-\floor{{j\nu_i}/{N}}}\prod_{j\nu_i/N\notin\Z} |x_i|^{-1}=\prod_{i}|x_i|^{-\{-j\nu_i/N\}}.
\]
For any smooth top form $\Upsilon$ with values in ${\cL^{(j)}}^{-1} (D^{(j)}) \otimes_\C \otimes \overline{ {\cL^{(j)}}^{-1}(D^{(j)})}$ we can associate an integrable top form $(\Upsilon)_h=f\overline g |s|^2_h\, \mathrm{vol}(Z)$ by fixing a volume form $\mathrm{vol}(Z)$ on $Z$ and writing locally $\Upsilon=fs\otimes \overline{gs}\, \mathrm{vol}(Z)$ for $s$ a local frame of ${\cL^{(j)}}^{-1} (D^{(j)})$. Therefore, we obtain a well-defined pairing, 
\begin{equation}\label{eq:cyclpo}
\begin{aligned}
\cA^{k}_Z(\log D^{(j)})\otimes {\cL^{(j)}}^{-1} \wedge \overline{\cA^{k}_Z(\log D^{(j)})\otimes {\cL^{(j)}}^{-1}} \to \C, \\
(m', \overline{m''}) \mapsto \frac{\varepsilon( k+1)}{(2\pi\sqrt{-1})^{n}}\int_{ Z}  \left(\sfX^{n-k} m'\wedge \overline{m''}\right)_h.  
\end{aligned}
\end{equation}
Since $\eta:\tilde Z\to Z$ is generic finite, we deduce from 
\[
\int_{\tilde{Z}} \eta^*\left(\sfX^{n-k} m'\wedge \overline{m''} \right)=N\int_Z   \left(\sfX^{n-k} m' \wedge \overline{m''}\right)_h
\]
that~\eqref{eq:cyclpo} induces the same polarization in the statement $(2)$ of the above lemma except for the constant $N$.

Back to our situation, we obtain that $\sV_{\alpha,J}(E^{\alpha,J})$ carries a canonical singular Hermitian metric $|-|_h$  with local {weight functions} $\prod_{j\in I\setminus I_\alpha}|z_j|^{-\{\alpha e_j\}}$ restricted on $Y^J$, where $z_i$ is the defining equation of $Y_i$. Provided the above two lemmas, the graded vector space
\[
\bigoplus_{k\in\Z} \bH^k\left(Y^J,\Omega_{Y^J}^{\bullet+\dim Y^J}(\log E^{\alpha,J})\otimes \sV_{\alpha,J}\right)
\]
is a polarized Hodge-Lefschetz structure of central weight $\dim Y^J$ for any non-empty subset $J$ of $I_\alpha$. Similarly to Example~\ref{sign} this is also determined by the filtered $\sD_{Y^J}$-module $(\cV_{\alpha,J},F_\bullet\cV_{\alpha,J})$ with the sesquilinear pairing 
$S_{\alpha,J}:\cV_{\alpha,J}\otimes_{\C}\overline{\cV_{\alpha,J}}\to \mathfrak C_{Y^J}$ is given by
\begin{equation}\label{cyclicpairing}
\left(\left[s_1 \otimes P_1\right], \overline{\left[s_2 \otimes P_2\right]}\right)\mapsto \frac{\varepsilon(\dim Y^J+1)}{(2\pi\sqrt{-1})^{\dim Y^J}}\int_{Y^J}  (P_1\overline{P_2}-) \left(s_1\wedge\overline{s_2}\right)_h \ 
\end{equation}
for local sections of $\cV_{\alpha,J}$ (see~\eqref{nrootdiag}) represented by $s_i\otimes P_i$ such that $s_i$ local sections of 
\[
\omega_{Y^J}(\log E^{\alpha,J})\otimes \sV_{\alpha,J}=\omega_{Y^J}\otimes \sV_{\alpha,J}(E^{\alpha,J})
\]
and $P_i$ is a differential operator $i=1,2$. Here, $(s_1\wedge \overline{s_2})_h$ is the top form induced by the singular Hermitian metric on $\sV_{\alpha,J}(E^{\alpha,J})$. Summarizing the results we have proved in this subsection:

\begin{cor}\label{cycHS}
In the given notation, the graded vector space 
\[
\bigoplus_{\ell\in \Z}\bH^\ell \left(Y^J,\DR_{Y^J}\cV_{\alpha,J}\right)
\] 
underlies a polarized Hodge-Lefschetz structure of central weight $\dim Y^J$ with the Hodge filtration induced by the sub complexes $F_\bullet\DR_{Y^J}\cV_{\alpha,J}$ and the polarization given by $\bigoplus_k S_{k}=\bigoplus_k\varepsilon(k)S_{k}'$, where
\[
\begin{aligned}
  S_{k}'\colon= \bH^k(Y^J,\DR_{Y^J}\cV_{\alpha,J})\otimes \overline{\bH^{-k}(Y^J,\DR_{Y^J}\cV_{\alpha,J})} &\to \bH^0(Y^J,\DR_{Y^J,\overline{Y^J}}\cV_{\alpha,J}\otimes_\C \overline{\cV_{\alpha,J}})\\
  & \xrightarrow{S_{\alpha,J}}  \bH^0(Y^J,\DR_{Y^J,\overline{Y^J}}\mathfrak{C}_{Y^J})\simeq \C.
\end{aligned}
\] 
\end{cor}

\begin{rem}
We cannot make the Hodge structure in the above corollary over $\Q$ because there is no eigenvalue decomposition of $\Q$-structure.
\end{rem}

\subsection{Sesquilinear pairing}

As in the reduced case, we need a sesquilinear pairing to construct the limiting mixed Hodge structure. The construction for the reduced case still works with a little modification. Note that for any test function \(\eta\) over a local chart $U$ and two local sections $\zeta_1\otimes P_1,\zeta_2\otimes P_2 \in \bH^0\left(U,\Omega^n_{X/\Delta}(\log Y)(-\ceil{\alpha Y})\otimes \sD_X\right)$, the function
\[
t\mapsto \frac{\varepsilon(n+1)}{(2\pi\sqrt{-1})^n}\int_{X_t} P_1\overline{P_2}(\eta) \zeta_1\wedge \overline{\zeta_2}.
\] 
may have order approximately at most ${|t|^{2\alpha}\left(-\log |t^2| \right)^{k}}$ near $t=0$ where $k+1$ is the number of components of $Y_{I_\alpha}$ that intersect in $U$. This suggests that we can define the pairing $S_\alpha$ on $\cM_\alpha$ by
\[
\begin{aligned}
&\langle S_\alpha([\zeta_1\otimes P_1],\overline{[\zeta_2\otimes P_2]}), \eta\rangle \\
&\colon = \Res_{s=-\alpha}\frac{\varepsilon(n+2)}{(2\pi\sqrt{-1})^{n+1}}\int_X |t|^{2s}P_1\overline{P_2}(\eta)\frac{dt}{t}\wedge \zeta_1 \wedge\overline{\frac{dt}{t}\wedge \zeta_2} \\
  &=\Res_{s=-\alpha}\frac{\varepsilon(2)}{2\pi\sqrt{-1}} \int_\Delta |t|^{2s} \frac{dt}{t}\wedge \overline{\frac{dt}{t}}\left(\frac{\varepsilon(n+1)}{(2\pi\sqrt{-1})^n}\int_{X_t} P_1\overline{P_2}(\eta) \zeta_1\wedge \overline{\zeta_2}\right).
\end{aligned}
\]
We have not yet shown that $S_\alpha$ is well-defined but let us perform some local calculations to gain some heuristic understanding.
\begin{ex}
Suppose $Y=2Y_0$ for $Y_0$ is a smooth manifold and $t$ is equal to $z^2_0$ on $X$. Then $R$ satisfies the equation $R(R-\frac{1}{2})=0$. We deduce that $\cM$ has two eigenspaces $\cM_0$ and $\cM_{\frac{1}{2}}$ by~\eqref{poly}. Then for any local sections $\zeta_i\otimes P_i=dz_1\wedge dz_2\wedge\cdots\wedge dz_n\otimes P_i$ of $\Omega^n_{X/\Delta}(\log Y)\otimes\sD_X$, $i=1,2$ representing classes of $\cM_0$, the calculation of the pairing $S_0([\zeta_1\otimes P_1],\overline{[\zeta_2\otimes P_2]})$ is exactly as in the reduced case and as it turned out 
\[
S_0([\zeta_1\otimes P_1],\overline{[\zeta_2\otimes P_2]})={i_{Y_0}}_+S_{Y_0}\left([\zeta_1\otimes P_1],\overline{[\zeta_2\otimes P_2]}\right).
\]

By Theorem~\ref{genR} $\cM_{\frac{1}{2}}$ is locally generated by the class represented by $dz_1\wedge dz_2\wedge\cdots\wedge dz_n\otimes z_0$. Now for any local sections $\zeta\otimes z_0P_i=dz_1\wedge dz_2\wedge\cdots\wedge dz_n\otimes z_0P_i$ representing classes of $\cM_{\frac{1}{2}}$, we have 
\[
\begin{aligned}
&\langle S_{\frac{1}{2}}([\zeta \otimes z_0 P_1], \overline{[\zeta \otimes z_0 P_2]}),\eta \rangle \\
&= \Res_{s=-\frac{1}{2}} \int_X |z_0|^{4s} P_1\overline{P_2}(\eta)  \bigwedge^n_{i=0}\left(\frac{\sqrt{-1}}{2\pi}dz_i\wedge d\overline{z_i}\right) \\ 
  &=  \int_X \frac{1}{2}\log |z_0|^2 \partial_0\overline{\partial_0} P_1\overline{P_2}(\eta)  \bigwedge^n_{i=0}\left( \frac{\sqrt{-1}}{2\pi} dz_i\wedge d\overline{z_i} \right) \\
  &= \int_{Y_0} \frac{1}{2}   P_1\overline{P_2}(\eta)  \bigwedge^n_{i=1}\left( \frac{\sqrt{-1}}{2\pi} dz_i\wedge d\overline{z_i} \right) \\
    & \quad  \text{by Poincar\'e-Lelong equation~\cite[Page 388]{GH}} \\
  & = \frac{1}{2}\langle {i_{Y_0}}_+S_{Y_0} (\left[\zeta_1\otimes P_1\right], \overline{\left[\zeta_2\otimes P_2\right]}),\eta \rangle \\
  & =\frac{1}{2}\langle {i_{Y_0}}_+S_{\frac{1}{2},\{0\}}(\left[\zeta_1\otimes z_0P_1\right], \overline{\left[\zeta_2\otimes z_0P_2\right]}),\eta \rangle,
\end{aligned} 
\] 
Recall $S_{\frac{1}{2},\{0\}}$ defined in~\eqref{cyclicpairing}: since we have the isomorphism $\sO_{Y_0}(2Y_0)= \sO_{Y_0}(Y)\simeq \sO_{Y_0}$ there exists a canonical singular Hermitian metric (this case is smooth) $|-|_h$ on $\sO_{Y_0}(-Y_0)$ by setting the local frame ${z_0}$ has norm $1$ so that 
\[
\begin{aligned}
&{i_{Y_0}}_+S_{\frac{1}{2},\{0\}}(\left[\zeta_1\otimes z_0P_1\right], \overline{\left[\zeta_2\otimes z_0P_2\right])},\eta \rangle\\
&=\int_X |z_0|^2_h P_1\overline{P_2}(\eta)\bigwedge^n_{i=0}\left( \frac{\sqrt{-1}}{2\pi} dz_i\wedge d\overline{z_i} \right) \\
&={i_{Y_0}}_+S_{Y_0} (\left[\zeta_1\otimes P_1\right], \overline{\left[\zeta_2\otimes P_2\right]}),\eta \rangle.  
\end{aligned}
\]
The above equality can also be explained as follows: the cyclic covering constructed by taking out of the second root of the constant section of  $\sO_{Y_0}(2Y_0)\simeq \sO_{Y_0}$ has two connected components and each component is isomorphic to $Y_0$.
\end{ex}

Let $\eta$ be a test function over an open subset $U$. For any two sections $m_1,m_2\in H^0(U, \Omega^{n}_{X/\Delta}(\log Y)(-\ceil{\alpha Y})\otimes \sD_X)$, the $(2n+2)$-form $\frac{dt}{t}\wedge m_1 \wedge \overline{\frac{dt}{t}\wedge m_2}$ is smooth of outside $Y$ and has pole along $Y$. Locally, the $(2n+2)$-form just is $|z_I|^{2\ceil{\alpha \mathbf{e} }}P_1\overline{P_2}(\eta)\frac{dt}{t}\wedge\zeta \wedge \overline{\frac{dt}{t}\wedge \zeta}$, where $m_j=\zeta\otimes z^{\ceil{\alpha \mathbf{e} }}_I P_j$ for $\zeta=\frac{dz_1}{z_1}\wedge\frac{dz_2}{z_2}\wedge\cdots \wedge \frac{dz_k}{z_k}\wedge\cdots \wedge dz_n$ and $j=1,2$. Let $F(s)=F(s,m_1,m_2,\eta)$ be the meromorphic extension of 
\[
\frac{\varepsilon(n+2)}{(2\pi\sqrt{-1})^{n+1}}\int_X |t|^{2s} \frac{dt}{t}\wedge m_1 \wedge \overline{\frac{dt}{t}\wedge m_2}(\eta)
\]
via integration by parts. The function \(F(s)\) is well defined when \(\mathrm{Re\,} s>-\alpha\) and has a pole at $s=-\alpha$. We only care about the polar part of $F(s)$ at $s=-\alpha$.

\begin{thm}
The polar part of $F(s)$ at $s=-\alpha$ only depends on the classes of $m_1$ and $m_2$ in $\cM_\alpha$.
\end{thm}
\begin{proof}
Let $\{\rho_\lambda\}$ be a partition of unity of the open covering $\{U_\lambda\}$ by local charts. Then 
\[
F(s)=\sum_\lambda\frac{\varepsilon(n+2)}{\left(2\pi\sqrt{-1}\right)^{n+1}}\int_{U_\lambda} |t|^{2s} \frac{dt}{t}\wedge m_1\wedge \overline{\frac{dt}{t}\wedge m_2}(\rho_\lambda \eta).
\]
Since $\rho_\lambda\eta$ is a test function over local chart $U_\lambda$, we can assume $U$ itself is a local chart. We assume $k+1$ components of $Y$ intersect in $U$.

\begin{lem} \label{func}
Under the assumption that $m_i=\zeta_\alpha \otimes P_i$ for $\zeta_\alpha=z^{\ceil{\alpha \mathbf{e} }}_I \frac{dz_1}{z_1}\wedge\frac{dz_2}{z_2}\wedge\cdots\wedge\frac{dz_k}{z_k}\wedge dz_{k+1}\wedge \cdots\wedge dz_n$ and for $i=1,2$, the followings are valid.
\begin{enumerate}
  \item the order of the pole of $F(s)$ at $s=-\alpha$ is at most $k+1$;
  \item if $P_i=t_\alpha P'_i$ for one of $i=1,2$, then $F(s)$ is holomorphic at $s=-\alpha$;
  \item for $0\leq j\leq k$ we have, 
    \[
    \begin{aligned}
    &F\left(s,\zeta_\alpha\otimes P_1, \overline{\zeta_\alpha\otimes\frac{1}{e_j}z_j\partial_j P_2},\eta\right) \\
    =& F\left(s, \zeta_\alpha\otimes \frac{1}{e_j}z_j\partial_j P_1,\overline{\zeta_\alpha\otimes P_2},\eta\right)  \\
     =&-\left(s+\frac{\ceil{\alpha e_j}}{e_j}\right) F(s, \zeta_1\otimes P_1,\overline{\zeta_2\otimes P_2},\eta).
    \end{aligned}
  \]
\end{enumerate}
\end{lem}

\begin{proof}[Proof of the lemma]
We work out Laurent series of $F(s)$ at $s=-\alpha$:
\[
\begin{aligned}
F(s) &= \int_X |z_I|^{2s\mathbf{e}+2\ceil{\alpha \mathbf{e} }-2 \cdot \mathbf{1}} P_1\overline{P_2}(\eta) \bigwedge^{n}_{i=0} (\frac{\sqrt{-1}}{2\pi}dz_i\wedge d\overline{z_i}) \\
  &= \int_X |z_I|^{2(s+\alpha) \mathbf{e}-2 \cdot \mathbf{1}}|z_I|^{2\{-\alpha \mathbf{e}\}} P_1\overline{P_2}(\eta) \bigwedge^{n}_{i=0} (\frac{\sqrt{-1}}{2\pi}dz_i\wedge d\overline{z_i}) \\
  &= \int_X (s+\alpha)^{-2(k+1)} |z_I|^{2(s+\alpha)\mathbf{e} }\eta' \bigwedge^{n}_{i=0} (\frac{\sqrt{-1}}{2\pi}dz_i\wedge d\overline{z_i})   \\
  &= \sum^{\infty}_{\ell=0} \frac{1}{\ell !} (s+\alpha)^{\ell-2(k+1)}\int_X (\log |z_I|^{2\mathbf{e} })^\ell  \eta'  \bigwedge^{n}_{i=0} (\frac{\sqrt{-1}}{2\pi}dz_i\wedge d\overline{z_i}),
\end{aligned}
\]
where $\eta'\colon= {\left(\prod_{i\in I}e_i^{-2}\right)\partial_I\overline{\partial_I} \left(|z_I|^{2\{-\alpha \mathbf{e}\}}P_1\overline{ P_2}\eta\right)}$ When $\ell < k+1$, then the form 
\[
(\log |z_I|^{2\mathbf{e} })^\ell  \eta'  \bigwedge^{n}_{i=0} (\frac{\sqrt{-1}}{2\pi}dz_i\wedge d\overline{z_i}).
\]
is actually exact because one of the $a_i$ must be zero in the expansion of $(\log |z_I|^{2\mathbf{e} })^\ell$ into the linear combination of $\prod^{k}_{i=0} \left(\log |z_i|^{2e_i}\right)^{a_i}$ such that $\sum^k_{i=0}{a_i}=\ell$. Therefore, the order of the pole at $s=-\alpha$ is at most $k+1$.

When $P_1=t_\alpha P'_1$, the form 
\[
|z_I|^{2(s+\alpha) \mathbf{e}-2 \cdot \mathbf{1}}|z_I|^{2\{-\alpha \mathbf{e}\}} t_\alpha P'_1\overline{P_2}(\eta) \bigwedge^{n}_{i=0} (\frac{\sqrt{-1}}{2\pi}dz_i\wedge d\overline{z_i})
\]
is integrable when $s=-\alpha$ where $\{-\alpha \mathbf{e}\}$ is the multi-index such that $\{-\alpha \mathbf{e}\}_i=\{-\alpha e_i\}$. Therefore, $F(s)$ is holomorphic at $s=-\alpha$. It is the same when $P_2=t_\alpha P'_2$.  

Lastly, by linearity, we can assume that $P_1=P_2=1$.
\begin{equation}\label{conju}
\begin{split}
  &F(s, \zeta_\alpha\otimes 1, \overline{\zeta_\alpha\otimes \frac{1}{e_j}z_j\partial_j},\eta) \\ 
  =& \frac{\varepsilon(n+2)}{(2\pi\sqrt{-1})^{n+1}}\int_X |t|^{2s}\left(\frac{1}{e_j}\overline{z_j\partial_j}\eta\right)\frac{dt}{t}\wedge \zeta_\alpha\wedge\overline{\frac{dt}{t}\wedge \zeta_\alpha}  \\
  =& \int_X\prod_{i\in I\setminus \{j\}}|z_i|^{2se_i+2\ceil{\alpha e_i}-2}  z^{se_j+\ceil{\alpha e_j}-1}_j\frac{1}{e_j}\overline{z^{se_j+\ceil{\alpha e_j}}_j \partial_j}\eta\bigwedge^n_{i=0}(\frac{\sqrt{-1}}{2\pi}dz_i\wedge{d\overline{z_i}})   \\
  =& \int_X -\left(s+\frac{\ceil{\alpha e_j}}{e_j}\right)\prod_{i\in I}|z_i|^{2se_i+2\ceil{\alpha e_i}-2}\eta\bigwedge^n_{i=0}(\frac{\sqrt{-1}}{2\pi}dz_i\wedge d\overline{z_i}) \\
  =& -\left(s+\frac{\ceil{\alpha e_j}}{e_j}\right)\frac{\varepsilon(n+2)}{(2\pi\sqrt{-1})^{n+1}} \int_X  |t|^{2s}\eta\frac{dt}{t}\wedge \zeta_\alpha\wedge\overline{\frac{dt}{t}\wedge \zeta_\alpha}. \\
  =& -\left(s+\frac{\ceil{\alpha e_j}}{e_j}\right)F(s, \zeta_\alpha\otimes 1,\overline{\zeta_\alpha\otimes 1},\eta). 
\end{split}
\end{equation}
The other equality in $(3)$ holds similarly. 
\end{proof}
Returning to the proof of the theorem. Since $\cM_\alpha$ is locally represented by 
\[
\zeta_\alpha\otimes \sD_X/(t_\alpha, D_1+\alpha_1,D_2+\alpha_2,\dots,D_n+\alpha_n)\sD_X
\] 
(see the proof of Theorem~\ref{Alpha}), and $(2)$ and $(3)$ in the lemma say that when one of $m_1$ and $m_2$ is in 
\[
\zeta_\alpha\otimes (t_\alpha, D_1+\alpha_1,D_2+\alpha_2,\dots, D_n+\alpha_n)\sD_X
\] 
then $F(s)$ is holomorphic since $\alpha_i$ equals $\ceil{\alpha e_i}/e_i-\ceil{\alpha e_0}/e_0$ for $1\leq i\leq k$ and equals zero otherwise.
\end{proof}

For two sections $\gamma_1,\gamma_2\in H^0(U, \cM)$ and any test function $\eta$ over $U$, we define the pairing $S_\alpha: \cM_\alpha\otimes_\C \overline{\cM_\alpha}\to \mathfrak C_X$ by 
\[
\langle S_\alpha ( \gamma_1, \overline{\gamma_2}),\eta \rangle=\Res_{s=-\alpha} \sum_{\lambda}F(s, \tilde \gamma_1, \overline{\tilde \gamma_2},\rho_\lambda\eta),
\]
 where $\{\rho_\lambda\}$ is a partition of unity with respect to an open covering by local charts $\{U_\lambda\}$ such that $\gamma_i$ has a local lifting of $\tilde \gamma_i$ over $U_\lambda$ for $i=1,2$. It is obvious that $S_\alpha$ is \(\sD_{X,\overline X}\)-linear. Thus, it is a sesquilinear pairing.  
As a corollary of Lemma~\ref{func}, we have
\begin{cor} \label{cor:srn}
We have \(S_\alpha\circ(\id\otimes_\C \overline{R_\alpha})=S_\alpha \circ({R_\alpha}\otimes_\C\id)\). 
\end{cor}

Because of the corollary, the sesquilinear pairing $S_\alpha$ induces pairings on the associated graded quotient of the weight filtration 
\[
S_\alpha:\gr^W_k\cM_\alpha \otimes_\C \overline{\gr^W_{-k}\cM_\alpha}\to \mathfrak{C}_X,
\]
as well as on the primitive part
\[
P_{R_\alpha}S_r=S_\alpha \circ \left(\id\otimes_\C \overline{R^r_\alpha}\right): \cP_{\alpha,r} \otimes_\C \overline{\cP_{\alpha,r}}\to \mathfrak{C}_X.
\]

\begin{thm}\label{mainN}
The isomorphism $\phi_{\alpha,r}: (\cP_{\alpha,r}, F_\bullet\cP_{\alpha,r})\to \cV_{\alpha,r}(-r)$ in Theorem~\ref{idenN} respects the sesquilinear pairings up to a constant scalar. More concretely, 
\[
P_{R_\alpha}S_r(m_1, \overline{m_2})=\bigoplus_{\substack{J\subset I_\alpha,\\ |J|=r+1}}\frac{(-1)^r}{(r+1)! C_J} \tau^{J}_+S_{\alpha,J}(\phi_{\alpha,r}m_1,\overline{\phi_{\alpha,r}m_2})
\]
for any local sections $m_1,m_2\in\cP_{\alpha,r}$ and $C_J=\prod_{j\in J}e_j$, where the pairing $S_{\alpha,J}:\cV_{\alpha,J}\otimes_{\C}\overline{\cV_{\alpha,J}}\to \mathfrak C_{Y^J}$ is defined in~\eqref{cyclicpairing}.
\end{thm}

\begin{proof}
Because the generators of $\cP_{\alpha,r}$ are all monomials dividing \(t_\alpha\) of degree \(\mu-r\) by Corollary~\ref{cor:locgenR}, it suffices to prove the theorem in the case when $m_i$ is represented by 
\[
\zeta_\alpha\otimes z_{{K_i}}=z^{\ceil{\alpha \mathbf{e}}}_{I} \frac{dz_1}{z_1}\wedge\cdots\wedge \frac{dz_k}{z_k}\wedge\cdots\wedge dz_n\otimes z_{K_i} \quad 
\]
where $K_i\subset I_\alpha$ with $|K_i|=\mu-r$ and $i=1,2$. Let $\eta$ be a test function over $U$. Then we have
\[
\langle S_\alpha ( m_1, \overline{R^r_\alpha m_2}),\eta\rangle= \Res_{s=-\alpha} (-(s+\alpha))^{r} \int_X |z_I|^{2s\mathbf{e}+2\ceil{\alpha\mathbf{e}}-2\cdot\mathbf{1}} z_{K_1}\overline{z_{K_2}} \bigwedge^n_{i=0} (\frac{\sqrt{-1}}{2\pi}dz_i\wedge\overline{dz_i}).
\]
If $m_1\neq m_2$, then the above is zero. Indeed, for $v\in K_2\setminus K_1$ we can make $R^r_\alpha$ be the left multiplication by $1\otimes \prod_{i\in I\setminus K_1\setminus \{v\}}\frac{1}{e_i}z_i\partial_i$ and then
\[
\langle S(R^r_\alpha m_1, \overline{m_2}),\eta \rangle=\Res_{s=-\alpha}  \int_X |z_I|^{2s\mathbf{e}-2\mathbf\cdot \mathbf{1}} |z_I|^{2\ceil{\alpha \mathbf{e}}} \frac{t_\alpha}{z_{v}}\overline{z_{v}}\tilde \eta \bigwedge^n_{i=0}(\frac{\sqrt{-1}}{2\pi}dz_i\wedge d\overline{z_i})
\]
where $\tilde \eta=C^{-1}_{I\setminus K_1\setminus \{v\}}\partial_{I\setminus K_1\setminus \{v\}}{\overline{z_{K_2}}}{\left(\overline{z_{v}}\right)^{-1}}\eta$ is a smooth function with compact support. The function 
\[
 \int_X |z_I|^{2s\mathbf{e}-2\mathbf\cdot \mathbf{1}} |z_I|^{2\ceil{\alpha \mathbf{e}}} \frac{t_\alpha}{z_{v}}\overline{z_{v}}\tilde \eta \bigwedge^n_{i=0}(\frac{\sqrt{-1}}{2\pi}dz_i\wedge d\overline{z_i})
\]
is holomorphic at $s=-\alpha$ because by setting $s=-\alpha$ the form
\[
 |z_I|^{-2\alpha\mathbf{e}-2\mathbf\cdot \mathbf{1}} |z_I|^{2\ceil{\alpha \mathbf{e}}} \frac{t_\alpha}{z_{v}}\overline{z_{v}}\tilde \eta \bigwedge^n_{i=0}(\frac{\sqrt{-1}}{2\pi}dz_i\wedge d\overline{z_i})=|z_{I\setminus I_\alpha}|^{-2\{\alpha \mathbf{e}\}}\frac{1}{\overline{t_\alpha}} \frac{\overline z_v}{z_v}\tilde \eta \bigwedge^n_{i=0}(\frac{\sqrt{-1}}{2\pi}dz_i\wedge d\overline{z_i})
\]
is integrable.

Therefore, we reduce the proof to the case when $m_1=m_2=m$ represented by $\zeta_\alpha\otimes z_K$. We shall prove that 
\[
S_\alpha(m, \overline{R^r_{\alpha}m})=\frac{(-1)^{r}}{(r+1)!C_{\overline{K}}} \tau^{\overline K}_+S_{\alpha,\overline{K}} (\phi_{\alpha,r}m, \overline{\phi_{\alpha,r}m}),
\]
where $\overline K$ is the complement of $K$ in $I_\alpha$. Without loss of generality, we can assume that $K=\{r+1,r+2,\dots,\mu\}$ and $\overline K=\{0,1,\dots,r\}$ so that $z_{K}=z_{r+1}z_{r+2}\cdots z_{\mu}$. We have
\begin{equation}\label{eq:resmnon}
   \langle S(m, \overline{R^r_\alpha m}),\eta\rangle= \Res_{s=-\alpha} g(s)    
 \end{equation}
 if we put 
 \[
  g(s)\colon=(-(s+\alpha))^r \int_X |z_K|^{2(s+\alpha) \mathbf{e}_K} |z_{I\setminus K}|^{2s \mathbf e_{I\setminus K}+2\ceil{\alpha \mathbf e_{I\setminus K}}-2}\eta \bigwedge^n_{i=0}(\frac{\sqrt{-1}}{2\pi}dz_i\wedge d\overline{z_i})
 \]
 where, for any index subset $J\subset I$, the $j$-th component of $\ceil{\alpha \mathbf e_J}$ is $\ceil{\alpha e_j}$ if $j\in J$ or zero otherwise.
Integration by parts for $\{dz_i,d\bar z_i\}_{i\in \overline{K}}$, %the identity~\eqref{eq:resmnon} equals to
\[
\begin{aligned}
g(s)=\int_X \frac{|z_{I_\alpha}|^{2(s+\alpha) \mathbf e_{I_\alpha}}}{C^2_{\overline K}(s+\alpha)^{2r+2}}|z_{I\setminus I_\alpha }|^{2s \mathbf e_{I\setminus I_\alpha}+2\ceil{\alpha \mathbf e_{I\setminus I_\alpha}}-2}\left(\partial_{\overline K}\overline{\partial_{\overline K}} \eta\right)\bigwedge^n_{i=0}(\frac{\sqrt{-1}}{2\pi} dz_i\wedge d\overline{z_i}) 
\end{aligned}
\]
\begin{equation} \label{eq:resmnon1}
\begin{aligned}
  = \frac{(-1)^r}{C^2_{\overline K}(s+\alpha)^{r+2}} \int_X  {  |t|^{2 (s+\alpha )}} \prod_{j\in I\setminus I_\alpha}|z_j|^{-2\{\alpha e_j\}}\left(\partial_{\overline K}\overline{\partial_{\overline K}}\eta\right)\bigwedge^n_{i=0}(\frac{\sqrt{-1}}{2\pi}dz_i\wedge d\overline{z_i}), 
\end{aligned}
\end{equation}
where $\partial_{\overline K}\overline{\partial_{\overline K}}= \prod_{j\in \overline{K}}\partial_j\overline{\partial_j}$. Because of the expansion
\[
|t|^{2(s+\alpha)}=\exp\left({\log |t|^2 \left(s+\alpha\right)}\right)=\sum^\infty_{\ell=0} \frac{\left( \log |t|^2\right)^\ell(s+\alpha)^\ell}{\ell!},
\]
we find that by~\eqref{eq:resmnon1} 
\begin{equation} \label{eq:resmnon2}
\begin{aligned}
  &\Res_{s=-\alpha} g(s)= \\
   &\frac{(-1)^r}{C^2_{\overline K} (r+1)!} \int_X {\left(\log |t|^2 \right)^{r+1}} \prod_{j\in I\setminus I_\alpha}|z_j|^{-2\{\alpha e_j\}} \left(\partial_{\overline K}\overline{\partial_{\overline K}}\eta\right)\bigwedge^n_{i=0}(\frac{\sqrt{-1}}{2\pi} dz_i\wedge d\overline{z_i}) 
\end{aligned}
\end{equation}
The expansion of \(\left(\log |t|^2 \right)^{r+1}\) is a linear combination of 
\[
\prod_{i\in I} \left(\log |z_{i}|^2 \right)^{a_i}
\]
for all partitions \(\sum_{i\in I}a_i=r+1\), but the differential form
\[
\prod_{i\in I} \left(\log |z_{i}|^2 \right)^{a_i}\prod_{j\in I\setminus I_\alpha}|z_j|^{-2\{\alpha e_j\}}\left(\partial_{\overline K}\overline{\partial_{\overline K}}\eta\right)\bigwedge^n_{i=0}(\frac{\sqrt{-1}}{2\pi}dz_i\wedge d\overline{z_i})
\]
is exact unless \(a_i\neq 0\) for any $i\in \overline K$, which is equivalent to $a_i=1$ for $i\in \overline K$ and $a_i=0$ for $i\notin \overline K$. It follows from~\eqref{eq:resmnon2} that $\Res_{s=-\alpha} g(s)$ is equal to
\[
\frac{(-1)^r}{C_{\overline K}(r+1)!} \int_X { \prod_{j\in \overline K}\log |z_j|^2}\prod_{j\in I\setminus I_\alpha}|z_j|^{-2\{\alpha e_j\}}\left(\partial_{\overline K}\overline{\partial_{\overline K}}\eta \right)\bigwedge^n_{i=0}(\frac{\sqrt{-1}}{2\pi} dz_i\wedge d\overline{z_i}).
\]
An application of Poinc\'are-Lelong equation~\cite[Page 388]{GH} gives 
\begin{equation}\label{eq:resmnon3}
 \Res_{s=-\alpha} g(s)=\frac{(-1)^r}{(r+1)!C_{\overline K}} \int_{Y^{\overline K}} \prod_{j\in I\setminus I_\alpha}|z_j|^{-2\{\alpha e_j\}}\eta\bigwedge^n_{i=r+1}(\frac{\sqrt{-1}}{2\pi}dz_i\wedge d\overline{z_i}) 
\end{equation}
Since $\phi_{\alpha,\overline K}m=\pm \frac{dz_{I\setminus K}}{z_{I\setminus K}}\wedge dz_{k+1}\wedge \cdots \wedge dz_n\otimes s_{\alpha,\overline K}$ is a local section of $\omega_{Y^{\overline K}}(E^{\alpha,\overline K})\otimes \sV_{\alpha,\overline K}$, we deduce that
\[
(\phi_{\alpha,\overline K}m\wedge \overline{\phi_{\alpha,\overline K}m})_h= \prod_{j\in I\setminus I_\alpha}|z_j|^{-2\{\alpha e_j\}}\bigwedge^n_{i=r+1}(\frac{\sqrt{-1}}{2\pi}dz_i\wedge d\overline{z_i}) 
\] 
from which we conclude that \eqref{eq:resmnon3} is equal to 
\[
  \frac{(-1)^r}{(r+1)!C_{\overline K}} \int_{Y^{\overline K}} \eta (\phi_{\alpha,\overline K}m\wedge \overline{\phi_{\alpha,\overline K}m})_h=\frac{(-1)^r}{(r+1)!C_{\overline K}} \langle S_{\alpha, \overline{K}}(\phi_{\alpha,\overline K}  m ,\phi_{\alpha,\overline{K}}  m),\eta\rangle.
\]
See~\eqref{cyclicpairing}. We have concluded the proof.
\end{proof}

\subsection{Construction of the limiting mixed Hodge structure}

We begin to construct a polarized bigraded Hodge-Lefschetz structure on $\gr^W \bH^\bullet(X,\DR_X\cM_\alpha)$. Fix a K\"ahler class $\omega$ on $X$ and let $L=\omega\wedge:\DR_X\cM_\alpha \to \DR_X\cM_\alpha [2]$ be the Lefschetz operator and $\sfX_1=2\pi\sqrt{-1}L$. Relabel the graded pieces of the first page of the weight spectral sequence by
\[
V^\alpha_{\ell,k}=\bH^{\ell}(X,\gr^W_k\DR_X\cM_\alpha)=\prescript{W}{}{E^{-k,\ell+k}_1}.
\]
Let $V^\alpha=\bigoplus_{\ell,k\in\Z}V^\alpha_{\ell,k}$ with the filtration $F_\bullet V^\alpha$ induced by $F_\bullet\cM_\alpha$. Denote by $E_i(R_\alpha)$ the induced operator by $R_\alpha$ on $\prescript{W}{}{E_i}$ and let $\sfY_2=E_1(R_\alpha)$. Denote by $S_{\ell,k}=\varepsilon(\ell)S_{\ell,k}'$ the induced pairing on $V^\alpha_{\ell,k}\otimes \overline{V^\alpha_{-\ell,-k}}$ with
\[
\begin{aligned}
  S_{\ell,k}'\colon &\bH^{\ell}(X,\gr^W_k\DR_X\cM_\alpha)\otimes \overline {\bH^{-\ell}(X,\gr^W_{-k}\DR_X\cM_\alpha}) \to \\
  &\bH^0(X,\DR_{X,\overline X}\gr^W_k\cM_\alpha\otimes_\C \overline{\gr^W_{-k}\cM_\alpha}) \to \bH^0_c(X, \DR_{X,\overline X}\mathfrak {C}_X) \simeq \C.
\end{aligned}
\]
Let $d_1$ be the differential of the first page of the spectral sequence. In terms of relabeling, we have 
\[
d_1: (V^\alpha_{\ell,k},F_\bullet V^\alpha_{\ell,k}) \to (V^\alpha_{\ell+1,k-1},F_\bullet V^\alpha_{\ell+1,k-1}).
\]
Exactly same to Theorem~\ref{rd1pbhl} and Corollary~\ref{rmhc} in the reduced case, we conclude that 

\begin{thm} \label{d1pbhl}
The tuple $(V^\alpha,\sfX_1,\sfY_2,F_\bullet V, \bigoplus S_{j,k}, d_1)$ gives a differential polarized bigraded Hodge-Lefschetz structure of central weight $n$. 
\end{thm}

\begin{cor}\label{mhc}
We have the following 
\begin{enumerate}
  \item the Hodge spectral sequence degenerates at $\prescript{}{F}{E_1}$;
  \item the weight spectral sequence degenerates at $\prescript{W}{}{E_2}$;
  \item the tuple $(\bigoplus_{\ell\in \Z}\gr^W\bH^\ell(X,\DR_X\cM_\alpha), \sfX_1,\sfY_2,F_\bullet)$ together with the pairing induced by $S_\alpha$ is a polarized bigraded Hodge-Lefschetz structure of central weight $n$.
\end{enumerate}
\end{cor}

The last statement in Corollary~\ref{mhc} implies that the monodromy filtration associated to $R_\alpha$ on $\bH^\ell(X,\DR_X\cM_\alpha)$ is the correct weight filtration. Thus, we have established Theorem~\ref{thm:main}.

The following is an immediate consequence of the existence of the bigraded limiting Hodge-Lefschetz structure.

\begin{cor}[Hard Lefschetz]
The Lefschetz operator induces an isomorphism between $\sO_\Delta$-modules
\[
\left(2\pi\sqrt{-1}L\right)^k: F_\ell \bR^{-k}f_* \relome{\bullet+n} \simeq F_{\ell-k} \bR^{k}f_* \relome{\bullet+n} \quad \text{for any integer } \ell.
\]
 As a result, we have the following decomposition in the derived category of coherent $\sO_\Delta$-modules:
\[
\bR f_* F_\ell \relome{\bullet+n} \simeq \bigoplus_{k \in\Z} F_\ell \bR^k f_*\relome{\bullet+n}[-k] \quad \text{ for any integer } \ell.
\]
\end{cor}
\begin{proof}
The first statement follows from the Hard Lefschetz on each fiber 
\[
\left(2\pi\sqrt{-1}L\right)^k:  F_\ell \bR^{-k} f_* \relome{\bullet+n}\otimes \C(p) \simeq F_{\ell-k}\bR^{k}f_* \relome{\bullet+n}\otimes \C(p),
\]
for every $p\in\Delta$. The second statement follows from the first one plus the main theorem in~\cite{Deligne68}.
\end{proof}

\section{Local invariant cycle theorem}\label{sec:app}

Let $Z$ be a complex a complex manifold of $\dim Z=n$ and $D=\sum_{i\in K} D_i$ be a simple normal crossing divisor where $D_i$'s are smooth divisors. Let $D_J=\sum_{j\in J} D_j$ for some index set $J$. The log de Rham complex induced by the canonical morphism $\sO_Z(-D_J)\otimes \sD_Z\to \Omega_Z(\log D_J)(-D_J)\otimes\sD_Z$:
\begin{equation}\label{eq:!log}
\Omega^{n+\bullet}_Z(\log D_J)(-D_J)\otimes \sD_Z
\end{equation}
is a filtered resolution of a holonomic $\sD_Z$-module $\omega_Z(!D_J)$, which can be proved using the idea in the proof of Lemma~\ref{lem:filresolution}.  Indeed, $\omega_Z(!D_J)$ is the holonomic dual of $\omega_Z(*D_J)$, by noticing that  $\Omega^{k}_Z(\log D_J)(-D_J) \simeq \bR\cHom_{\sO_Z}(\Omega^{m-k}_Z(\log D_J),\omega_Z)$ but we will not use this fact. The short exact sequence
\[
\begin{aligned}
0 \to \Omega^{n+\bullet}_Z(\log D_{J\cup \{k\}})(-D_{J\cup \{k\}})&\otimes \sD_Z \to \Omega^{n+\bullet}_Z(\log D_J)(-D_J)\otimes \sD_Z\\
 &\to \Omega^{n+\bullet}_{D_k}(\log D_{J}\cup {D_k})(-\log D_{J}\cup {D_k})\otimes\sD_Z \to 0
\end{aligned}
\]
indicates that, by taking the cohomology of the induced long exact sequence, 
\begin{equation}\label{eq:cone!}
0 \to \tau^{k}_+\omega_{D_k}(!  D_{J}\cap {D_k}) \to \omega_Z(! D_{J\cup \{k\}}) \to \omega_Z(! D_J) \to 0
\end{equation}
is a filtered exact sequence. Here, $\tau^{k}:D_k\to Z$ is the closed embedding. 

We also need that the morphism of $\omega_Z(!D) \to \omega_Z(*D)$ factors through $\omega_Z$ because 
\[
\Omega_Z^{n+\bullet}(\log D)(-D)\otimes \sD_Z\to \Omega_Z^{n+\bullet}(\log D)\otimes \sD_Z
\]
factors through $\Omega_Z^{n+\bullet}\otimes\sD_Z$. Since $\omega_Z$ is simple on each connected component, we have 
\begin{equation}\label{eq:st18}
\omega_Z(!D) \twoheadrightarrow \omega_Z \hookrightarrow \omega_Z(*D).
\end{equation}
Moreover, all morphisms strictly respect the good filtrations thanks to
\[
F_\ell\omega_Z(!D)=F_\ell\omega_Z(*D)=F_\ell\omega_Z=0
\] 
for $\ell< -n$ and 
\[
\underbrace{\Omega^n_Z(\log D)(-D)}_{=F_{-n} \omega_Z(!D)} \twoheadrightarrow \omega_Z \hookrightarrow \underbrace{\Omega_Z^n(\log D)}_{=F_{-n}\omega_Z(*D)}
\]
by~\eqref{eq:!log} and~\eqref{starex}.

Now we present the proof of Theorem~\ref{thm:lic}, which is equivalent to the following statement, using the same notation introduced in \S\ref{sec:malpha}:

\begin{thm}\label{loc}
The following sequence is exact in the category of mixed Hodge structures:
\[
\begin{tikzcd}[column sep=small]
H^{\ell}(Y,\C[n])\arrow{r} & {\bH^\ell}(X,\DR_X\cM) \arrow{r}{R} & {\bH^\ell}(X,\DR_X\cM)(-1).
\end{tikzcd}
\]
\end{thm}

\begin{proof}
Note that \({\ker } R\) is contained in \(\cM_0\) and \(W_{-j}{\ker } R=R^{j}{\ker } R^{j+1}\) for \(j\geq 0\) and is trivial for \(j< 0\) where $W$ is the filtration induced by \(W=W(R)\) on \(\cM_0\). It follows that $(\gr^W_{-j}\ker  R,F)\simeq (\omega_{\tilde Y^{(j+1)}},F)$ for $j\geq 0$ by Theorem~\ref{idenN}. Since $\gr^W_{-j} {\ker } R$ is a summand of $\gr^W_{-j}\cM_0$ for $j\geq 0$ according to the Lefschetz decomposition on $\gr^W\cM_0$, we have the following short exact sequence of Hodge structures on the first page of the weight spectral sequences:
\[
\begin{aligned}
  0\to {\bH^{\ell+\bullet}}(X,\gr^W_{-j-\bullet} \DR_X{\ker } R) \to E_1^{j+\bullet, \ell-j}  \xrightarrow{R}  E_1^{j+2+\bullet, \ell-j}(-1) \to 0,
\end{aligned}
\]
for $-j-\bullet\leq 1$ where $E_1^{j+\bullet, \ell-j}={\bH^{\ell+\bullet}}(X,\gr^W_{-j-\bullet} \DR_X\cM_0)$ is a Hodge structure of weight $n+\ell-j$. The associated long exact sequence gives the relation between the second page of the weight spectral sequences for $j\geq -1$: 
\[
 \to  \gr^W_{-j}{\bH^\ell}(X,\DR_X{\ker } R) \to \gr^W_{-j}{\bH^\ell}(X,\DR_X\cM_0)\xrightarrow{R} \gr^W_{-j-2}{\bH^\ell}(X,\DR_X\cM_0)(-1)\to .
\]
Then the proof would be completed if we can show that there is canonical quasi-isomorphism  $\DR_X\ker R \simeq \C_Y[n]$ and 
\[
(\bH^\ell(X,\DR_X\ker R),F,W) \simeq (\bH^\ell(X,\C_Y[n]),F,W)
\]  
as mixed Hodge structures.

We make use of the notation in the proof of Theorem~\ref{idenN}. Taking $\alpha=0$ in~\eqref{eq:lesnr} gives an exact sequence:
  \[
  0\to \cH^{-1}_0  \to \cM_0 \xrightarrow{R} \cM_0(-1) \to \cH^0_0(-1) \to 0,
  \]
  which implies that $(\cH^{-1}_0,F) \simeq (\ker R,F)$. Recall that $\cH^{j}_0=\sH^j(\Omega_X^{\bullet+n+1}(\log Y)|_{Y_{\rm red}}\otimes\sD_X)$. On the other hand, the filtered $\sD_X$-modules $\cH^{j}_0$ also fit into the exact sequence of filtered $\sD_X$-modules:
  \[
  0\to \cH^{-1}_0 \to \omega_X(!Y) \to \omega(*Y) \to \cH^0_0 \to 0
  \]
  because of~\eqref{eq:!log},~\eqref{starex} and the short exact sequence
  \[
  0\to \Omega_X^{\bullet+n}(\log Y)(-Y_{\rm Red})\otimes\sD_X \to \Omega_X^{\bullet+n}(\log Y)\otimes\sD_X \to \Omega_X^{\bullet+n}(\log Y)|_{Y_{\rm Red}}\otimes\sD_X \to 0.
  \] 
As the image of $\omega_X(!Y)\to \omega(*Y)$ is precisely $\omega_X$ by~\eqref{eq:st18}, we obtain a filtered short exact sequence
  \[
  0\to \cH^{-1}_0 \to \omega(!Y) \to  \omega_X\to 0.
  \]
  
Now we describe a weight filtration on $\omega_X(!Y)$ which also induces a weight filtration on $\cH^{-1}_0$. Define the \v{C}ech complex
\[
(K_J,F)=\left\{(\omega(!Y_J),F) \xrightarrow{d'} \bigoplus_{i\in J} (\omega_X(! Y_{J\setminus\{i\}}),F) \xrightarrow{d'} \cdots \xrightarrow{d'} (\omega_X,F) \right\}
\]
placed in degree $-|J|,\dots, -1,0$ for any subset $J$ of $I$, recalling that $Y_J=\bigcup_{j\in J}Y_j$. Indeed, $(K_J,F)\simeq (\tau^J_+ \omega_{Y^J},F)[|J|]$ by~\eqref{eq:cone!} because $(K_J,F)$ is the iterated mapping cones (up to a shift) of the canonical maps $\omega_X(!Y_i)\to \omega_X$. It follows that
\begin{equation}\label{eq:ceck}
(\DR_X K_J,F) \simeq \left(\Omega^{n+1+\bullet}_{Y^J},F\right).
\end{equation}
Consider the double complex
\[
K^{\bullet,\bullet}=\left\{\omega_X\xrightarrow{d''} \bigoplus_{i\in I} K_{\{i\}} \xrightarrow{d''} \bigoplus_{i,j\in I} K_{\{i,j\}}\xrightarrow{d''} \cdots \xrightarrow{d''} K_I \right\}
\]
where the $p$-th column $K^{\bullet,p}=\bigoplus_{|J|=p+1}K_{J}$ if $p\geq 0$, $K^{\bullet,-1}=\omega_X$ and the $q$-th row $K^{q,\bullet}$ is given by iterating the mapping cone (up to a shift) of the identity map $\omega_X(!Y_J) \to \omega_X(!Y_J)$ exactly $(|I|+q)$-many times (hence filtered acyclic):
\[
\begin{aligned}
\bigoplus_{|J|=-q}\left\{\omega_X(!Y_J)\xrightarrow{d''} \bigoplus_{|I|+q} \omega_X(!Y_J) \xrightarrow{d''} \cdots \xrightarrow{d''} \omega(!Y_J)\right\}[q+1] \quad \text{for} \quad -|I|<q\leq 0 
\end{aligned}
\]
and $\omega_X(!Y)[-|I|+1]$ for $(q=-|I|)$. It follows that the total complex $K\colon = \text{Tot}(K^{\bullet,\bullet})$ is filtered quasi-isomorphic to $K^{-|I|, |I|-1}[1]=\omega_X(!Y)[1]$ because all other the rows are filtered acyclic. We define the weight filtration by $W_{-j} K\colon=\text{Tot}(K^{\bullet, \geq j})$ and it induces a weight filtration on $\cH^{-1}(K)=\omega_X(!Y)$. It follows from the short exact sequence
\[
0\to W_{0} K \to \underbrace{\omega_X(!Y)[1]}_{=W_{1}K} \to \underbrace{\omega_X[1]}_{=\gr^W_1 K} \to 0
\]
that $(W_0K,F)\simeq (\cH^{-1}_0[1],F)$. Moreover, we see that $W_{-j}K$ is supported on cohomological degree $1$ for any $j\geq 0$ by descending induction and the short exact sequence
\[
0\to W_{-j-1} K \to W_{-j}K \to \underbrace{\tau^{(j+1)}_+ \omega_{\tilde Y^{(j+1)}}[1]}_{=\gr^W_{-j} K} \to 0.
\]
Therefore, $(W_{-j}\cH^{-1}_0[1],F)\simeq (W_{-j}K,F)$ for $j\leq 0$. 

A diagram chasing implies 
\begin{equation}\label{eq:weightr}
  W_{-j}\cH^{-1}_0\simeq \sH^{-1}(W_{-j}K) =\ker\left(\omega_X(!Y) \to \bigoplus_{J\subset I, |J|=j} \omega_X(! Y_J) \right).
\end{equation}
Indeed, since $d''\colon \omega_X(!Y_J)\to \bigoplus_{|I|+q} \omega_X(!Y_J)$ is injective for any subset $J\subset I$ with $|J|=-q$, for any element $m \in W_{-j}\omega_X(!Y)$, it has a unique lifting in $\sH^{-1}(W_{-j}K)$ represented by the cocycle
\[
m_{-j}+m_{-j-1}+\cdots +m_{-|I|+1}
\]
where $m_\ell\in K^{-\ell,\ell-1}$ is the image of $m$ under the morphism 
\[
\omega_X(!Y) \to \bigoplus_{J\subset I, |J|=-\ell+1} \omega_X(! Y_J)
\]
up to a sign (determined by the cocycle condition). The cocycle condition also imposes that $d'm_{-j}=0\in K^{j,-j}= \bigoplus_{|J|=j+1}\bigoplus_{i\in J} \omega_X(! Y_{J\setminus\{i\}})$ and hence, we have proved~\eqref{eq:weightr}.

By definition $\DR_X W_{0}K$ with the weight filtration $\DR_X W_{-j} K$ for $j\geq 0$ bifiltered quasi-isomorphic, by~\eqref{eq:ceck}, to the total complex of
\[
\left\{\bigoplus_{i\in I} \Omega^{n+1+\bullet}_{Y^{\{i\}}} \to \bigoplus_{i,j\in I} \Omega^{n+1+\bullet}_{Y^{\{i,j\}}}\to \cdots \to \Omega^{n+1+\bullet}_{Y^I} \right\}
\]
with the standard filtrations computing the mixed Hodge structure on the cohomology of $\C_Y[n+1]$; see~\cite[(4.2)]{GS75} or~\cite[(3.5)]{Ste76}. More precisely, we have
\[
(\bH^\ell(X,\DR_X\cH^{-1}_0),W,F) \simeq (H^\ell(Y,\C_Y[n]),F,W)
\]
as mixed Hodge structures.

Lastly, we prove that the isomorphism $\ker R \to\cH^{-1}_0$ preserves the weight filtrations, i.e., $W_{-j}\ker R \to W_{-j}\cH^{-1}_0$ is an isomorphism, which will imply that 
\[
(\bH^\ell(X,\DR_X\ker R),F,W) \simeq (H^\ell(Y,\C_Y[n]),F,W)
\] 
as mixed Hodge structures. To prove this, it suffices to show that the image of $W_{-j}\ker R\to \omega_X(!Y)$ lies in $W_{-j}\cH^{-1}_0$ because of descending induction: if we assume that $W_{-j}\ker R \to W_{-j}\cH^{-1}_0$ is an isomorphism then 
\[
\gr^W_{-j}\ker R \to \gr^W_{-j}\cH^{-1}_0
\]
is a surjection, which automatically implies it is an isomorphism because the characteristic cycles of the source and the target are the same. Thus, $W_{-j-1}\ker R \to W_{-j-1}\cH^{-1}_0$ is an isomorphism. 

It remains to show that the image of $W_{-j}\ker R$ is in  $W_{-j}\cH^{-1}_0$. We must show that $W_{-j}\ker R \to \omega_X(!Y)$ is contained in $\ker\left(\omega_X(!Y) \to \omega_X(!Y_J) \right)$ for any subset $J\subset I$ of cardinality $j$. Using a similar idea of the proof for Lemma~\ref{lem:strictex} the complex of $\sD_X$-modules $\Omega^{n+1+\bullet}_X(\log Y)(-Y_J)\otimes \sD_X$ is a filtered resolution of a filtered $\sD_X$-module $\omega_X(!Y_J)^*$ and $\omega_X(!Y_J) \to \omega_X(!Y_J)^*$ is a strict inclusion. Therefore, 
\[
\ker(\omega_X(!Y)\to \omega_X(!Y_J)^*)=\ker(\omega_X(!Y)\to \omega_X(!Y_J)).
\] 
The short exact sequence
\[
\begin{aligned}
0\to \Omega^{n+1+\bullet}_X(\log Y)(-Y)\otimes \sD_X &\to \Omega^{n+1+\bullet}_X(\log Y)(-Y_J)\otimes \sD_X  \to T^\bullet \to 0
\end{aligned}
\] 
where $T^\bullet\colon=\Omega^{n+1+\bullet}_X(\log Y)(-Y_J)|_{Y_{I\setminus J}}\otimes \sD_X$, implies that $\ker(\omega_X(!Y)\to \omega_X(!Y_J)^*)\simeq \sH^0\left(T^\bullet\right)$. Therefore, it suffices to prove that $R^j\ker R^{j+1}$ is contained in the image of 
\begin{equation}\label{eq:last}
\sH^0\left(T^\bullet\right)  \to \cH^{-1}_0 \to \cM.
\end{equation}
We can verify this using local coordinates. For any local section $m\in R^j\ker R^{j+1}$, it can be locally represented by 
\[
m=\zeta_0 \otimes z_{I\setminus \{a\}} \partial_{I\setminus \{a\}\setminus J}\cdot P
\] 
for some $a\notin J$. Then we consider the form
\[
\frac{dz_0}{z_0}\wedge \cdots \wedge \hat{\frac{dz_a}{z_a}} \wedge \cdots \wedge \frac{dz_k}{z_k} \wedge \cdots {dz_n} \otimes z_{I\setminus \{a\}} \partial_{I\setminus \{a\}\setminus J}\cdot P
\]
where the hat denotes the omission of that term. This form defines a class in $\sH^0\left(T^\bullet\right)$ and its image under~\eqref{eq:last} is precisely $m$, which completes the proof.
\end{proof}

%%%%%%%%%%%%%%%%%%%%%%%%%%%%%%%
%bibliography
%%%%%%%%%%%%%%%%%%%%%%%%%%%%%%%

\bibliography{bib}

\begin{thebibliography}{KKMSD73}

\bibitem[Be{\u{\i}}87]{Be87}
A.~A. Be{\u{\i}}linson.
\newblock How to glue perverse sheaves.
\newblock In {\em {$K$}-theory, arithmetic and geometry ({M}oscow,
  1984--1986)}, volume 1289 of {\em Lecture Notes in Math.}, pages 42--51.
  Springer, Berlin, 1987.

\bibitem[Cle77]{Cle1977}
C.~H. Clemens.
\newblock Degeneration of {K}\"{a}hler manifolds.
\newblock {\em Duke Math. J.}, 44(2):215--290, 1977.

\bibitem[Del68]{Deligne68}
Pierre Deligne.
\newblock Th\'eor\`eme de lefschetz et crit\`eres de d\'eg\'en\'erescence de
  suites spectrales.
\newblock {\em Publications Math\'ematiques de l'IH\'ES}, 35:107--126, 1968.

\bibitem[Del70]{DP70}
Pierre Deligne.
\newblock {\em \'{E}quations diff\'{e}rentielles \`a points singuliers
  r\'{e}guliers}.
\newblock Lecture Notes in Mathematics, Vol. 163. Springer-Verlag, Berlin-New
  York, 1970.

\bibitem[Del71]{Hodge2}
Pierre Deligne.
\newblock Th\'{e}orie de {H}odge. {II}.
\newblock {\em Inst. Hautes \'{E}tudes Sci. Publ. Math.}, (40):5--57, 1971.

\bibitem[EV86]{Viehweg}
H\'{e}l\`ene Esnault and Eckart Viehweg.
\newblock Logarithmic de {R}ham complexes and vanishing theorems.
\newblock {\em Invent. Math.}, 86(1):161--194, 1986.

\bibitem[EV92]{EV}
H\'{e}l\`ene Esnault and Eckart Viehweg.
\newblock {\em Lectures on vanishing theorems}, volume~20 of {\em DMV Seminar}.
\newblock Birkh\"{a}user Verlag, Basel, 1992.

\bibitem[Fuj99]{Taro99}
Taro Fujisawa.
\newblock Limits of {H}odge structures in several variables.
\newblock {\em Compositio Math.}, 115(2):129--183, 1999.

\bibitem[Fuj08]{Taro08}
Taro Fujisawa.
\newblock Mixed {H}odge structures on log smooth degenerations.
\newblock {\em Tohoku Math. J. (2)}, 60(1):71--100, 2008.

\bibitem[Fuj14]{Taro14}
Taro Fujisawa.
\newblock Polarizations on limiting mixed {H}odge structures.
\newblock {\em J. Singul.}, 8:146--193, 2014.

\bibitem[GH14]{GH}
P.~Griffiths and J.~Harris.
\newblock {\em Principles of Algebraic Geometry}.
\newblock Wiley Classics Library. Wiley, 2014.

\bibitem[Gin86]{Gins}
V.~Ginsburg.
\newblock Characteristic varieties and vanishing cycles.
\newblock {\em Invent. Math.}, 84(2):327--402, 1986.

\bibitem[GNA90]{hl}
F.~Guill\'{e}n and V.~Navarro~Aznar.
\newblock Sur le th\'{e}or\`eme local des cycles invariants.
\newblock {\em Duke Math. J.}, 61(1):133--155, 1990.

\bibitem[GS75]{GS75}
Phillip Griffiths and Wilfried Schmid.
\newblock Recent developments in {H}odge theory: a discussion of techniques and
  results.
\newblock In {\em Discrete subgroups of {L}ie groups and applicatons to moduli
  ({I}nternat. {C}olloq., {B}ombay, 1973)}, pages 31--127. 1975.

\bibitem[HTT08]{HTT}
Ryoshi Hotta, Kiyoshi Takeuchi, and Toshiyuki Tanisaki.
\newblock {\em {$D$}-modules, perverse sheaves, and representation theory},
  volume 236 of {\em Progress in Mathematics}.
\newblock Birkh\"{a}user Boston, Inc., Boston, MA, 2008.
\newblock Translated from the 1995 Japanese edition by Takeuchi.

\bibitem[Kas84]{KM84}
Masaki Kashiwara.
\newblock The {R}iemann-{H}ilbert problem for holonomic systems.
\newblock {\em Publ. Res. Inst. Math. Sci.}, 20(2):319--365, 1984.

\bibitem[KKMSD73]{KKMS73}
G.~Kempf, Finn~Faye Knudsen, D.~Mumford, and B.~Saint-Donat.
\newblock {\em Toroidal embeddings. {I}}.
\newblock Lecture Notes in Mathematics, Vol. 339. Springer-Verlag, Berlin-New
  York, 1973.

\bibitem[Meb84]{MZ84}
Zoghman Mebkhout.
\newblock Une autre \'{e}quivalence de cat\'{e}gories.
\newblock {\em Compositio Math.}, 51(1):63--88, 1984.

\bibitem[MS]{MS}
Philippe Maisonobe and Claude Sabbah.
\newblock Aspects of the theory of {D}-modules (kaiserslautern 2002), version
  july 2011.

\bibitem[Nak21]{Nak21}
Yukiyoshi Nakkajima.
\newblock An ideal proof for {F}ujisawa's result and its generalization, 2021.

\bibitem[PS08]{PS08}
Chris A.~M. Peters and Joseph H.~M. Steenbrink.
\newblock {\em Mixed {H}odge structures}, volume~52 of {\em Ergebnisse der
  Mathematik und ihrer Grenzgebiete. 3. Folge. A Series of Modern Surveys in
  Mathematics [Results in Mathematics and Related Areas. 3rd Series. A Series
  of Modern Surveys in Mathematics]}.
\newblock Springer-Verlag, Berlin, 2008.

\bibitem[Sab02]{SC02}
C.~Sabbah.
\newblock Vanishing cycles and {H}ermitian duality.
\newblock {\em Tr. Mat. Inst. Steklova}, 238(Monodromiya v Zadachakh Algebr.
  Geom. i Differ. Uravn.):204--223, 2002.

\bibitem[Sai88]{Sai88}
Morihiko Saito.
\newblock Modules de {H}odge {P}olarisables.
\newblock {\em Publications of the Research Institute for Mathematical
  Sciences}, 24(6):849--995, 1988.

\bibitem[Sai90]{Sai90}
Morihiko Saito.
\newblock Mixed {H}odge {M}odules.
\newblock {\em Publ. Res. Inst. Math. Sci.}, 26(2):221--233, May 1990.

\bibitem[Sch73]{Schmid}
Wilfried Schmid.
\newblock Variation of {H}odge structure: the singularities of the period
  mapping.
\newblock {\em Invent. Math.}, 22:211--319, 1973.

\bibitem[Ste95]{SteLog}
J.~H.~M. Steenbrink.
\newblock Logarithmic embeddings of varieties with normal crossings and mixed
  {H}odge structures.
\newblock {\em Math. Ann.}, 301(1):105--118, 1995.

\bibitem[Ste76]{Ste76}
J.~H.~M. Steenbrink.
\newblock Limits of {H}odge structures.
\newblock {\em Invent. Math.}, 31(3):229--257, 1975/76.

\end{thebibliography}
\bibliographystyle{alpha}
%\bibliographystyle{alpha}
%\nocite{*}
%\bibliography{url}

\Addresses

\end{document}